\documentclass[12pt]{amsart}

\usepackage[dvipsnames]{xcolor}
\usepackage{graphicx}
\usepackage{latexsym}
\usepackage{amsfonts}
\usepackage{amscd}
\usepackage{amssymb}
\usepackage{amsmath}
\usepackage{amsthm}
\usepackage{bm}
\usepackage[all]{xy}
\usepackage[lmargin=1.25in,rmargin=1.25in,tmargin=1.25in,bmargin=1in,dvips]{geometry}
\usepackage{pinlabel}
\usepackage{paralist}
\usepackage{mathtools}
\usepackage{enumitem}
\usepackage[normalem]{ulem}
\usepackage[unicode, linktocpage, bookmarksnumbered=true, backref=true, bookmarksdepth=3]{hyperref}
\usepackage{tikz}

\usepackage{thmtools}
\usepackage{thm-restate}

\usepackage{subfig}
\usepackage[textsize=scriptsize]{todonotes}

\newtheorem*{theorem*}{Theorem}
\newtheorem{theorem}{Theorem}[section]
\newtheorem{lemma}[theorem]{Lemma}

\newtheorem{proposition}[theorem]{Proposition}

\theoremstyle{definition}
\newtheorem{definition}[theorem]{Definition}

\newtheorem{remark}[theorem]{Remark}
\newtheorem{question}[theorem]{Question}

\newtheorem{example}[theorem]{Example}

\def\N{\mathbb{N}}
\def\R{\mathbb{R}}

\def\Z{\mathbb{Z}}
\def\Q{\mathbb{Q}}
\def\F{\mathbb{F}}

\def\CZhat{\widehat{\mathcal{C}}_\Z}
\def\CZ{{\mathcal{C}}_\Z}
\def\CF {\operatorname{CF}}

\def\CFKa {\widehat{\operatorname{CFK}}}
\def\HFKa {\widehat{\operatorname{HFK}}}

\def\CFp {\operatorname{CF}^+}
\def\CFm {\operatorname{CF}^-}
\def\CFi {\operatorname{CF}^\infty}
\def\CFa {\operatorname{\widehat{CF}}}

\def\HFa {\operatorname{\widehat{HF}}}

\def\CFpt {\ul{\operatorname{CF}}^+}

\def\CFit {\ul{\operatorname{CF}}^\infty}

\def\CFic {\mathbf{CF}^\infty}

\def\CFmc {\mathbf{CF}^-}

\def\CFK {\operatorname{CFK}}

\def\CFKm {\operatorname{CFK}^-}

\def\CFKi {\operatorname{CFK}^\infty}

\def\GR {\Gamma_m}

\def\spincs {\mathfrak{s}}
\def\spinct {\mathfrak{t}}
\def\spincu {\mathfrak{u}}
\def\spincv {\mathfrak{v}}
\def\spincx {\mathfrak{x}}
\def\spincy {\mathfrak{y}}

\def\DD {\mathcal{D}}

\def\II {\mathcal{I}}
\def\MM {\mathcal{M}}

\def\T{\mathbb{T}}
\def\X{\mathbb{X}}

\def\a {\mathbf{a}}
\def\b {\mathbf{b}}
\def\q {\mathbf{q}}
\def\r {\mathbf{r}}
\def\x {\mathbf{x}}
\def\y {\mathbf{y}}

\def\ul {\underline}

\newcommand{\abs}[1] {\left\lvert #1 \right\rvert}

\newcommand{\gen}[1] {\langle #1 \rangle}
\newcommand{\floor}[1] {\left\lfloor #1 \right\rfloor}
\newcommand{\ceil}[1] {\left\lceil #1 \right\rceil}

\def\minus{\smallsetminus}
\def\co{\colon\thinspace}

  \DeclareMathOperator{\sign}{sign}
 \DeclareMathOperator{\Spin}{Spin}
 \DeclareMathOperator{\nbd}{nbd}
\DeclareMathOperator{\PD}{PD} \DeclareMathOperator{\gr}{gr}
 \DeclareMathOperator{\Cone}{Cone}

\def\absgr{\operatorname{\widetilde{gr}}}

\def\conn{\mathbin{\#}}
\def\bconn{\mathbin{\natural}}

\renewcommand{\MR}[1]{}

\definecolor{darkgreen}{rgb}{0,0.5,0}
\definecolor{purple}{rgb}{0.5,0,0.5}

\def\Al {\tilde A}
\def\AlNorm {A}

\def\JJ {\mathcal{J}}

\def\MultComb {\mathcal{N}}
\def\Gens {\mathfrak{S}}

\numberwithin{equation}{section}

\begin{document}

\title{A Filtered Mapping cone formula for cables of the knot meridian}

\author{Hugo Zhou}
\thanks{The author was partially supported by NSF grant DMS-1552285.}
\address{School of Mathematics, Georgia Institute of Technology, Atlanta, GA 30332}
\email{hzhou@gatech.edu}

\allowdisplaybreaks

\begin{abstract}
We construct a filtered mapping cone formula that computes the knot Floer complex of the $(n,1)$--cable of the knot meridian in any rational surgery, generalizing Truong's result about the $(n,1)$--cable of the knot meridian in large surgery (\cite{Truong}) and Hedden-Levine's filtered mapping cone formula (\cite{HL}). As an application, we show that there exist knots in integer homology spheres with arbitrary $\varphi_{i,j}$ values for any $i>j\geq 0$, where $\varphi_{i,j}$ are the concordance homomorphisms defined in \cite{Homoconcor}.
\end{abstract}

\maketitle

\section{Introduction}\label{sec:intro}

Among all the applications of the Heegaard Floer package, the mapping cone formula proved by Ozsv\'ath and Szab\'o first for integer surgery \cite{integer} then for rational surgery \cite{rational}, is one of the most influential tools. It connects the Heegaard Floer theory with three and four manifold problems, and has seen applications in all aspects of low dimensional topology, to name a few examples, in the cosmetic surgery conjecture \cite{cosmetic}, surgery obstructions \cite{surgeryob1,surgeryob2}, the Berge conjecture \cite{berge1, berge2, bergehedden}, the cabling conjecture \cite{cabling} and  exceptional surgeries \cite{reducible,half-integer}.
 
There are a handful of generalizations of the original mapping cone formula, including the filtered mapping cone formula \cite{HL}, the link surgery formula \cite{link}, the involutive mapping cone formula \cite{invocone} and the involutive filtered mapping cone formula \cite{filteredinvolutive}.
Hedden and Levine's  filtered mapping cone formula defines a second filtration on the original mapping cone, and thus computes the knot Floer complex of the dual knot in the knot surgery. There has been quite a few success in utilizing this tool to understand topological questions, see for example \cite{homologyconcordance,hugo}. In the other direction, Truong proved the ``large surgery" theorem for the $(n,1)$--cable of the dual knot in \cite{Truong}. Her key observation was that the original diagram used by Ozsv\'ath and Szab\'o to compute the large surgery of a knot also specifies the $(n,1)$--cable of the dual knot, with the addition of a second basepoint.
Using this observation as an ingredient, we proved a filtered mapping cone formula for the $(n,1)$--cable of the knot meridian, generalizing both Hedden-Levine and Truong's results.

Moreover, same as Hedden and Levine's filtered mapping cone formula, our mapping cone agrees with Ozsv\'ath and Szab\'o's original mapping cone apart from the second filtration. Indeed, the $2$-handle cobordism maps in the exact surgery triangle are all the same in the above constructions, while the placement of an extra basepoint determines the second filtration. This perspective was adopted by Eftekhary as early as in \cite{Eftekhary1} and  \cite{Eftekhary2}. Comparing to Hedden-Levine's result, we merely shift the second basepoint, which results in a refiltering of the mapping cone. 
From this perspective, our result also can be seen as a generalization of the original mapping cone formula, where it further demonstrates how the different placement of a second basepoint affects the knot filtration on the mapping cone.  

On the other hand, our new formula is practically meaningful.  For the knots in $3$-manifolds other than $S^3$, even those in integer homology spheres, in the regard of knot Floer information there is very little known to us. 
 This is mainly due to a lack of computable examples. At present,  for the knots in homology spheres, the only effective tools for computing the knot Floer complex  are the filtered mapping cone formula and the knot lattice homology (proved invariant in \cite{invariance1,invariance2}). We hope that by studying the examples produced by this new filtered mapping cone formula, one can understand more about the properties of various different knots in homology spheres and as well as the manifolds themselves.

The construction of our mapping cone follows closely Hedden-Levine's framework in \cite{HL}, but it comes with its own challenges. Basically, there are a few choices to make when constructing a filtered mapping cone, and the choices that are suitable for generalization for our purpose do not always agree with those made in \cite{HL}. We will point out each time when we make a different choice.

\begin{figure}
\labellist
 
 \pinlabel $K_{n,\lambda}$ at 72 130

{ \Large
 \pinlabel $K$ at 117 30
}

\endlabellist
\includegraphics{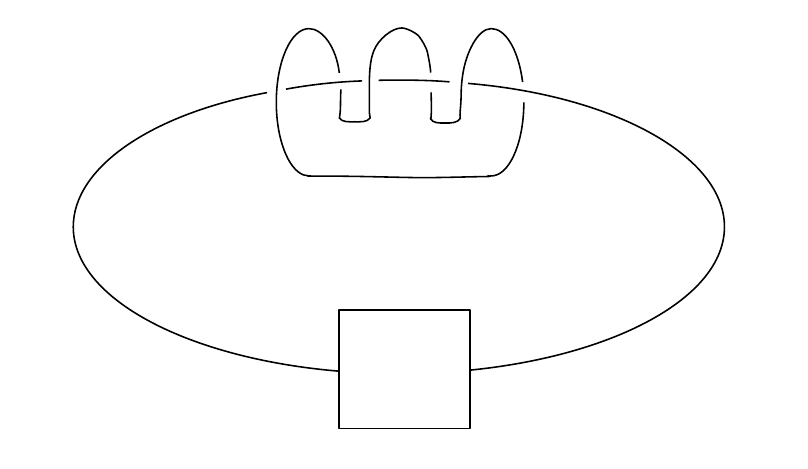}
\caption{The $(n,1)$--cable of a meridian of a knot $K$, where $n=3,$ inside the $\lambda$--framed surgery on $K$.}
\label{fig: knot}
\end{figure}

\subsection{Applications}
It turns out that the examples produced by the new filtered mapping cone formula are rich and abundant. We have
\begin{theorem}\label{thm: phi}
 For any integers $i$ and $j$ such that $i>j\geq 0,$ and any given $h\in \Z$,
 there exists $(Y,K)\in \CZhat$ such that $\varphi_{i,j}(K)=h$. 
 \end{theorem}
 To explain the notations used here, the (smooth) homology concordance group $\CZhat$ is generated by pairs $(Y,K),$ where $Y$ is an integer homology sphere bounding a homology ball and $K$ a knot in $Y.$ Two classes  $(Y_1,K_1)$ and $(Y_2,K_2)$ are equivalent in $\CZhat$ if there is a homology cobordism from $Y_1$ to $Y_2$ in which $K_1$ and $K_2$ cobound a smoothly embedded annulus. The concordance invariants $\varphi_{i,j}$ defined in \cite{Homoconcor}
 are homomorphisms $\CZhat \rightarrow \Z,$ where $(i,j)\in (\Z \times \Z^{\geq 0}) - (\Z^{<0} \times \{0\})$. They generalize the concordance homomorphisms $\varphi_i$ defined in \cite{Moreconcor} and are used to prove the existence of a $\Z^\infty$ summand in $\CZhat/\CZ$, where $\CZ$ is the subgroup of $\CZhat$ generated by the knots in $S^3.$ The reader is referred to the original sources to learn more about those concordance homomorphisms. We also offer a brief review in Section \ref{ssec: ring}.
 
 Note that Theorem \ref{thm: phi} is in contrast to the examples that come from the original filtered mapping cone formula. It was shown according to \cite[Corollary 1.3]{filteredinvolutive} that for any knot $K\subset S^3,$ the knot meridian inside the $1/p$-surgery on $K$ for any integer $p$  will have $\varphi_{i,j}=0$ for all $\abs{i-j}>2.$ (\cite[Corollary 1.3]{filteredinvolutive} only proved the case for $p=\pm 1$, but the case for rational surgeries  follows similarly.)

 Theorem \ref{thm: phi} is the immediate consequence of the following computational result. Performing  $+1$-surgery on the torus knot $T_{2,4k+3}$,  let $J_{n,k}$ denote  the $(n,1)$--cable of the knot meridian in $S^3_1(T_{2,4k+3})$ connected sum with the unknot\footnote{ The invariants $\varphi_{i,j}$ were defined for knots in any integer homology spheres, so we could also talk about $\varphi_{i,j}$ of the $(n,1)$--cable of the knot meridians. Since connected summing with the unknot does not change $\varphi_{i,j}$, those $\varphi_{i,j}$ would have the same values as in Proposition \ref{prop: phi}.} in $-S^3_1(T_{2,4k+3})$. The ambient manifold  $S^3_1(T_{2,4k+3}) \conn -S^3_1(T_{2,4k+3})$ is homology cobordant to $S^3$. 
\begin{restatable}[]{proposition}{propphi}
\label{prop: phi}
  For any $k\geq 0$ and $n\geq 1$, we have
  \begin{align*}
      \varphi_{i,k}(J_{n,k})=  
      \begin{cases}
        -1  & i=k+n\\
        0   & i>k+n .
      \end{cases}
  \end{align*}
\end{restatable}
   Proposition \ref{prop: phi} is proved in Section \ref{ssec: examples}, and we also refer the reader to Figure \ref{fig: complex_example} for one example of this infinite family. 

 Next, let $\widehat{\mathcal{C}}_{\Z,\text{top}}$ be the topological version of the homology concordance group, where two classes  $(Y_1,K_1)$ and $(Y_2,K_2)$ are equivalent if $K_1$ and $K_2$ cobound a locally flat annulus in a \emph{topological} homology cobordism from $Y_1$ to $Y_2$ (which does not need to have a smooth structure).  Let    $\psi \co \CZhat \rightarrow \widehat{\mathcal{C}}_{\Z,\text{top}}$ be the natural map that forgets the smooth structure. 
\begin{theorem}\label{thm: phiker}
The classes $(Y,K)$ in Theorem \ref{thm: phi} can be taken inside $\operatorname{ker}\psi.$
 \end{theorem}
 \begin{proof}
 The proof proceeds the same way as the proof of \cite[Theorem 1.7]{hugo}. For the convenience of the reader, we repeat the complete proof here. According to \cite[Proposition 6.1]{twoTorsion}, up to quotienting acyclic complexes, the positive-clasped untwisted Whitehead double of $T_{2,3}$, denoted by $D$, has the  same knot Floer complex as $T_{2,3}$. Thus instead of applying the filtered mapping cone formula to $T_{2,4k+3}$, applying it to the knot $D_{k} \coloneqq (2k+1)D$ yields a complex with identical concordance invariants.  
 
 On the other hand, by the work of Freedman and Quinn \cite{FQbook}, the knot $D_k$ has trivial Alexander polynomial, thus is topologically slice. Consider the $4$-manifold $W_k$ obtained by attaching a $+1$-framed two-handle to $B^4$ along $D_k$. Inside $W_k$, the core of the two-handle and the (topologically) slice disk of $D_k$ form a sphere; let $Z$ denote a tubular neibourhood of this sphere. Then $Z$ is a disk bundle with Euler number one, therefore has $S^3$ as its boundary. Deleting $Z$ and gluing back in a $B^4$, the resulting manifold $W'_k$ has the same homology as a point  according to Mayer-Vietoris, thus is contractible due to Whitehead Theorem. At the same time, in the $+1$-surgery of $D_K$, the dual knot is isotopic to a $0$-framed longitude, which bounds a locally flat disk in $B_4$ disjoint from the slice disk of $D_k$. Therefore $S^3_1(D_k)$ bounds a contractible manifold $W'_k$, in which the dual knot of $D_k$ bounds a locally flat disk. Finally notice that the $(n,1)$--cable of a slice knot is slice.
 \end{proof}
 
 \begin{remark}
 It is worth noting the contractible manifold that $S^3_1(D_k)$ bounds is a topological manifold. In fact, the $d$--invariant obstructs $S^3_1(D_k)$ from bounding  a contractible smooth manifold.
 \end{remark}
 
 So far, the examples that generate the $\Z^\infty$ summand in $\CZhat/\CZ$ are necessarily in distinct integer homology spheres. Answering a question raised in \cite{Homoconcor}, we have the following.

 \begin{theorem}
 The $\Z^\infty$ summand in $\CZhat/\CZ$ can be generated by knots inside one single integer homology sphere $Y$. 
 \end{theorem}
\begin{proof}
By fixing  $k$  and varying $n,$ we obtain an infinite family of knots $\{J_{n,k}\}_{n\in \N}$, which immediately implies that the homomorphism $\bigoplus_{i>k} \varphi_{i,k}$ is surjective onto $\Z^\infty$ by Proposition \ref{prop: phi}.  
\end{proof}

As a slight variant of the group $\CZhat$, \cite[Remark 1.13]{homologyconcordance} considered the subgroup of  $\CZhat$ consisting of all pairs $(Y,K)$ such that $Y$ bounds a homology $4$-ball in which $K$ is freely nullhomotopic (or equivalently, in which $K$ bounds an immersed disk). As they remarked,  this variant is arguably a more appropriate generalization of the concordance group than $\CZhat$,  as it measures the failure
of replacing immersed disks by embedded ones. Clearly, if $Y$ bounds a smooth $4$-manifold that is contractible, the above requirements are automatically satisfied for any knot $K\subset Y$. This prompts the question:

\begin{question}
Can the $\Z^\infty$ summand in $\CZhat/\CZ$ be generated by knots inside an integer homology sphere $Y$ that bounds a contractible smooth manifold? 
\end{question}

Another interesting aspect worth pointing out is that the new filtered mapping cone  captures all the  information from the input if $n$ is sufficiently large. As a comparison, recall that according to \cite[Theorem 3.1]{hugo},  up to local equivalence, for the $+1$-surgery on $L$--space knots, the knot Floer complex of the dual knot only depends on two parameters: the total length of horizontal edges in the top half of the complex, and the congruence class of the number of generators mod $4$. (This is more or less natural in view of \cite[Theorem 1.2]{filteredinvolutive}, which states that the dual knot complex admits a relatively small local model.) In contrast, we have the following.

\begin{proposition}\label{prop: intro_middle}
 Suppose knot $K\subset S^3$ is a knot of genus $g$. When $n\geq 2g$, there is a quotient complex of the filtered mapping cone of the $(n,1)$--cable of the knot meridian that is filtered homotopy equivalent to $\CFKi(S^3,K).$
\end{proposition}

 See Proposition \ref{prop: app_middle} for the precise statement. In general,  up to local equivalence, the knot Floer complex of the $(n,1)$--cable of the dual knot depends on much more than the two parameters mentioned above. 

Moreover, there is a curious phenomenon which we shall name \emph{stabilization} that is associated to the behavior of the knot  Floer complex of the $(n,1)$--cable of the dual knot when $n\geq 2g$. Basically, when $n\geq 2g$ and as $n$ increases, the length of some edges of the complex will increase, but aside from that, the complex stops changing ``shape'', instead, a new copy of the complex $\CFKi(S^3,K)$ gets added to the complex ``in the middle'' each time $n$ increases by $1$. For more details, read the discussion below Proposition \ref{prop: app_middle}.

While the algebraic reason for stabilization is rather straightforward, it is somewhat mysterious why this phenomenon happens geometrically. Are $(n,1)$--cables of the knot meridian for $n\geq 2g$ in fact special topologically? We ask the following question:
\begin{question}
What is the geometrical explanation for stabilization?
\end{question}
In the subsection that follows, we present the statement of the new filtered mapping cone formula.

\subsection{Statement of the main theorem.}\label{ssec: statement}
We start by introducing some notations. 

Suppose $K$ represents a class of order $d>0$ in $H_1(Y;\Z)$. Fix a tubular neighborhood $\nbd(K)$ of $K$, let $\mu$ be the canonical choice of the right-handed meridian in $\nbd(K)$ and  $\lambda$ a choice of longitude,  both of which we may view as curves on $\partial(\nbd(K))$. 
  Seen as elements of $H_1(\partial (\nbd(K)))$, $\lambda$ and $\mu$ satisfy $\partial F = d\lambda - k\mu$ for some $k\in\Z$, where $F$ is a rational Seifert surface of $K$. In fact, the framing is completely determined by $k$.  Assume from now on $k\neq 0.$ 
  
  Let $Y_\lambda(K)$ denote the surgery on $K$ with the framing specified by $\lambda$. Note that inside  $Y_\lambda(K)$, $\mu$ is isomorphic to the core of surgery solid torus. Following Hedden-Levine's convention, we will let our cables of the meridian inherit the orientation from the \emph{left-handed} meridian. More precisely, define $K_{n,\lambda} \in Y_\lambda(K)$ to be $n$ copies of $-\mu$ joined with trivial band sums (see Figure \ref{fig: knot}).  The knot $K_{n,\lambda}$ is the $(n,1)$--cable of the dual knot with its framing specified by $\partial(\nbd(K))$. In rest of the paper, the standalone letters $k,d$ and $n$ are in general reserved for the quantities described above.

 We will assume from the reader certain familiarity of Heegaard Floer package. The basic construction and certain helpful propositions are reviewed in Section \ref{sec: prelim}. For a more thorough introduction or survey on the topic, see for example
\cite{IntroHeegaard,survey}.

 Recall that the \emph{relative spin$^c$ structures} are the set of homology classes of  vector fields that are non-vanishing  on $Y \minus \nbd(K)$ and tangent to the boundary along $\partial (\nbd(K))$, denoted by $\ul\Spin^c(Y,K)$.  Ozsv\'ath and Szab\'o defined the map
\[
G_{Y,K} \co \ul\Spin^c(Y,K) \to \Spin^c(Y),
\]
which is equivariant with respect to the restriction map
\[
H^2(Y,K;\Z) \to H^2(Y;\Z).
\]
Following the convention in \cite{HL}, the \emph{Alexander grading} of each $\xi \in \ul\Spin^c(Y,K)$ is defined as
\begin{equation} \label{eq: spinc-alex}
\AlNorm_{Y,K}(\xi) = \frac{\gen{c_1(\xi), [F]} + [\mu] \cdot [F] }{2 [\mu] \cdot [F]} \in \frac{1}{2d} \Z,
\end{equation}
where $F$ is a rational Seifert surface for $K$. For each $\spincs \in \Spin^c(Y)$, the values of $\AlNorm_{Y,K}(\xi)$, for all the  $\xi \in \ul\Spin^c(Y,K)$ such that $G_{Y,K}(\xi)=\spincs$, belong to a single coset in $\Q/\Z$; denote it by $\AlNorm_{Y,K}(\spincs)$. Moreover, any $\xi \in \ul\Spin^c(Y,K)$ is uniquely determined by the pair $(G_{Y,K}(\xi), \AlNorm_{Y,K}(\xi))$.


Choose a spin$^c$ structure $\spinct$ on $Y_\lambda(K)$. 
There is a bijection between $\Spin^c(W_{\lambda})$ and $\ul\Spin^c(Y,K)$, where $W_{\lambda}$ is the two handle cobordism from $Y$ to $Y_\lambda$ (see Definition \ref{def: spincbij}  for more details). Consider all the spin$^c$ structures in $\Spin^c(W_{\lambda})$ that extend $\spinct$. We may see these as spin$^c$ structures in $\ul\Spin^c(Y,K)$ through the bijection; denote them by $(\xi_l)_{l \in \Z}$.  Let $\spincs_l = G_{Y,K}(\xi_l)$ and $s_l = \AlNorm_{Y,K}(\xi_l)$.  The index is pinned down first by the relation $\spincs_{l+1} = \spincs_l + \PD[K]$, so that the sequence $(\spincs_l)_{l \in \Z}$ repeats with period $d$, while $s_{l+1} = s_l + \frac{k}{d}$, and then by the conventions (for more discussions see the end of Section \ref{sssec: spinc})
\begin{gather}
\label{eq: xil-bound-pos}
\frac{(2l-1)k}{2d}  < \AlNorm_{Y,K}(\xi_l) \le \frac{(2l+1)k}{2d} \qquad \text{if } k>0, \\
\label{eq: xil-bound-neg}
\frac{(2l+1)k}{2d} \le  \AlNorm_{Y,K}(\xi_l) < \frac{(2l-1)k}{2d}   \qquad \text{if } k<0.
\end{gather}

For each $l \in \Z$, let $A^\infty_{\xi_l}$ and $B^\infty_{\xi_l}$ each denote a copy of $\CFKi(Y,K,\spincs_l)$. Define a pair of filtrations $\II_\spinct$ and $\JJ_\spinct$ and an absolute grading $\gr_\spinct$ on these complexes as follows:
\begin{align}
\intertext{For $[\x,i,j] \in A^\infty_{\xi_l}$,}
\label{eq: It-def-A} \II_\spinct([\x,i,j]) &= \max\{i,j-s_l\} \\
\label{eq: Jt-def-A} \JJ_\spinct([\x,i,j]) &= \max\{i-n,j-s_l\} + \frac{2nds_l+nk-n^2 d}{2k} \\
\label{eq: grt-def-A} \gr_\spinct([\x,i,j]) &= \absgr([\x,i,j]) + \frac{(2ds_l - k)^2 }{4dk} + \frac{2-3\sign(k)}{4} \\
\intertext{For $[\x,i,j] \in B^\infty_{\xi_l}$,}
\label{eq: It-def-B} \II_\spinct([\x,i,j]) &= i \\
\label{eq: Jt-def-B} \JJ_\spinct([\x,i,j]) &= i-n + \frac{2nds_l+nk-n^2 d}{2k} \\
\label{eq: grt-def-B} \gr_\spinct([\x,i,j]) &= \absgr([\x,i,j]) + \frac{(2ds_l - k)^2 }{4dk} + \frac{-2-3\sign(k)}{4}
\end{align}
Here $\absgr$ denotes the $\Q$--valued Malsov grading on $\CFKi(Y,K,\spincs_l)$.  The values of $\II_\spinct$ are integers, while the values of $\JJ_\spinct$ live in the coset $\AlNorm_{Y_\lambda,K_{n,\lambda}}(\spinct)$. See Figure \ref{fig: filtration} for an example of the filtrations when $k=d=1.$

For each $\xi_l$, there is a filtered chain homotopy equivalence 
\[
\Psi^\infty_{\xi_l} \co \CFKi(Y,K,\spincs_l) \longrightarrow \CFKi(Y,K,\spincs_{l+1}),
\]
often referred to as the ``flip map". (We will discuss this in more details in Section \ref{sec: prelim}.) In particular, for any null-homologous knot in an $L$--space, the reflection map that exchanges $i$ and $j$ suffices to play the role of   $\Psi^\infty_{\xi_l} $.

 Let $A^-_{\xi_l}$ (resp.~$B^-_{\xi_l}$) denote the subcomplex of $A^\infty_{\xi_l}$ (resp.~$B^\infty_{\xi_l}$) with $\II<0$, and let $A^+_{\xi_l}$ (resp.~$B^+_{\xi_l}$) denote the quotient.
Define $v^\infty_{\xi_l} \co A^\infty_{\xi_l} \to B^\infty_{\xi_l}$ to be the identity map on $\CFKi(Y,K,\spincs_l)$, and $h^\infty_{\xi_l} \co A^\infty_{\xi_l} \to B^\infty_{\xi_{l+1}}$ given by the ``flip map'' $\Psi^\infty_{\xi_l} $. Both $v^\infty_{\xi_l}$ and $h^\infty_{\xi_l}$ are doubly-filtered and homogeneous of degree $-1$. So $v^\infty_{\xi_l}$ (resp.~$h^\infty_{\xi_l}$ ) restricted to $A^-_{\xi_l}$ maps to $B^-_{\xi_l}$ (resp.~$B^-_{\xi_{l+1}}$), and hence also induces a map from $A^+_{\xi_l}$ to $B^+_{\xi_l}$ (resp.~$B^+_{\xi_{l+1}}$). Define these maps to be $v^-_{\xi_l}$ (resp.~$h^-_{\xi_l}$ ) and $v^+_{\xi_l}$ (resp.~$h^+_{\xi_l}$ ). The above definitions agree with those in \cite{HL}.

Let $\mathbb{A}_{\lambda,\spinct,a,b}=\bigoplus_{l=a}^{b} A^\infty_{\xi_l}   $ and   $\mathbb{B}_{\lambda,\spinct,a,b}=\bigoplus_{l=a+1}^{b} B^\infty_{\xi_l}   $, both inherit the double filtrations. Let $v^\infty = \bigoplus_{l} v^\infty_{\xi_l} $ and $h^\infty = \bigoplus_{l} h^\infty_{\xi_l} $ be maps from $\mathbb{A}_{\lambda,\spinct,a,b}$ to $\mathbb{B}_{\lambda,\spinct,a,b}$.  Define $X^\infty_{\lambda, \spinct,n, a,b}$  to be the mapping cone of $v^\infty + h^\infty$. The following is our main theorem.

\begin{theorem} \label{thm: mapping-cone}
Let $K$ be a knot in a rational homology sphere $Y$, let $\lambda$ be a nonzero framing on $K$, and let $\spinct \in \Spin^c (Y_\lambda(K))$. Then for all $a \ll 0$ and $b \gg 0$, the chain complex $X^\infty_{\lambda,\spinct,n,a,b}$, equipped with the filtrations $\II_\spinct$ and $\JJ_\spinct$, is doubly-filtered chain homotopy equivalent to $\CFKi(Y_\lambda, K_{n,\lambda}, \spinct)$.
\end{theorem}

Theorem \ref{thm: mapping-cone} can be generalized to the case of rational surgeries as well. See Section \ref{sec: rational} for details.

\begin{remark}
Although not clear from the construction, the $a$ and $b$ in the above theorem can be decided by $d,n$ and $g$ quite straightforwardly, where $g$ is the genus of the knot $K$. Suppose $K$ is null-homologous, so we have $d=1.$ In this case taking $a=-g+1$ and $b=g+n-1$ would suffice. See the beginning of Section \ref{sec: examples} for an explanation.
\end{remark}

\begin{remark}
The filtration in \cite{Truong} only agrees with ours after a reflection with respect to the  line $i=j$. This is because Truong used the right-handed meridian while we use the left-handed meridian. In fact, there are two versions of the filtered mapping cone formula that one can prove: the current one with $(w,z,z_n)$ basepoints where the additional basepoint $z_n$ is placed to the left of $w$ and $z$, and another version with $(w,z,w_n)$ basepoints where $w_n$ basepoint is placed to the right of $w$ and $z$. The latter will result in a filtration matching the one in \cite{Truong}. We chose the current version since it generalizes the filtrations defined in \cite{HL}.  
\end{remark}

We learned that Ian Zemke is working on a different method that would achieve the similar filtered mapping cone formula for the $(n,1)$--cables of the knot meridian: if one takes the connected sum of $K$ with a $T_{2,2n}$ torus link, the $(n,1)$--cable of the meridian of $K$ can be obtained by performing a surgery on one of the link components. The gradings are then readily read off by considering the grading changes in the cobordism maps. The base case with the $T_{2,2}$ torus link is worked out in \cite{filteredinvolutive}.

\subsection*{Organization} In Section \ref{sec: prelim}, we review the basic construction of the knot Floer homology. In Section \ref{sec: alex}, we construct the diagrams for the $(n,1)$--cable of the knot meridian and compute the Alexander grading shifts and first Chern class evaluations of holomorphic polygons. Those would form the theoretical core for the upcoming sections.  Then in Section \ref{sec: largesurgery}, \ref{sec: exacttri} and \ref{sec: proofcone},  we prove the large surgery formula for the $(n,1)$--cable of the knot meridian, the  filtered surgery exact triangle and finally the filtered  mapping cone formula. In Section \ref{sec: rational} we generalize the filtered mapping cone formula to the case of rational surgeries. In Section \ref{sec: examples} we perform the computations that lead to the proof of Proposition \ref{prop: phi} and Theorem \ref{thm: phi}.

\subsection*{Acknowledgement}
I want to thank my advisor Jen Hom for her continued support and encouragement. I am grateful to Adam Levine, Ian Zemke, Matt Hedden, John Etnyre, JungHwan Park and Chuck Livingston for helpful comments. I would like to thank Matt Hedden and Adam Levine for providing such a clear and thorough guideline in \cite{HL}, which inspired the current paper. I also want to thank the group of young mathematicians  I met during GSTGC at Georgia Tech, to whom I attribute much of the motivation for accomplishing this project.

\section{Preliminary on knot Floer complexes} \label{sec: prelim}
 Ozsv\'ath and Szab\'o defined a package of Floer invariants, including the Heegaard Fleor complex for three manifolds (see \cite{OSthreemanifold}) and the knot Floer complex for the knots (see \cite{OSknot}). In this section we collect some basic definitions and propositions necessary for the rest of the paper. For a more thorough introduction or survey on the topic, see
\cite{IntroHeegaard,survey}.
 
 A \emph{pointed Heegaard diagram} for a three manifold $Y$ is a quadruple $(\Sigma,\bm\alpha,\bm\beta,w),$ where $\Sigma$ is an oriented surface of genus $g$, $\bm\alpha =\{ \alpha_1,\cdots, \alpha_g\}$ is a set of disjoint simple close curve on $\Sigma$ that indicates the attaching circles for one-handles of $Y$, and similarly $\bm\beta =\{\beta_1,\cdots, \beta_g\}$ indicates the  attaching circles for $2$-handles of $Y$.  Then 
 \[
 z\in \Sigma -\alpha_1 -\cdots -\alpha_g -\beta_1 -\cdots -\beta_g
 \]
 is a choice of reference point. Together $(\Sigma,\bm\alpha,\bm\beta,w)$ enables the construction of a suitable variant of Lagrangian Floer homology in the $g$--fold symmetric product of $\Sigma$. Define
 \[
 \T_\alpha=\alpha_1 \times\cdots \times\alpha_g \subset \text{Sym}^g(\Sigma) \quad \text{ and  }  \quad T_\beta=\beta_1 \times\cdots \times\beta_g \subset \text{Sym}^g(\Sigma). 
 \]
 The complex $\CFi(\Sigma,\bm\alpha,\bm\beta,w)$ is freely generated over $\F$ by generators $[\x,i]$, where $\x$ is an intersection point of  $\T_\alpha$ and $\T_\beta$ and $i$ is any integer. The differential is given by
 \[
 \partial([\x,i]) = \sum_{\y \in \T_\alpha \cap \T_\beta} \sum_{\substack{ \phi \in \pi_2(\x,\y) \\ \mu(\phi)=1}} \# \widehat\MM(\phi) \, [\y, i-n_z(\phi)],
 \]
where $\pi_2(\x,\y)$ denotes the space of homotopy class of holomorphic disks from $\x$ to $\y$, $\mu(\phi)$ denotes its Maslov index, $\widehat\MM(\phi)$ denotes the moduli space of pseudo-holomorphic representatives of $\phi$ quotient by $\R$, and $n_w(\phi)$ counts the algebraic intersection number of $\phi$ with $\{z\}\times \text{Sym}^{g-1}(\Sigma)$. There is an $U$--action reflected on the integer component, by the relation $U\cdot[\x,i]=[\x,i-1]$. The complex $\CFi(\Sigma,\bm\alpha,\bm\beta,z)$ has underlying structure an $\F[U]$ module. Define the subcomplex generated by $[\x,i]$ with $i<0$ to be $\CFm(\Sigma,\bm\alpha,\bm\beta,z)$,  the quotient of it $\CFp(\Sigma,\bm\alpha,\bm\beta,z)$ and the sub-quotient complex generated by $[\x,i]$ with $i=0$ as $\CFa(\Sigma,\bm\alpha,\bm\beta,z)$. For $t\in \N$, the "truncated complex" $\CF^t(\Sigma,\bm\alpha,\bm\beta,z)$ is defined to be the sub-quotient complex generated by $[\x,i]$ with $0\leq i \leq t$; or equivalently, the kernel of action $U^{t+1}$ on $\CFp(\Sigma,\bm\alpha,\bm\beta,z)$. Note that  $\CF^t(\Sigma,\bm\alpha,\bm\beta,z)$ is isomorphic to the quotient 
\[
\CFm(\Sigma, \bm\alpha, \bm\beta, z, \spincs) / (U^{t+1} \cdot \CFm(\Sigma, \bm\alpha, \bm\beta, z, \spincs)),
\]
up to a grading shift; or equivalently $[\x,i]$ with $-t\leq i \leq 0$. 
It is proved in \cite{OSclosed} that the chain homotopy type of the chain complex $\CF^\circ(\Sigma,\bm\alpha,\bm\beta,z)$ is a topological invariant of $Y$, so it makes sense to drop the choice of the Heegaard diagram from the notation and write $\CF^\circ(Y)$.

As a refinement of above construction, the knot Floer complex is a topological invariant for the isotopy class of the knot $K\subset Y$. For simplicity, specialize $Y$ to be a rational homology sphere from now on. First, a \emph{doubly pointed Heegaard diagram} $(\Sigma,\bm\alpha,\bm\beta,w,z)$ is a pointed Heegaard diagram with the addition of a basepoint $z$, which  specifies a knot $K\subset Y$ in the following way: letting $t_\alpha$ be an arc from $z$ to $w$ in $\Sigma \setminus \bm\alpha$, and   $t_\beta$ an arc from $w$ to $z$ in $\Sigma \setminus \bm\beta$, then $K$ is obtained from pushing $t_\alpha$ slightly into the $\alpha$--handlebody and $t_\beta$ slightly into the $\beta$--handlebody.

Ozsv\'ath and Szab\'o defined a map $\spincs_z \co \T_\alpha \cap\T_\beta \to \Spin^c(Y)$, which associate each $\x\in \T_\alpha \cap\T_\beta$ a spin$^c$ structure $\spincs_z(\x) \in \Spin^c(Y)$. The map $\spincs_w$ can also be defined in the same way, with the relation between them given by $\spincs_z(\x) = \spincs_w(\x) + PD[K]$. The Heegaard Floer complex decomposes over $\Spin^c(Y)$ as 
\[
\CF^\circ(Y)=\bigoplus_{\spincs \in \Spin^c(Y)} \CF^\circ(Y,\spincs),
\]
where $\CFp(Y,\spincs)$ is generated by $[\x,i]$ satisfying $\spincs_z(\x)=\spincs.$ Similarly,  Ozsv\'ath and Szab\'o also defined a map $\ul\spincs_{w,z} \co \T_\alpha \cap\T_\beta \to \ul\Spin^c(Y,K)$ with the property that $G_{Y,K} (\ul\spincs_{w,z}(\x)) = \spincs_w(\x)$, where $G_{Y,K} \co \ul\Spin^c(Y,K) \to \Spin^c(Y)$ is the restriction map defined in Section \ref{ssec: statement}. The Alexander grading of $\x$ is defined as
\begin{align}
\label{eq: alex-def}
\AlNorm_{w,z}(\x) &= \AlNorm_{Y,K}(\ul\spincs_{w,z}(\x)),
\end{align}
where the Alexander grading of a relative spin$^c$ structure is given by (\ref{eq: spinc-alex}). Suppose generators $\x$ and $\y$ satisfy $\spincs_z(\x) = \spincs_z(\y)$, for any disk $\phi \in \pi_2(\x,\y)$, then we have
\begin{equation} \label{eq: rel-alex}
\AlNorm_{w,z}(\x) - \AlNorm_{w,z}(\y) = n_z(\phi) - n_w(\phi).
\end{equation}
In particular, the difference of the Alexander grading is an integer. Thus the Alexander grading $\AlNorm_{w,z}(\x)$ of the generators in the same spin$^c$ class $\spincs_z(\x)=\spincs$ belongs to the same coset of $\Q/\Z$; denote it $\AlNorm_{w,z}(\spincs)$.    For each $\spincs \in \Spin^c(Y)$, the knot Floer complex $\CFKi(\Sigma, \bm\alpha, \bm\beta, w, z, \spincs)$ is freely generated over $\F$ by all $[\x, i, j]$, where $\spincs_z(\x) = \spincs$, $i \in \Z$, and $j-i = \AlNorm_{w,z}(\x)$. The $U$--action is given by 
$U \cdot [\x,i,j] = [\x,i-1,j-1]$, and the differential is given by
\begin{equation} \label{eq: CFKi-diff}
\partial [\x,i,j] = \sum_{\y \in \T_\alpha \cap \T_\beta} \sum_{\substack{\phi \in \pi_2(\x,\y) \\ \mu(\phi)=1}} \# \widehat{\MM}(\phi) [\y, i-n_w(\phi), j-n_z(\phi)].
\end{equation}
Note that the $j$ coordinate of any generator is in the coset specified by $\AlNorm_{Y,K}(\spincs)$. 

There is a forgetful map $\Omega_z \co \CFKi(\Sigma, \bm\alpha, \bm\beta, w, z, \spincs) \to \CFKi(\Sigma, \bm\alpha, \bm\beta, z,  \spincs)$ given by $\Omega_z([\x,i,j]) = [\x,i]$. Under this forgetful map, one can also view the $j$ coordinate as a second filtration on $\CFKi(\Sigma, \bm\alpha, \bm\beta, z,  \spincs)$, which we call the \emph{Alexander filtration}. Define $\AlNorm_{w,z}([\x,i]) = \AlNorm_{w,z}(\x) + i.$  By equipping different flavors of $\CF^\circ$ with this filtration, the corresponding  doubly-filtered chain complexes are denoted by $\CFK^-,\CFK^+, \CFKa$ and $\CFK^t$. Similar to the Heegaard Floer complex, the filtered chain homotopy type of $\CFK^\circ$ is a topological invariant of the isotopy class of the  knot $K\subset Y$, and at times we drop the Heegaard diagram infomation and simply write $\CFK^\circ(Y,K)$.

The difference between two distinct $\Q$--valued Maslov gradings is given by (see for example \cite{zemkegrading})
\[
\absgr_w(\x) - \absgr_z(\x) = 2\AlNorm_{w,z}(\x). 
\]
We stick to the Maslov grading specified by the basepoint of Heegaard Floer complex throughout the paper, and drop the index when the context is clear.

For each $s\in \AlNorm_{w,z}(\spincs),$   there is a map
\[
\Psi^\infty_{\spincs,s}\co \CFKi(\Sigma, \bm\alpha, \bm\beta, w, z, \spincs) \longrightarrow \CFKi(\Sigma, \bm\alpha, \bm\beta, w, z, \spincs+ \PD[K])
\]
that by the work of Juh\'asz–Thurston–Zemke in \cite{mappingclassheegaard} and the work of Zemke in \cite{graph}, is independent of any choice up to homotopy equivalence for any knot in a rational homology sphere. By \cite[Lemma 2.16]{HL}, the map $\Psi^\infty_{\spincs, s}$ is a filtered homotopy equivalence with respect to the $j$ filtration on the domain and the $i$ filtration on the range, in the sense that, for any $t \in \AlNorm_{Y,K}(\spincs)$, $\Psi^\infty_{\spincs, s}$  restricts to a homotopy equivalence from the $j \le t$ subcomplex of $\CFKi(\Sigma, \bm\alpha, \bm\beta, w, z, \spincs)$ to the $i \le t-s$ subcomplex of $\CFKi(\Sigma, \bm\alpha, \bm\beta, w, z, \spincs + \PD[K])$.

It is enough to know the map  $\Psi^\infty_{\spincs, s}$ for a singular value of $s$, since for other values of $s$, the maps $\Psi^\infty_{\spincs, s}$  are related by:
\[
\Psi^\infty_{\spincs, s+1} = U \circ \Psi^\infty_{\spincs, s} = \Psi^\infty_{\spincs, s} \circ U.
\]

The pair $(\spincs,s)$ determines and is determined by a relative spin$^c$ structure $\xi \in \ul\Spin^c(Y,K),$ thus we can  denote  the map by $\Psi^\infty_{\xi}.$ This is the ``flip map" used in the mapping cone.

In general, the map $\Psi^\infty_{\xi}$ is difficult to determine from the definition. However, for any null-homologous knot in an $L$--space, by \cite[Lemma 2.18]{HL}, $\Psi^\infty_{\xi}$ is filtered homotopy equivalent to the reflection map that exchanges $i$ and $j$.

At the end of this section, we recall two technical lemmas from \cite{HL} for the reader's convenience.  The first deal with the relation between $\CF^t$ and $\CFm$ complexes, and the second is the filtered version of the exact triangle detection lemma.

\begin{lemma}[Lemma 2.7 in \cite{HL}] \label{lemma: cft}
Let $\CFm(Y_1)$ and  $\CFm(Y_2)$ be Heegaard chain complexes of three manifolds equipped with a second filtration. Suppose for all $t\geq 0$, the complexes $\CF^t(Y_1)$ and $\CF^t(Y_2)$ are $\F[U]$--equivariantly doubly-filtered quasi-isomorphic, then $\CFm(Y_1)$ and  $\CFm(Y_2)$ are $\F[U]$--equivariantly doubly-filtered  quasi-isomorphic.
\end{lemma}
This is stated and proved in \cite[Section 2]{HL}.

\begin{lemma} [Lemma 2.9 in \cite{HL}]\label{lemma: mapping-cone}
Let $(C_i, \partial_i)_{i \in \Z}$ be a family of filtered chain complexes (over any ring). Suppose we have filtered maps $f_i \co C_i \to C_{i+1}$ and $h_i \co C_i \to C_{i+2}$ so that:
\begin{enumerate}
    \item $f_i$ is an anti-chain map, i.e., $f_i \circ \partial_i + \partial_{i+1} \circ f = 0$.
    \item $h_i$ is a null-homotopy of $f_{i+1} \circ f_i$, i.e., $f_{i+1} \circ f_i + h_i \circ \partial_i + \partial_{i+2} \circ h_i = 0$;
    \item $h_{i+1} \circ f_i + f_{i+2} \circ h_i$ is a filtered quasi-isomorphism from $C_i$ to $C_{i+3}$.
\end{enumerate}
Then the anti-chain map
\[
\begin{pmatrix} f_i \\ h_i \end{pmatrix}\co C_i \to \Cone(f_{i+1})
\]
is a filtered quasi-isomorphism (and hence a filtered homotopy equivalence when working over a field).
\end{lemma}

\begin{proof}
This follows from the proof of \cite[Lemma 4.2]{branched}, adding the key word "filtered" when necessary.
\end{proof}

\section{Alexander grading and surgery cobordisms} \label{sec: alex}
In this section we study the cobordisms involved in the surgery exact triangle, mostly through the lens of periodic domains. We then compute the Alexander grading of holomorphic polygons. These computations play a crucial role in the proof of the filtered surgery exact triangle. We also examine the spin$^c$ structure in the cobordisms more carefully.

\subsection{Triply-periodic domains and relative periodic domains.} \label{ssec: domains}
Assume $Y$ is a rational homology sphere, $K\subset Y$ is a knot of rational genus $g$ with order $d>0$ in $H_1(Y;\Z).$ Let $\lambda$ be a nonzero framing for $K$, and $W_\lambda$ the $2$-handle cobordism from $Y$ to $Y_\lambda.$

Ozsv\'ath and Szab\'o explained how a \emph{pointed Heegaard triple} $( \Sigma,\bm\alpha,\bm\beta,\bm\gamma,z)$ gives rise to the cobordism $W_\lambda$ in  \cite[Section 2.2]{fourmanifold}. First let the curve $\gamma_g$ be a $\lambda$--framed longitude that intersects $\beta_g$ at one point and disjoint from all other $\beta$ curves, and  each $\gamma_i$ for $i= 1,\cdots,g-1 $ be a small pushoff of $\beta_i$, intersecting $\beta_i$ at two points. Clearly a pointed Heegaard triple specifies three manifolds $Y_{\alpha\beta}=Y$, $Y_{\alpha \gamma}=Y_\lambda$ and $Y_{\beta \gamma}=\conn^{g-1}(S^1\times S^2)$. These will be the boundary of the cobordism.

Let $U_\alpha$, $U_\beta$ and $U_\gamma$ denote the handlebody corresponding to each set of attaching circles.  Let $\Delta$ denote the  two-simplex, with vertices $v_\alpha$, $v_\beta$, $v_\gamma$ labeled clockwise, and let $e_i$
denote the edge $v_j$ to $v_k$, where $\{i, j, k\} = {\alpha, \beta, \gamma}$.  Then, form the identification
space
\[
X_{\alpha \beta \gamma}= \frac{(\Delta\times \Sigma)\sqcup (e_\alpha \times U_\alpha)\sqcup (e_\beta \times U_\beta) \sqcup (e_\gamma \times U_\gamma) }{ (e_\alpha \times \Sigma) \sim (e_\alpha \times \partial U_\alpha),(e_\beta \times \Sigma) \sim (e_\beta \times \partial U_\beta),(e_\gamma \times \Sigma) \sim (e_\gamma \times \partial U_\gamma)}
.
\]
Following the standard topological argument, one can smooth out the corners
to obtain a smooth, oriented, four-dimensional  manifold we also call $X_{\alpha \beta \gamma}$. Under the natural orientation conventions
implicit in the above description, we have
\[
\partial X_{\alpha \beta \gamma} =  -Y_{\alpha \beta}  \sqcup -Y_{\beta \gamma} \sqcup   Y_{\alpha \gamma}.
\]
One can use three-handles to kill the boundary component $Y_{\beta \gamma}$, which consists of copies of $S^1\times S^2$. As a result, $X_{\alpha \beta \gamma}$ represents the cobordism $W_\lambda.$ Ozsv\'ath and Szab\'o defined also defined a \emph{triply-periodic domain} to be a two-chain on $\Sigma,$ with multiplicity zero at the basepoint $z$ and  whose boundaries consists of multiples of $\alpha, \beta$  and $\gamma$ curves. The triply-periodic domains represent homology classes in $H_2(X_{\alpha \beta \gamma} ;\Z).$ Ozsv\'ath and Szab\'o proved a formula for the evaluation of the first Chern class (see \cite[Section 6.1]{fourmanifold}) using triply-periodic domains and holomorphic triangles. 

In order to obtain a periodic domain that represents the dual knot in the knot surgery, and suitable for certain homological computations, we make some modification to the above constructions.   This is demonstrated in \cite[Section 4]{HL}. First on $( \Sigma,\bm\alpha,\bm\beta,w,z)$, we now require that $\beta_g$ to be a meridian of $K$ that intersects $\alpha_g$ at one point and disjoint from all other $\alpha$ curves. Moreover, there is an arc $t_\alpha$ from $z$ to $w$  intersecting $\beta_g$ in a single point and disjoint from all other $\alpha$ and $\beta$ curves.

Orient the curves $\alpha_g, \beta_g, \gamma_g$ so that $\#(\alpha_g \cap \beta_g) = \#(\gamma_g \cap \beta_g) = \#(t_\alpha \cap \beta_g) = 1 $. Pick the orientation for $\alpha_i, \beta_i$ arbitrarily for $i\in \{1,\cdots,g-1 \}$, and orient $\gamma_i$ parallel to $\beta_i$.

We wind the $\gamma_g$ curve extra $\text{max}\{k,n\}$ times in the \emph{winding region}, which is a local region that contains all the basepoints, add a second basepoint $z_n$ to the left of the $n$-th winding, and then wind $\gamma_g$ back the same number of times to preserve the framing. The purposes of the windings are such that 
\begin{itemize}
    \item every spin$^c$ structure is represented by  a generator in the winding region;
\item the $w$ (or equivalently $z$) and $z_n$ basepoints represent the knot $K_{n,\gamma}.$ 
\end{itemize}

 With these modifications in place,  a \emph{relative periodic domain}  is defined to be a two-chain on $\Sigma,$ with multiplicity zero at the basepoint $z$, and whose boundary consists of multiples of $\alpha, \beta$ curves and a longitude for the knot (specified by $w$ and $z$). More details are discussed in \cite{periodic}.

There is a triply-periodic domain $P_\gamma$ (see Figure \ref{fig: twistgreen}), with $n_z(P_\gamma) = 0$, $n_w(P_\gamma) = k$, and
\begin{equation}\label{eq: Pgamma-boundary}
    \partial P_\gamma = -d\alpha_g  -k \beta_g + d\gamma_g + \sum_{i=1}^{g-1} (a_i \alpha_i + b_i \beta_i)
\end{equation}
for some integers $a_i, b_i$. We add extra reference points $z_j$ for $j=1,\cdots, n-1$ such that each $z_j$ is in the region on the left of  the $j$--th winding of the $\gamma_g$ curve to the left of $\beta_g.$

 This periodic domain can be viewed as either:
 \begin{enumerate}
     \item  a triply-periodic domain $(\Sigma,\bm\alpha,\bm\beta,\bm\gamma, z)$ representing the class of a capped-off Seifert surface in $H_2(W_\lambda)$ (note the multiplicity of $\gamma_g$);
     \item  a relative periodic domain $(\Sigma,\bm\alpha,\bm\beta,w,z)$ representing $\lambda$--framed longitude of $K$ in $Y$;
       \item  a relative periodic domain $(\Sigma,\bm\alpha,\bm\gamma,w,z_1)$ representing the dual knot of $K$ in $Y_\lambda$ (which we would not use in this paper).
 \end{enumerate}

As of now,  $P_\gamma$ is not  a relative periodic domain $(\Sigma,\bm\alpha,\bm\gamma,w,z_n)$ for the knot $K_{n,\lambda}$ (since the boundary consists of merely copies of the meridian),  but we will demonstrate the modification necessary to achieve this in Section \ref{ssec: spinc}.

 \begin{remark}
 Note that our $P_\gamma$ is slightly different from the setup in \cite{HL}. First, their second basepoint is to the right of $\beta_g$ curve. But in order to represent $(n,1)$--cable of the left-handed meridian, we have to choose our basepoint $z_n$ to the left of $\beta_g$. By setting $n=1$, the argument in this paper will recover their results, despite a different choice of second basepoint.
 
 Our winding is also chosen to the left of $\beta_g$, different from the choice in \cite{HL}. Loosely speaking, this is in order to capture the information of the extra windings coming from $K_{n,\lambda}$. 
 \end{remark}

\begin{figure}

\labellist

 \pinlabel $w$ at 295 70
 \pinlabel $z$ at 255 70
  \pinlabel $z_1$ at 231 140
 
 \pinlabel $z_n$ at 151 140

\pinlabel $\Theta_{\beta\gamma}$ at 285 123

 \pinlabel {{\color{red} $\alpha_g$}} [l] at 390 48
 \pinlabel {{\color{blue} $\beta_g$}} [r] at 267 173
 \pinlabel {{\color{OliveGreen} $\gamma_g$}} [l] at 388 120

 {\tiny
  \pinlabel $(n+1)d$  at 141 52
  }

    \pinlabel $k+d$ at 325 40
   \pinlabel $k+d$ at 327 145
  \pinlabel $d$ at 245 40
   \pinlabel $\cdots$ at 195 40
   \pinlabel $nd$ at 137 32
  \pinlabel $nd$  at 141 102
 \pinlabel $\cdots$  at 177 102
 \pinlabel $d$  at 209 102
  \pinlabel $0$  at 249 100
  \pinlabel $k$  at 325 100

\endlabellist

\includegraphics{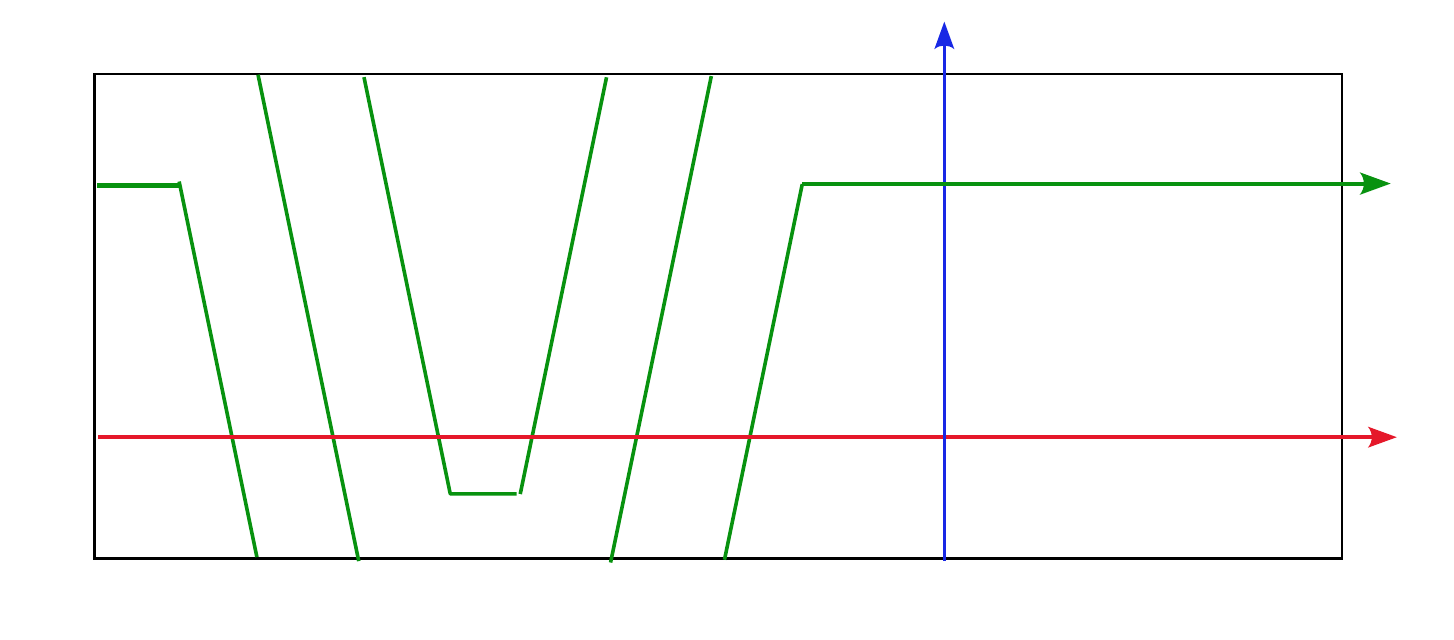}
\caption{The winding region of diagram $P_{\gamma}$, where $n=3$.}
\label{fig: twistgreen}
\end{figure}

Similarly, one can define a relative periodic domain $P_{\delta}$ for large surgery.
 For an integer $b$ such that $0\leq b \leq m-n$, let $\bm\delta^{m,b} = (\delta_1^{m,b}, \dots, \delta_g^{m,b})$ be a tuple of curves obtained from $\bm\gamma$ as follows: Let $\delta_g$ be a parallel pushoff of $\gamma_g$ by performing $m$ left-handed Dehn twists parallel to $\beta_g$, where $b$ (resp. $m-b$) of these twists are performed in the winding region on the same side of $\beta_g$ as $w$ (resp. $z$).  For $i=1, \dots, g-1$, $\delta_i^{m,b}$ is simply a small pushoff of $\gamma_i$ meeting it in two points.  The pointed Heegaard triple $(\Sigma, \bm\alpha, \bm\beta, \bm\delta^{m,b}, z)$ represents the cobordism $W_{\lambda + m\mu}$ from  $Y$ to $Y_{\lambda + m\mu}$. When $m$ and $b$ are understood from the context, we omit the superscripts from the $\delta$ curves.

The relative periodic domain $P_\delta$  shown in Figure \ref{fig: twistpurple} satisfies $n_z(P_\delta) = 0$, $n_w(P_\delta) = k+md$, $n_{z_j}(P_\delta) = jd$ for $j=1,\cdots, n$ and
\begin{align}
\label{eq: Pdelta-boundary}
\partial P_\delta &= -d\alpha_g  -(k+md) \beta_g + d\delta_g + \sum_{i=1}^{g-1} (a_i \alpha_i + b_i \delta_i).
\end{align}
for some integers $a_i, b_i$.

 This periodic domain can be viewed as either:
 \begin{enumerate}
     \item  a triply-periodic domain $(\Sigma,\bm\alpha,\bm\beta,\bm\delta, z)$ representing the class of a capped-off Seifert surface in $H_2(W_{\lambda+m\mu})$ (note the multiplicity of $\delta_g$);
     \item  a relative periodic domain $(\Sigma,\bm\alpha,\bm\beta,w,z)$ representing a ${\lambda+m\mu}$--framed longitude of $K$ in $Y$;
          \item  a relative periodic domain $(\Sigma,\bm\alpha,\bm\delta,w,z_1)$ representing the dual knot  of $K$ in $Y_{\lambda+m\mu}$ (which we would not use in this paper).
 \end{enumerate}
Similar as before,  $P_\delta$ also is not  a relative periodic domain $(\Sigma,\bm\alpha,\bm\delta,w,z_n)$ for the knot $K_{n,\lambda+m\mu}$ yet, while the modification necessary to achieve this is demonstrated in Section \ref{ssec: spinc}.

 For $j=1, \dots, g-1$, there are small periodic domains $S_{\beta\gamma}^j$ and $S_{\gamma\delta}^j$ with $\partial S_{\beta\gamma}^j = \beta_j - \gamma_j$ and $\partial S_{\gamma\delta}^j = \gamma_j - \delta_j$, supported in a small neighborhood of each pair of curves. We will refer to these as \emph{thin domains}.

Relating the periodic domains $P_\gamma$ to $P_\delta,$ there is a $(\beta,\gamma,\delta)$ triply periodic domain $Q$ with
\[
\partial Q = m\beta_g + \gamma_g - \delta_g ,
\]
so that
\[
P_\gamma - P_\delta = dQ + \text{thin domains}.
\]

 Furthermore, between the basepoints $z$ and $z_n$, we add some extra reference points  $u_j$  for $j=1,\cdots, n$, where  each   $u_j$ is in the region   on the right of the $j$--th winding of $\gamma_g$ to the left of $\beta_g$, as shown in Figure \ref{fig: Q}. These reference points are needed to compute the Alexander grading shifts in certain cobordisms.  Note that for $j=1,\cdots, n$, $z_j$ and $u_j$ are only separated by the curve  $\gamma_g$ and  $z_{j-1}$ and $u_j$ are only separated by the curve $\delta_g$ (considering $z$ to be $z_0$ here).

Finally, define the $(\alpha, \gamma, \delta)$ periodic domain
\[
R = \tfrac{m}{\nu} P_\gamma + \tfrac{k}{\nu} Q
\]
where $\nu = \gcd(k,m)$; it has
\[
\partial R = -\frac{md}{\nu} \alpha_g  +  \frac{k+md}{\nu}\gamma_g  - \frac{k}{\nu} \delta_g + \frac{m}{\nu} \sum_{i=1}^{g-1} (a_i \alpha_i + b_i \gamma_i).
\]

\begin{figure}
\labellist
 \pinlabel $w$ at 210 70
 \pinlabel $z$ at 170 70
 \pinlabel $z_n$ at 66 120

 \pinlabel {{\color{red} $\alpha_g$}} [l] at 345 35
 \pinlabel {{\color{blue} $\beta_g$}} [r] at 178 160
 \pinlabel {{\color{purple} $\delta_g$}} [l] at 345 70
 
 \tiny
 \pinlabel $\Theta_{\delta\beta}$ [br] at 204 80
 
  \pinlabel $k+md$ at 320 70
    \pinlabel $-bd$ at 320 60
 
\pinlabel $(m-b+1)d$ at 30 130
\pinlabel $k+md$ at 200 120
\pinlabel $+d$ at 200 110
\pinlabel $k+md$ at 234 120
\pinlabel $k+md$ at 266 120
\pinlabel $-d$ at 266 110
\pinlabel $k+md$ at 310 120
\pinlabel $-(b-1)d$ at 310 110

\pinlabel $k+md$ at 210 35
\pinlabel $+d$ at 210 25
\pinlabel $k+md$ at 250 35
\pinlabel $...$ at 280 35

\pinlabel $k+md$ at 317 35
\pinlabel $-(b-1)d$ at 317 25

\pinlabel $nd$ at 55 72
 \pinlabel $nd$ at 85 32
 \pinlabel $...$ at 90 72
 \pinlabel $...$ at 120 32
 \pinlabel $d$ at 150 120
  \pinlabel $d$ at 155 32
 \pinlabel $0$ at 155 62

\endlabellist
\includegraphics{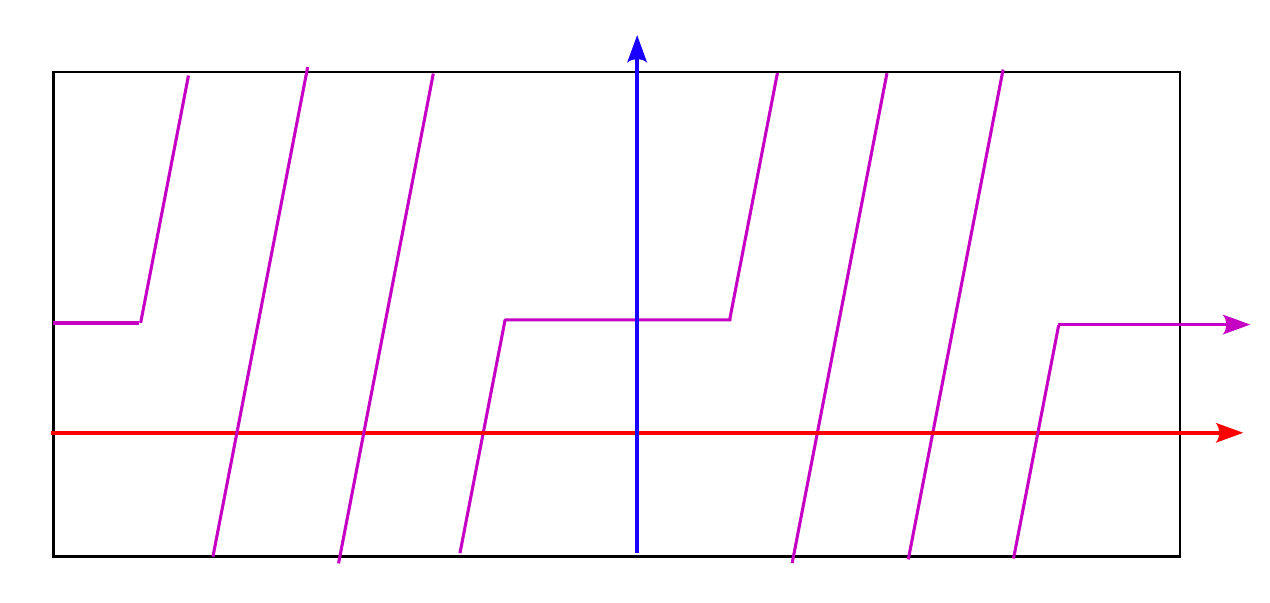}
\caption{The winding region of diagram $P_{\delta}$, in the case where $m=6$ and $b=3$.}
\label{fig: twistpurple}
\end{figure}

\begin{figure}

\labellist

 \pinlabel $w$ at 285 70
 \pinlabel $z$ at 260 70
 {\scriptsize
  \pinlabel $z_3$ at 161 145
 \pinlabel $u_3$ at 178 145
 \pinlabel $z_2$ at 194 145
 \pinlabel $u_2$ at 209 145
 
 \pinlabel $z_1$ at 241 145
 \pinlabel $u_1$ at 255 120
 }

 \pinlabel {{\color{red} $\alpha_g$}} [l] at 390 48
 \pinlabel {{\color{purple} $\delta_g$}} [l] at 389 74
 \pinlabel {{\color{blue} $\beta_g$}} [r] at 267 173
 \pinlabel {{\color{OliveGreen} $\gamma_g$}} [l] at 388 120

 {\tiny

 \pinlabel $\Theta_{\delta\beta}$ at 265 93
\pinlabel $\Theta_{\gamma\delta}$ at 303 136
\pinlabel $\Theta_{\beta\gamma}$ at 263 136

\pinlabel $0$ at 213 80
    \pinlabel $-1$ at 196 80
    \pinlabel $0$ at 182 80
    \pinlabel $-1$ at 165 80
    \pinlabel $0$ at 140 72

 \pinlabel $-1$ at 109 72
  \pinlabel $-2$ at 80 72
    \pinlabel $-3$ at 45 72

 \pinlabel $-m-1$ at 285 116
  \pinlabel $-m$ at 300 100
   \pinlabel $-m+1$ at 320 88
   \pinlabel $-m+2$ at 360 100
   \pinlabel $-m+3$ at 370 72
   
     \pinlabel $-m$ at 285 145
   \pinlabel $\cdots$ at 324 145

   \pinlabel $-m+3$ at 366 145

 \pinlabel $-1$ at 135 116
  \pinlabel $-2$ at 100 116
    \pinlabel $-3$ at 71 116
        \pinlabel $-4$ at 41 115

   \pinlabel $0$ at 246 73
  \pinlabel $0$ at 245 40
    \pinlabel $-1$ at 220 40
    \pinlabel $0$ at 203 40
   \pinlabel $\cdots$ at 182 40

  \pinlabel $-1$  at 249 105

 }

\endlabellist

\includegraphics{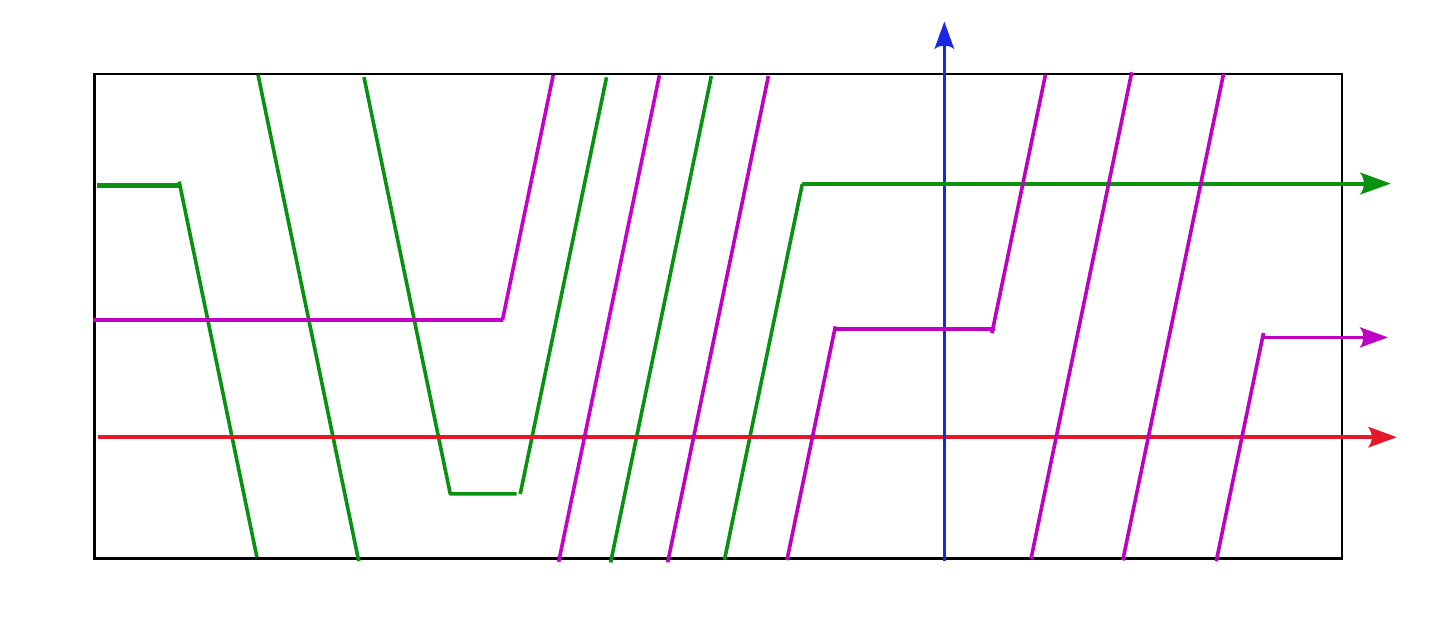}
\caption{The relative periodic domain $Q$, where $m=6$ and $n=3$. Here $\alpha_g$ has multiplicity zero in $\partial Q$.}
\label{fig: Q}
\end{figure}

The multiplicities of the periodic domains at the various basepoints are as follows:
\begin{center}
\begin{tabular} {|c|c|c|c|c|} \hline
& $n_z$ & $n_w$ & $n_{z_n}$  \\ \hline
$P_\gamma$ & $0$ & $k$ & $nd$  \\
$P_\delta$ & $0$ & $k+md$ & $nd$  \\
$Q$ & $0$ & $-m$ & $0$  \\
$R$ & $0$ & $0$ & $\frac{nmd}{\nu}$  \\ \hline
\end{tabular}
\end{center}

\subsection{Topology of the cobordisms} \label{ssec: cob}
Consider the topology of the cobordisms related to the Heegaard diagram $(\Sigma,
\bm\alpha, \bm\beta, \bm\gamma, \bm\delta)$.

According to the construction demonstrated in \cite[Section 8.1.5]{OSclosed}, we have three separate $4$-manifolds $X_{\alpha\beta\gamma\delta}$, $X_{\alpha\gamma\delta\beta}$, and $X_{\alpha\delta\beta\gamma}$, with:
\begin{align*}
\partial X_{\alpha\beta\gamma\delta} &= {-Y_{\alpha\beta}} \sqcup {-Y_{\beta\gamma}} \sqcup
{-Y_{\gamma\delta}} \sqcup {Y_{\alpha\delta}} \\
\partial X_{\alpha\gamma\delta\beta} &= {-Y_{\alpha\gamma}} \sqcup {-Y_{\gamma\delta}} \sqcup {-Y_{\delta\beta}} \sqcup Y_{\alpha\beta} \\
\partial X_{\alpha\delta\beta\gamma} &= {-Y_{\alpha\delta}} \sqcup {-Y_{\delta\beta}} \sqcup {-Y_{\beta\gamma}}  \sqcup {Y_{\alpha\gamma}}
\end{align*}
These $4$-manifolds each admit a pair of decompositions as follows:
\begin{align}
\label{eq: Xabgd-decomp}
X_{\alpha\beta\gamma\delta} &= X_{\alpha\beta\gamma} \cup_{Y_{\alpha\gamma}} X_{\alpha\gamma\delta} = X_{\alpha\beta\delta} \cup_{Y_{\beta\delta}} X_{\beta\gamma\delta} \\
\label{eq: Xagdb-decomp}
X_{\alpha\gamma\delta\beta} &= X_{\alpha\gamma\delta} \cup_{Y_{\alpha\delta}} X_{\alpha\delta\beta} =
X_{\alpha\gamma\beta} \cup_{Y_{\gamma\beta}} X_{\gamma\delta\beta} \\
\label{eq: Xadbg-decomp}
X_{\alpha\delta\beta\gamma} &= X_{\alpha\delta\beta} \cup_{Y_{\alpha\beta}} X_{\alpha\beta\gamma} =
X_{\alpha\delta\gamma} \cup_{Y_{\delta\gamma}} X_{\delta\beta\gamma},
\end{align}
where the $3$-manifolds in above notations are precisely:
\begin{align*}
Y_{\alpha\beta} &= Y & Y_{\alpha\gamma}&= Y_\lambda(K) & Y_{\alpha\delta} &= Y_{\lambda+m\mu}(K) \\
Y_{\beta\gamma} &= \conn^{g-1}(S^1 \times S^2)  & Y_{\gamma\delta} &= L(m,1) \conn^{g-1}(S^1 \times S^2) & Y_{\delta\beta} &= \conn^{g-1}(S^1 \times S^2) \\
Y_{\gamma\beta} &= -Y_{\beta\gamma} & Y_{\delta\gamma} &= -Y_{\gamma\delta} & Y_{\beta\delta} &= -Y_{\delta\beta}
\end{align*}
Note also that $X_{\alpha\gamma\beta} = -X_{\alpha\beta\gamma}$, and so on.

If we let $\bar X_{\alpha\beta\gamma}$, $\bar X_{\alpha\beta\gamma\delta}$, etc. denote the manifolds obtained by attaching $3$-handles to kill the $S^1 \times S^2$ summands in $Y_{\beta\gamma}$, $Y_{\gamma\delta}$, and $Y_{\delta\beta}$, then we have analogues of \eqref{eq: Xabgd-decomp}, \eqref{eq: Xagdb-decomp}, and \eqref{eq: Xadbg-decomp} for these manifolds as well.

The periodic domains $\{P_\gamma, P_\delta, Q, R\}$ represent homology classes which survive in $H_2(\bar X_{\alpha\beta\gamma\delta})$ and  the following relations hold:
\begin{align*}
[P_\delta] &= [P_\gamma] - d[Q], & [R] = \tfrac{m}{\nu} [P_\gamma] + \tfrac{k}{\nu} [Q].
\end{align*}
Hence, we may also write $[R] = \tfrac{m}{\nu}[P_\delta] + \tfrac{k+md}{\nu}[Q]$. The same relations are also satisfied in $H_2(\bar X_{\alpha\gamma\delta\beta})$ and $H_2(\bar X_{\alpha\delta\beta\gamma})$.

We can obtain the cobordism $W_\lambda$ from $\bar X_{\alpha\beta\gamma}$, simply by gluing a $4$-handle to kill the $S^3$ boundary component left over from $Y_\beta\gamma$. Let $\spincs^0_{\beta\gamma}$ denote the unique torsion spin$^c$ structure on $Y_{\beta\gamma}$. Let $\Theta_{\beta\gamma}$ and $\Theta_{\gamma\beta}$ denote the standard top-dimensional generators for $\CF^{\le0}(\Sigma, \bm\beta, \bm\gamma, z)$ and $\CF^{\le0}(\Sigma, \bm\gamma, \bm\beta, z)$, both of which use the unique intersection point in $\beta_g \cap \gamma_g$ as shown in Figure \ref{fig: twistgreen}.

Similarly, one can close off $\bar X_{\alpha\delta\beta}$ by attaching a $4$-handle to the remaining $S^3$ boundary component of $Y_{\delta\gamma}$. This will give us cobordism $W'_{\lambda+m\mu}$, which is $W_{\lambda+m\mu}$ with the orientation reversed, viewed as a cobordism from $Y_{\lambda+m\mu}(K)$ to $Y$. Define $\spincs^0_{\beta\delta}$, $\Theta_{\beta\delta}$, and $\Theta_{\delta\beta}$ analogously. Both $\Theta_{\beta\delta}$ and $\Theta_{\delta\beta}$ use the unique intersection point in $\beta_g \cap \delta_g$ as shown in Figure \ref{fig: twistpurple}.

Next let $W_{\gamma\beta\delta}$ denote the $4$-manifold obtained from $\bar X_{\gamma\beta\delta}$ by deleting a neighbourhood of an arc connecting $Y_{\beta\gamma}$ and $Y_{\beta\delta}$ (both are $\conn^{g-1}S^1\times S^2$). If we let   $B_m$ be the Euler number $m$ disk bundle over $S^3,$ which has the boundary $L(m,1)$, $W_{\gamma\beta\delta}$ is diffeomorphic to $((\#^{g-1} S^1 \times S^2) \times I) \bconn B_m$, and $Q$ corresponds to the homology class of the zero section in $B_m$. 

Following \cite[Definition 6.3]{rational}, define $\spincs^0_{\gamma\delta}$ to be the unique spin$^c$ structure on $Y_{\gamma\delta}$ that is torsion and has an extension $\spinct$ to $W_{\gamma\beta\delta}$ which satisfies $\gen{c_1(\spinct), [S^2]} = \pm m$. Pair the $m$ intersection points of $\gamma_g \cap \delta_g$ the top-dimensional intersection points of $\gamma_j \cap \delta_j$ ($j=1, \dots, g-1$) to obtain $m$ canonical cycles in $\CF^{\le 0}(\Sigma, \bm\gamma, \bm\delta, z)$, each of which represents a different torsion spin$^c$ structure on $Y_{\gamma\delta}$. Let $\Theta_{\gamma\delta}$ be the generator which uses the point of $\gamma_j \cap \delta_j$ that is adjacent to $w$, as shown in Figure \ref{fig: Q}.

The following is proved by Hedden and Levine:

\begin{lemma}[Lemma 5.2 in \cite{HL}] \label{lemma: Theta-gamma-delta}
The generator $\Theta_{\gamma\delta}$ represents $\spincs^0_{\gamma\delta}$.
\end{lemma}

Moreover, there is a class $\tau^+_0 \in \pi_2(\Theta_{\beta\gamma}, \Theta_{\gamma\delta}, \Theta_{\beta\delta})$ such that the intersection of its domain with the winding region is the small triangle above $w$ in Figure \ref{fig: Q}. A similar computation as in \cite{HL} yields
\[
\gen{c_1(\spincs_z(\tau_0^+)), [Q]} = -m.
\]

Let $X_*$ be any of the $4$-manifolds defined above (either with three or four subscripts).  For the rest of the paper, we will use $\Spin^c_0(X_*)$ to denote the set of spin$^c$ structures that restricts to $\spincs^0_{\beta\gamma}$ on $Y_{\beta\gamma}$, $\spincs^0_{\gamma\delta}$ on $Y_{\gamma\delta}$, and $\spincs^0_{\delta\beta}$ on $Y_{\delta\beta}$, for whichever applicable. Note that all such spin$^c$ structures extend uniquely to $\bar X_*$.

We take a look at the intersection forms on different $4$-manifolds $X_*$ to finish this subsection. In $X_{\alpha\beta\gamma\delta}$, the classes $[P_\gamma]$, $[P_\delta]$, $[Q]$, and $[R]$ can be represented by surfaces contained in $X_{\alpha\beta\gamma}$, $X_{\alpha\beta\delta}$, $X_{\beta\gamma\delta}$, and $X_{\alpha\gamma\delta}$, respectively. In particular,  the pair $[P_\gamma]$ and $[R]$ can be represented by disjoint surfaces in $X_{\alpha\beta\gamma\delta}$. The same is for the pair $[P_\delta] $and $ [Q]$. Thus 
\begin{equation} \label{eq: abgd-int-form}
[P_\gamma] \cdot [R] = [P_\delta] \cdot [Q] = 0,
\end{equation}
With orientation given by the cobordisms, the other pairs are
\begin{equation} \label{eq: abgd-PD-self-int}
\begin{aligned}
[P_\gamma]^2 &= dk & [P_\delta]^2 &= d(k+md) \\ [Q]^2 &=-m &  [R]^2 &= -\frac{m k (k+md)}{\nu^2} .
\end{aligned}
\end{equation}

In $X_{\alpha\gamma\delta\beta}$, the same classes have different self-intersection numbers due to a change of the orientation, and different pairs are disjoint according to the decomposition given by (\ref{eq: Xagdb-decomp}). In this case,
\begin{gather*}
[R] \cdot [P_\delta] = [P_\gamma]\cdot [Q] = 0\\
\begin{aligned}
[P_\gamma]^2 &= -dk & [P_\delta]^2 &= -d(k+md) \\
[Q]^2 &=-m &  [R]^2 &= -\frac{m k (k+md)}{\nu^2}.
\end{aligned}  
\end{gather*}

The sign of the self intersection $[P_\gamma]^2$ (resp.~$[P_\delta]^2$) is reversed because it is contained in $X_{\alpha\gamma\beta}$ (resp.~$X_{\alpha\delta\beta}$), which is diffeomorphic to $-X_{\alpha\beta\gamma}$ (resp.~$-X_{\alpha\beta\delta}$). This reverse of the orientation is equivalent to turning the cobordism $W_\lambda$ (resp.~$W_{\lambda+m\mu}$) around; denote the resulting cobordism $W'_\lambda(K)$ (resp.~$W'_{\lambda+m\mu}$). Parallel results hold for $X_{\alpha\delta\beta\gamma}$ as well.

\begin{figure}[!ht]

\subfloat[]{
\labellist
\tiny

 \pinlabel $nd$ at 130 120
  \pinlabel $2nd$ at 100 120

 \pinlabel $n(k+md$ at 185 130
  \pinlabel $+d)$ at 185 115

 \pinlabel $nd$ at 155 30
 \pinlabel $2nd$ at 118 30
  \pinlabel $...$ at 78 30

 \pinlabel $0$ at 145 50

 \pinlabel $n(k+md)$ at 195 67
\endlabellist
\includegraphics[width=0.5\textwidth]{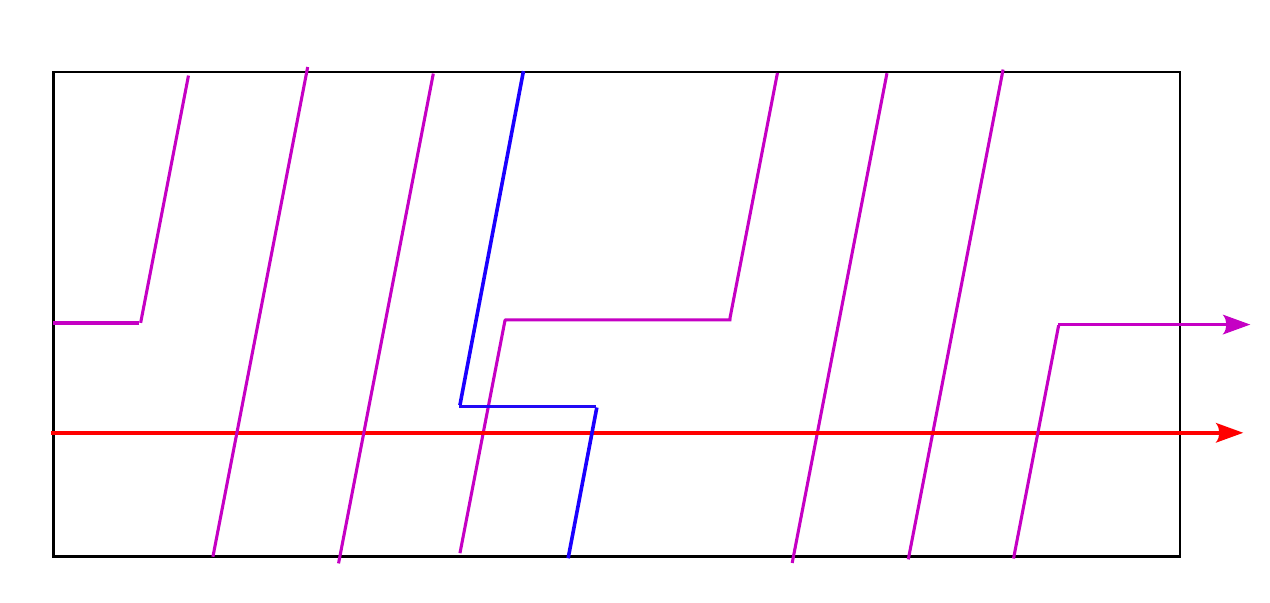}
\label{subfig: purplea}
}
\subfloat[]{
\labellist
\tiny
\pinlabel {{\color{purple} $k+md$}} at 100 30
\pinlabel {{\color{purple} $...$}} at 145 30
\pinlabel {{\color{purple} $k+md$}} at 76 120

\pinlabel {{\color{purple} $2(k$}} at 120 120
\pinlabel {{\color{purple} $+md)$}} at 118 100
\endlabellist
\includegraphics[width=0.5\textwidth]{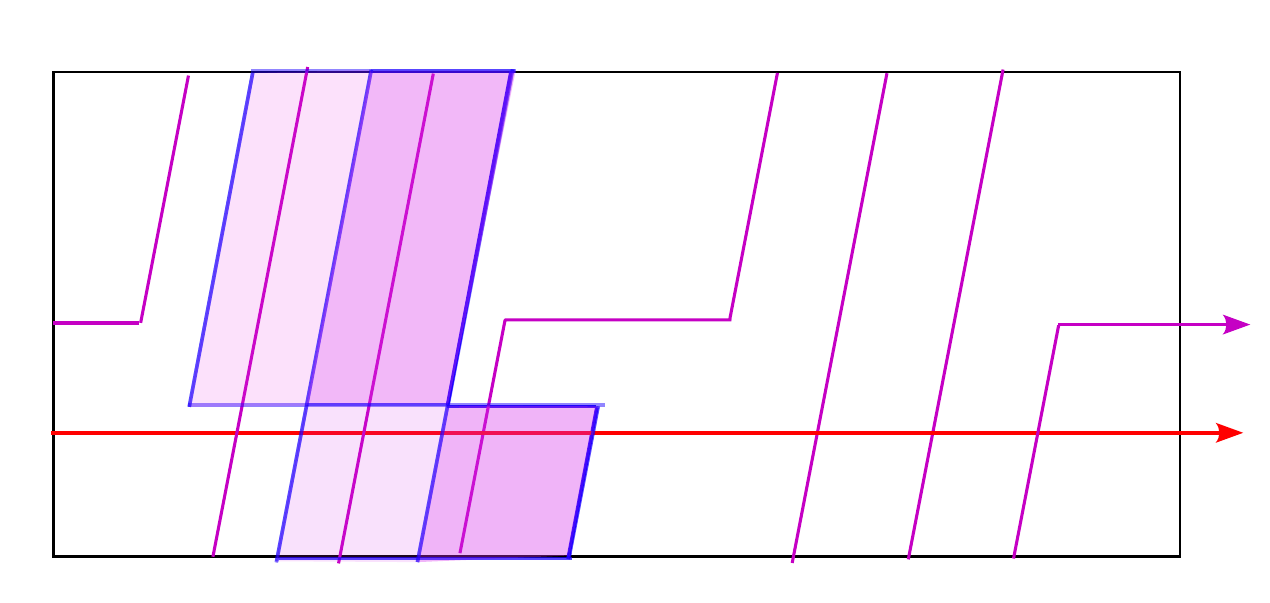}
\label{subfig: purpleb}
}

\subfloat[]{
\labellist

 \pinlabel $z$ at 150 70
  \pinlabel $w$ at 205 70
 \pinlabel $z_n$ at 52 135
 
 \tiny
   \pinlabel $-n(k+md)$ at 17 135
   \pinlabel $+(n+1)nd$ at 18 125
  \pinlabel $-n(k+md)$ at 48 105
  \pinlabel $+n^2d$ at 48 95
  
  \pinlabel $-(n-1)(k+md)$ at 44 53
  \pinlabel $n^2d$ at 56 63
 
 \pinlabel $...$ at 78 105
 \pinlabel $...$ at 100 105
 
 \pinlabel $nd$ at 170 120
 
 \pinlabel $-(k+md)$ at 125 120
  \pinlabel $+nd$ at 117 110
  
   \pinlabel $0$ at 225 120
     \pinlabel $-nd$ at 258 120
       \pinlabel $-2nd$ at 304 120
 
  \pinlabel $k+md$ at 175 52
  \pinlabel $k+md-nd$ at 245 52
  
  \pinlabel $...$ at 120 32
  \pinlabel $...$ at 137 32
  \pinlabel $...$ at 282 52
  
    \pinlabel $nd$ at 185 32
     \pinlabel $0$ at 240 32
     \pinlabel $-nd$ at 272 32
       \pinlabel $-2nd$ at 304 32

\endlabellist
\includegraphics{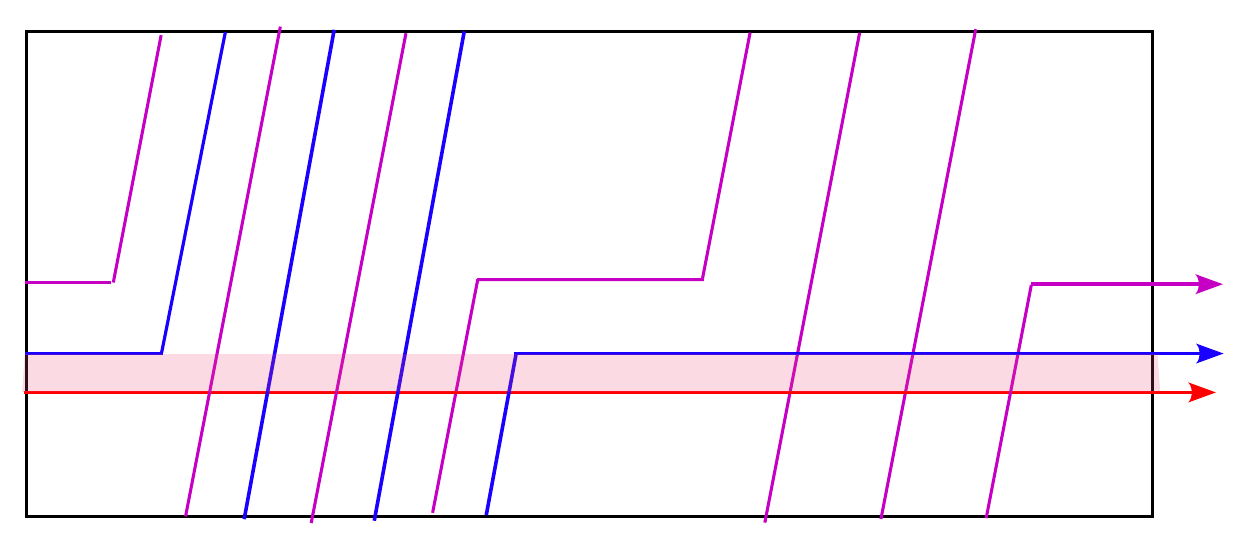}
\label{subfig: purplec}
}
\caption{The construction of the relative periodic domain $P_{n,\delta}$, in the case where $m=6, b=3$ and $n=3$.}
\label{fig: purple}
\end{figure}
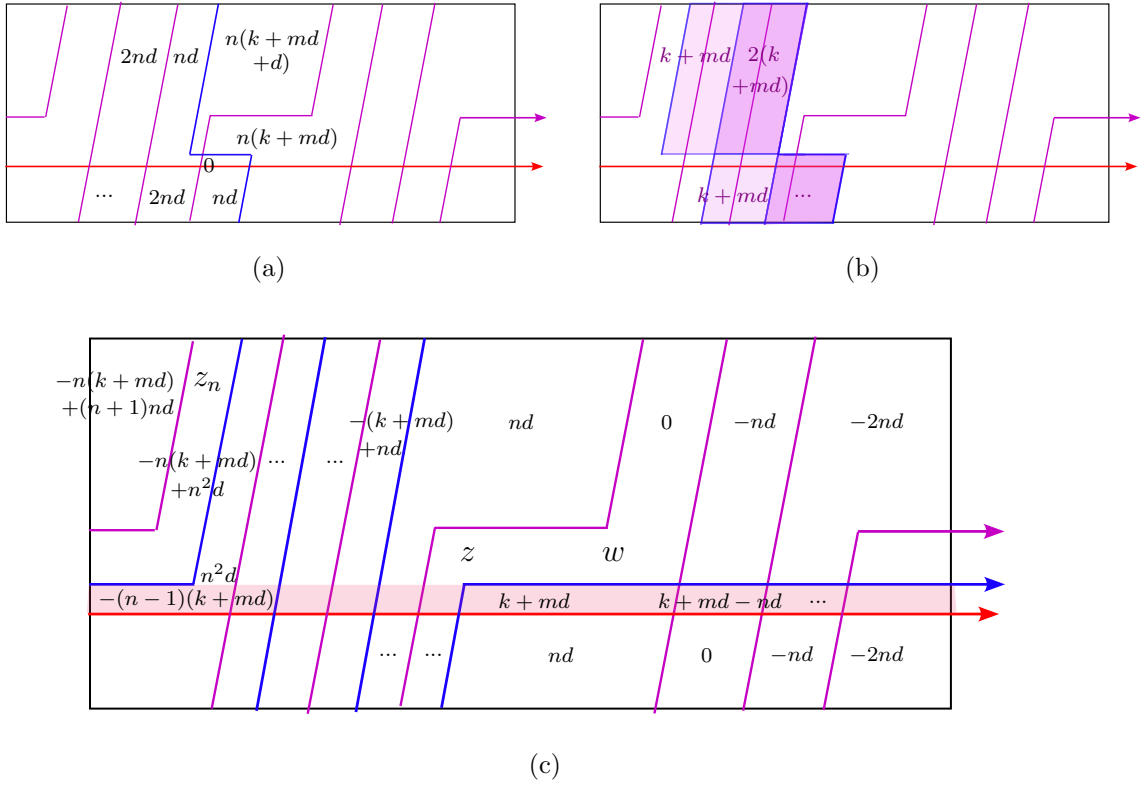

\subsection{Polygons, spin$^c$ structures, and Alexander gradings} \label{ssec: spinc}
In this section, we compute of the Alexander grading shift and the first Chern class evaluation with respect to the holomorphic triangles and rectangles. We will make use of these computations in both the proof of the large surgery and the proof of the surgery exact triangle.

Following the convention from \cite{HL}, throughout the rest of the paper we will generally refer to elements of $\T_\alpha \cap \T_\beta$ as $\x$ or $\y$, elements of $\T_\alpha \cap \T_\gamma$ as $\q$ or $\r$, and elements of $\T_\alpha \cap \T_\delta$ as $\a$ or $\b$. We also introduce the following notational shorthands:
\begin{align*}
\AlNorm(\x) &= \AlNorm_{w,z}(\x) & \AlNorm(\q) &= \AlNorm_{w,z_n}(\q) & \AlNorm(\a) &=  \AlNorm_{w,z_n}(\a) \\
\Al(\x) &= d \AlNorm_{w,z}(\x) & \Al(\q) &= k \AlNorm_{w,z_n}(\q) & \Al(\a) &= (k+md) \AlNorm_{w,z_n}(\a).
\end{align*}

Although implicit in the notation, we remind the reader that   $\AlNorm(\q), \Al(\q), \AlNorm(\a)$ and $\Al(\a)$ are  dependent on the number $n$ through the basepoint $z_n.$

The following Alexander grading formula using relative periodic domains and holomorphic triangles, proved in \cite[Section 2.3]{HL}, is going to be our main tool along with the first Chern class formula from \cite{fourmanifold}.

\begin{proposition}[Proposition 1.3 in \cite{HL}] \label{prop: abs-alex} Let $(\Sigma, \alpha,\beta,w,z)$ be a doubly-pointed Heegaard diagram for a knot $(Y,K)$  representing a class in $H_1(Y)$ of order $d$, and let $P$ be a relative periodic domain specifying a homology class $[P]\in H_2(Y,K)$.  Then
the absolute Alexander grading of a generator $\x \in \T_\alpha \cap \T_\beta$ is given by
\begin{align}
\label{eq: abs-alex} \AlNorm_{w,z}(\x) &= \frac{1}{2d} \left( \hat\chi(P) + 2n_\x(P) - n_{\bar z}(P) - n_{\bar w}(P) \right)
\end{align}
where $\hat\chi(P)$ is the Euler measure of $P$, $n_\x(P)$ denotes the sum of the average of the four local multiplicities of $P$ in the regions abutting $x_j$ for all the $x_j \in \x$, and $n_{\bar w}(P)$ (resp.~$n_{\bar z}(P)$) denotes the average of the multiplicities of $P$ on either side of the longitude at $w$ (resp.~$z$).

\end{proposition}

\subsubsection{Modified relative periodic domains.}

A key ingredient of our argument is the relative periodic domain that represents the $(n,1)$--cable of the meridian. We now demonstrate the construction.

Starting with n copies of $P_\delta$, isotope the $\beta_g$ curve as in Figure \ref{subfig: purplea}. First add rectangle strips with multiplicity $(n-1)(k+md),\cdots,2(k+md),k+md$, from right to left, as shown in Figure \ref{subfig: purpleb}. In this stage we already obtain a periodic domain whose boundary, after gluing in the disks along $\alpha$ and $\delta$ attaching curves,  is identified with $k+md$ copies of the $(n,1)$--cable of the meridian, as required. The only problems are that the new blue curve is immersed, and the multiplicity at $z$ basepoint is not $0$. Next, add the shadowed rectangle strip with multiplicity $k+md$ depicted in Figure \ref{subfig: purplec}: this is equivalent to sweeping the horizontal portion of the blue curve across the disk attached along $\alpha_g$, thus replacing this portion of the blue curve with the remaining boundary of the disk. Finally adjust the multiplicity by adding $-n(k+md)\Sigma$ to the domain, resulting in the relative periodic domain  in Figure \ref{subfig: purplec}; denote it as $P_{n,\delta}$. 

The two-chain $P_{n,\delta}$ is a relative periodic domain $(\Sigma,\bm\alpha,\bm\delta,w,z_n)$ for the knot $K_{n,\lambda+m\mu}$ in $Y_{\lambda+m\mu}(K)$. In $P_{n,\delta}$, $w$ and $z$ basepoints are interchangeable, $n_z(P_{n,\delta}) = n_w(P_{n,\delta}) = 0$, $n_{z_n}(P_\gamma) = -n(k+md) + n^2 d$, and
\[
\partial P_\gamma = (k+md-nd)\alpha_g  - (k + md) \beta_{n,g} + nd\gamma_g + \sum_{i=1}^{g-1} (na_i \alpha_i + nb_i \beta_i)
\]
for some integers $a_i, b_i$.

We have $[P_{n,\delta}]=n[P_\delta] \in H_2(W_{\lambda})$, since each time attaching a rectangle strip is equivalent to a homotopy.

\begin{proposition}
The Euler measure $\hat\chi(P_{n,\delta})=n\hat\chi(P_{\delta}).$
\end{proposition}

\begin{proof}
According to \cite[Lemma 6.2]{fourmanifold}, for a triply-periodic domain $P=\sum_i n_i \mathcal{D}_i,$ the Euler measure can be calculated by
\[
\hat\chi(P) = \sum_i n_i \left(  \chi(\text{int}\mathcal{D}) - \frac{1}{4}(\# \text{corner points of }\mathcal{D}) \right). 
\]
From the formula we see that isotopy, adding rectangle domains and copies of $\Sigma$ do not change the Euler measure. Thus $\hat\chi(P_{n,\delta})=n\hat\chi(P_{\delta}).$
\end{proof}

\subsubsection{Alexander grading shifts on holomorphic triangles.} \label{sssec: triangle}
We study the intersection points in $(\Sigma, \bm\alpha, \bm\beta, \bm\delta^{m,b}, w, z, z_n)$ more carefully (see Figure \ref{fig: twistpurple}) before moving on to the computation.
The curve $\alpha_g $ and $\delta_g^{m,b}$ intersects $m$ times in the winding region, we will label the intersection points as such: following the orientation of $\delta_g^{m,b}$, let $p_{b-m}, \dots, p_{-1}$ denote the $b-m$ points to the left of $\beta_g$, with $p_{-1}$ being the closest, and $p_{0}, \dots, p_{b-1}$ denote the $b$ points to the right of $\beta_g$, with $p_{0}$ being the closest. Let $q$ be the unique intersecting point of $\alpha_g$ and $\beta_g$. For $\x \in \T_\alpha\cap \T_\beta$ and $l \in \{b-m, \dots, b-1\}$, we define $\x_l^{m,b} \in \T_\alpha \cap\T_{\delta^{m,b}}$ to be the point obtained by replacing $q$ with $p_l$ and taking ``nearest points'' in thin domains. There is a \emph{small triangle} $\psi_{\x,l}^{m,b} \in \pi_2(\x_l^{m,b}, \Theta_{\delta\beta}, \x)$ in the winding region satisfying
\begin{align}\label{eq: smalltriangle}
    (n_w(\psi_{\x,l}),n_z(\psi_{\x,l}),n_{z_n}(\psi_{\x,l}))=&
    \begin{cases}
    (l,0,0) & l\geq 0  \\
    (0,-l,0) & -n<l<0  \\
    (0,-l,-l-n) & l\leq -n.
    \end{cases}
\end{align}

\begin{figure}

\labellist

 \pinlabel $w$ at 275 95
 \pinlabel $z$ at 245 95
 {\scriptsize
  \pinlabel $z_n$ at 163 145
 }

 \pinlabel {{\color{red} $\alpha_g$}} [l] at 390 48

 \pinlabel {{\color{blue} $\beta_{n,g}$}} [l] at 389 74
 \pinlabel {{\color{OliveGreen} $\gamma_g$}} [l] at 388 120

  \pinlabel $0$ at 325 105
   \pinlabel $k$ at 325 69
   \pinlabel $nd$ at 320 40
    \pinlabel $nd$ at 304 145
   
 {\tiny

  \pinlabel $-k$ at 196 80
    \pinlabel $nd$ at 196 90

    \pinlabel $-k$ at 184 105
    \pinlabel $2nd$ at 188 115

   \pinlabel $2nd$ at 167 90
    \pinlabel $-2k$ at 165 80

    \pinlabel $n(-k$ at 140 75
    \pinlabel $+nd)$ at 140 65

 \pinlabel $-nk$ at 131 135
  \pinlabel $(n+1)nd$ at 131 145

\pinlabel $-nk$ at 100 106
  \pinlabel $n^2d$ at 100 116

    \pinlabel $-k$ at 223 35
    \pinlabel $nd$ at 225 45
    \pinlabel $-k$ at 203 35
    \pinlabel $2nd$ at 205 45
    
     \pinlabel $-2k$ at 185 35
    \pinlabel $2nd$ at 188 45
    
   \pinlabel $\cdots$ at 172 40

 }

\endlabellist
\includegraphics{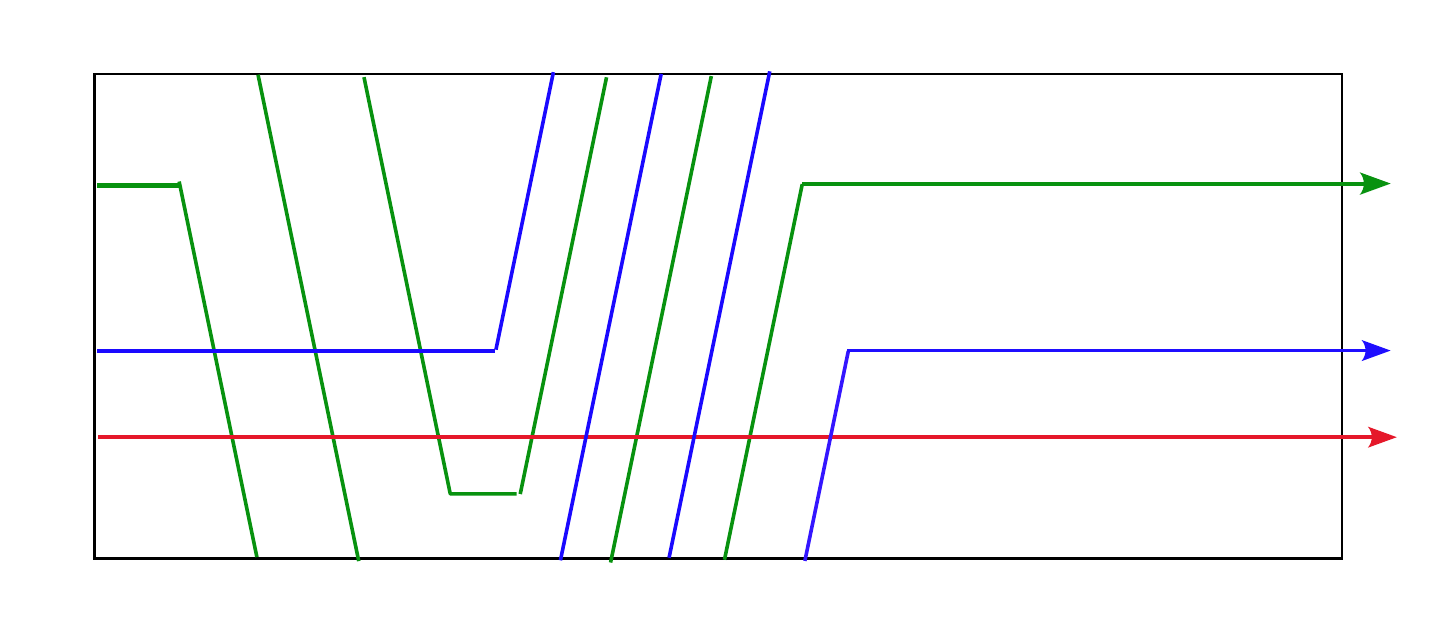}
\caption{The winding region of diagram $P_{n,\gamma}$, where $n=3$.}
\label{fig: green2}
\end{figure}
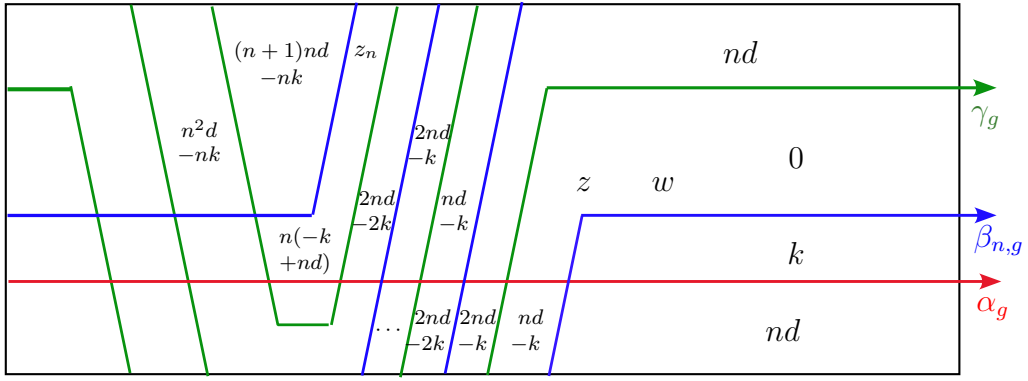

 The following computational result is one of our main ingredients. Compare \cite[Proposition 5.5]{HL}.
\begin{proposition} \label{prop: alex-triangle}
Let $\x \in \T_\alpha \cap \T_\beta$, $\q \in \T_\alpha \cap \T_\gamma$, and $\a \in \T_\alpha \cap \T_\delta$. Assume also $\a$ uses $p_l$, for some $l=b-m,\cdots,b-1$ in the winding region.
\begin{enumerate}
\item
For any $\psi \in \pi_2(\x, \Theta_{\beta\gamma}, \q)$,
\begin{align} \label{eq: abg-alex}
n\Al(\x) - \Al(\q) &= -nd n_w(\psi) - (k - nd) n_z(\psi) + k n_{z_n}(\psi) -\frac{nk-n^2 d}{2} \\
\label{eq: abg-c1-x} \gen{c_1(\spincs_z(\psi)), [P_\gamma]}
&= 2 \Al(\x) + 2d n_w(\psi) - 2d n_z(\psi) + k \\
\label{eq: abg-c1-q}\gen{c_1(\spincs_z(\psi)), [P_{n,\gamma}]}
&= 2 \Al(\q) + 2k n_{z_n}(\psi) - 2k n_w(\psi) + n^2 d
\end{align}

\item
For any $\psi \in \pi_2(\q, \Theta_{\gamma\delta}, \a)$,
\begin{align}
\label{eq: agd-alex}
\Al(\q) - \Al(\a) &=  -(k+ md) n_z(\psi) + (k+ md) n_{z_n}(\psi)  - md \sum_{j=1}^n \Big( n_{z_j}(\psi) - n_{u_j}(\psi) \Big) - \frac{nmd}{2} \\ 
  \nonumber \langle c_1(\spincs_z(\psi)), n[R]\rangle  &= \frac{m}{\nu} \Bigg( 2 \Al(\a)  +  2(k+md)\Big(n_{z_n}(\psi) - n_z(\psi)  -  \sum_{j=1}^n \left( n_{z_j}(\psi) - n_{u_j}(\psi) \right)\Big) \\ 
   \label{eq: agd-c1-a} &   - n(k + md) + n^2 d \Bigg) \\
\label{eq: agd-c1-q}
&= \frac{m}{\nu} \left( 2 \Al(\q) + 2k \sum_{j=1}^n \Big( n_{z_j}(\psi) - n_{u_j}(\psi) \Big)   -  nk  + n^2 d \right)
\end{align}

\item
For any $\psi \in \pi_2(\a, \Theta_{\delta\beta}, \x)$,
\begin{equation}  \label{eq: adb-alex}
\Al(\a) - n\Al(\x) =  (k+md) n_{z_n}(\psi) - nd n_w(\psi) - (k+md - nd) n_z(\psi)  + \frac{n(k+md) -  n^2 d}{2}
\end{equation}
\begin{align}
\label{eq: adb-c1-a}
\gen{c_1(\spincs_z(\psi)), [P_{n,\delta}]}
&= 2\Al(\a)  + 2(k+md)n_w(\psi)  - 2(k+md) n_{z_n}(\psi)  + n^2 d \\
\label{eq: adb-c1-x}
\gen{c_1(\spincs_z(\psi)), [P_\delta]}&= 2\Al(\x) + 2d n_z(\psi) - 2d n_w(\psi)  + (k+md)
\end{align}
\end{enumerate}
\end{proposition}

\begin{proof}
We will compute  part ($3$) first, namely the statements about $(\alpha,\delta,\beta)$ triangles. 

Up to permuting the indices of the $\beta$ curves, each $\x\in \T_\alpha \cap \T_\beta$ consists of points $\x_j \in \alpha_j \cap \beta_j$ for $j=1,\cdots,g,$ where $x_g$ is the unique point $q\in \alpha_g \cap \beta_g$. For $j=1,\cdots,g-1,$ suppose the local multiplicities of $P_\delta$ around $x_j$ are $c_j, c_j+a_j, c_j+a_j+b_j, c_j+b_j$ for some $c_j$. Hence, we compute
\begin{align*}
    &n_\x(P_\delta) = \frac{k+md+d}{2} + \sum_{j=1}^{g-1} \left(c_j + \frac{a_j + b_j}{2} \right) \\
    &n_{\bar w}(P_\delta) + n_{\bar z}(P_\delta) = k+md+d\\
    &n_{\bar w}(P_{n,\delta}) + n_{\bar z_{n}}(P_{n,\delta}) =-(n-1)(k+md)+n^2 d.
\end{align*}

We begin by showing (\ref{eq: adb-alex}) and (\ref{eq: adb-c1-x}) hold for $\psi_{\x,l}^{m,b} \in \pi_2(\x_l^{m,b}, \Theta_{\delta\beta}, \x)$, where $\psi_{\x,l}^{m,b}$ are the small triangles in winding region defined above.(We henceforth omit the superscripts for simplicity.) For $l=b-m,\cdots,b-1,$
\begin{align*}
    n_{\x_l}(P_{n,\delta})=\frac{k+md}{2} + n \sum_{j=1}^{g-1} \left(c_j + \frac{a_j + b_j}{2} \right) + &
    \begin{cases}
    -lnd & l\geq 0  \\
    -l(-(k+md)+nd) & -n<l<0  \\
    -n(k+md)-lnd & l\leq -n.
    \end{cases}
\end{align*}
To prove (\ref{eq: adb-alex}), we compute using (\ref{eq: abs-alex})
\begin{align*}
   (k+md)\AlNorm_{w,z_n}(\x_{l})=\frac{1}{2}\left(  n\hat\chi(P_\delta) + 2n_{\x_l}(P_{n,\delta}) - (n_{\bar z}(P_{n,\delta}) - n_{\bar w}(P_{n,\delta}))  \right)\\
   =nd\AlNorm_{z,w}(\x) +\frac{n(k+md)-n^2 d}{2} +  
   \begin{cases}
    -lnd & l\geq 0  \\
    -l(-(k+md)+nd) & -n<l<0  \\
    -n(k+md)-lnd & l\leq -n.
    \end{cases}
\end{align*}
Comparing with (\ref{eq: smalltriangle}), the last term on the right hand side is exactly  equal to 
\[
(k+md)n_{z_n}(\psi_{\x,l})-ndn_w(\psi_{\x,l})-(k+md-nd)n_z(\psi_{\x,l})
\]
as required.

For (\ref{eq: adb-c1-x}), we use the first Chern class formula from \cite[Proposition 6.3]{fourmanifold}.(There is a sign inconsistency in the definition of the dual spider number, see the footnote in \cite{HL} below the proof of Lemma 4.2.)
\begin{align*}
\gen{c_1(\spincs_z(\psi_{\x,l})), [P_\delta]} &= \hat\chi(P) + \#(\partial P_\delta) - 2n_z(P_\delta) + 2\sigma(\psi_{\x,l}, P_\delta) \\
&= \hat\chi(P_\delta) + \left(-(k+md) + \sum_{j=1}^{g-1} (a_j + b_j) \right) - 2 \cdot 0 + 2 \left( k+md-dl + \sum_{j=1}^{g-1} c_j \right) \\
&= \hat\chi(P_\delta)  + \sum_{j=1}^{g-1} (a_j + b_j + 2c_j) + 2d (-l) + k+md \\
&= \hat\chi(P_\delta) + 2 n_\x(P_\delta) - n_{\bar w}(P_\delta) - n_{\bar z}(P_\delta)  + 2d( n_z(\psi_{\x,l}) - n_w(\psi_{\x,l}) )+ k+md \\
&= 2d \AlNorm_{w,z}(\x) + 2d n_z(\psi_{\x,l}) - 2d n_w(\psi_{\x,l}) + k + md
\end{align*}
as required.

Next, we consider an aribitrary triangle $\phi \in \pi_2(\a,\Theta_{\delta\beta},\x).$ There are some $r\in \Z$ and some $l\in \{b-m,\cdots,b-1 \}$ such that
\[
n_w(\phi) - n_z(\phi)=r(k+md)+l.
\]
Let $\psi' = \psi - rP_\delta \in \pi_2(\x, \Theta_{\beta\gamma}, \q)$; then $n_w(\psi') - n_z(\psi') = l$. The composite domain $\phi$ with $\DD (\phi)= \DD(\psi') - \DD(\psi_{\x,l})$ is a disk in $\pi_2(\a, \x)$, so $\spincs_z(\psi') = \spincs_z(\psi_{\x,l})$.  We then compute:
\begin{align*}
\AlNorm_{w,z_n}(\a) - \AlNorm_{w,z_n}(\x_l) &= n_{z_n}(\phi) - n_w(\phi) \\
&= n_{z_n}(\psi') - n_{z_n}(\psi_{\x,l}) - n_w(\psi') + n_w(\psi_{\x,l}) \\
&= n_{z_n}(\psi) - n_w(\psi) -r(nd-(k+md)) + 
\begin{cases}
    l & l\geq 0  \\
    0 & -n<l<0  \\
    l+n & l\leq -n.
    \end{cases}
\end{align*}
\begin{align*}
& (k+md)\AlNorm_{w,z_n}(\a)- nd\AlNorm_{w,z}(\x) \\
&= ((k+md)\AlNorm_{w,z_n}(\x_l) - nd\AlNorm_{w,z}(\x) ) + (k+md)(\AlNorm_{w,z_n}(\a) - \AlNorm_{w,z_n}(\x_l)) \\
&= \frac{n(k+md)-n^2 d}{2}  +(k+md)(n_{z_n}(\psi) - n_w(\psi)) + (k+md-nd)(k+md)r + (k+md-nd)l \\
&= \frac{n(k+md)-n^2 d}{2} +(k+md)(n_{z_n}(\psi) - n_w(\psi)) +(k+md-nd)(n_w(\psi) - n_z(\psi)) 
\end{align*}
as required. Similarly, we have:
\begin{align*}
\gen{c_1(\spincs_z(\psi)), [P_\delta]}
&= \gen{c_1(\spincs_z(\psi' + rP_\delta)), [P_\delta]} \\
&= \gen{c_1(\spincs_z(\psi') + r \PD[P_\delta]), [P_\delta]} \\
&= \gen{c_1(\spincs_z)(\psi'), [P_\delta]} +   2r [P_\delta]^2 \\
&= \gen{c_1(\spincs_z)(\psi_{\x,l}), [P_\delta]} -   2r (k+md)d \\
&= 2d \AlNorm_{w,z}(\x) + k + md  - 2dl -   2r (k+md)d \\
&= 2d \AlNorm_{w,z}(\x) + k + md - 2d (n_w(\psi) - n_z(\psi)) 
\end{align*}
as required.

Then,  (\ref{eq: adb-c1-a}) follows immediately from (\ref{eq: adb-alex}) and (\ref{eq: adb-c1-x}). This concludes the proof of part  ($3$).

For  part ($1$), namely the statements about $(\alpha,\beta,\gamma)$ triangles can be proved in a similar manner and are left for the reader. The periodic diagram $P_{n,\gamma}$ is displayed in Figure \ref{fig: green2}.

Lastly, we prove part ($2$), namely the statements about $(\alpha,\gamma,\delta)$ triangles. Consider $\psi \in \pi_2(\q, \Theta_{\gamma\delta}, \a)$ for some $\q \in \T_\alpha \cap \T_\gamma$ and $\a \in \T_\alpha \cap \T_\delta$ . Note that $n_z(\psi) = n_w(\psi)$ since $w$ and $z$ are only separated by $\beta_g$.

Choose an arbitrary triangle $\phi \in \pi_2(\x, \Theta_{\beta\gamma}, \q)$ for some $\x \in \T_\alpha \cap \T_\beta$. By \eqref{eq: abg-alex},
\[
n\Al(\x) - \Al(\q) = -nd n_w(\phi) + k n_{z_n}(\phi) - (k-nd) n_z(\phi) - \frac{nk-n^2 d}{2}.
\]
For the simplification of notation, let $\tau = \tau_0^+ \in \pi_2(\Theta_{\beta\gamma}, \Theta_{\gamma\delta}, \Theta_{\beta\delta})$ be the class represented by the small triangle in the center of Figure \ref{fig: Q} (see Lemma \ref{lemma: Theta-gamma-delta}). The intersection number is zero for $\tau$ at all reference points. If we let $\sigma$ be the composite domain with $\DD(\sigma) = \DD(\phi) + \DD(\psi) - \DD(\tau)$, then $\sigma$ is almost the domain of a triangle in $\pi_2(\x, \Theta_{\beta\delta}, \a)$, except that the boundary of $\sigma$ includes $\gamma_g$ with multiplicity
\begin{align*}
r &= n_{z_j}(\sigma) - n_{u_j}(\sigma), \quad j=1,\cdots,n. 
\end{align*}
As a result
\begin{align*}
nr &= \sum^n_{j=1} \left(n_{z_j} (\sigma) - n_{u_j}(\sigma) \right)\\
&=\sum^n_{j=1} \left(n_{z_j} (\psi) - n_{u_j}(\psi) \right) + \sum^n_{j=1} \left(n_{z_j} (\phi) - n_{u_j}(\phi) \right) \\
&=\sum^n_{j=1} \left(n_{z_j} (\psi) - n_{u_j}(\psi) \right) +n_{z_n}(\phi) - n_z(\phi).
\end{align*}
There is an actual triangle class $\sigma' \in \pi_2(\x, \Theta_{\beta\delta}, \a)$ with $\DD(\sigma') = \DD(\sigma) - rQ$. Moreover, the composites $\rho_1 = \phi * \psi$ and $\rho_2 = \sigma' * \tau$ are each quadrilaterals in $\pi_2(\x, \Theta_{\beta\gamma}, \Theta_{\gamma\delta}, \a)$  satisfying $\DD(\rho_1) = \DD(\rho_2) + rQ$.

We compute
\begin{align*}
n\Al(\x) &- \Al(\a)
= -nd n_w(\sigma') + (k+md) n_{z_n}(\sigma') - (k+md -nd) n_z(\sigma') - \frac{n(k+md) -n^2 d}{2} \\
&= -nd n_w(\sigma) + (k+md) n_{z_n}(\sigma) - (k+md -nd) n_z(\sigma)  -mdnr - \frac{n(k+md) -n^2 d}{2} \\
&=  (k+md) n_{z_n}(\psi) - (k+md) n_z(\psi) -nd n_w(\phi) + (k+md) n_{z_n}(\phi)\\
&\qquad - (k+md -nd) n_z(\phi) -mdnr - \frac{n(k+md) -n^2 d}{2}\\
  & =(k+md)(n_{z_n}(\psi)-n_z(\psi)) -nd n_w(\phi) + (k+md) n_{z_n}(\phi) - (k+md -nd) n_z(\phi)  \\
  & \qquad -md \left( \sum^n_{j=1} \left(n_{z_j} (\psi) - n_{u_j}(\psi) \right) +n_{z_n}(\phi) - n_z(\phi)\right) - \frac{n(k+md) -n^2 d}{2} \\
  &= (k+md)(n_{z_n}(\psi)-n_z(\psi)) -md \sum^n_{j=1} \left(n_{z_j} (\psi) - n_{u_j}(\psi) \right)  \\
  &\qquad -nd n_w(\phi) + k n_{z_n}(\phi) - (k -nd) n_z(\phi) - \frac{n(k+md) -n^2 d}{2}
\end{align*}
The first equality above follows from the same computations in part ($3$), but using $(\alpha,\beta,\delta)$ triangles  instead. Subtracting the value of $n\Al(\x) - \Al(\q)$, we have:
\begin{align*}
\Al(\q) - \Al(\a)
&=  (k + md ) n_{z_n}(\psi)  - (k+md) n_z(\psi)  -md \sum^n_{j=1} \Big(n_{z_j} (\psi) - n_{u_j}(\psi)\Big) - \frac{nmd}{2}
\end{align*}
proving \eqref{eq: agd-alex}. Likewise, using \eqref{eq: abgd-PD-self-int}, we have:
\begin{align*}
&\quad \gen{ c_1(\spincs_z(\psi)), n[R]} \\
&= \gen{c_1(\spincs_z(\rho_1)), n[R]} \\
&= \gen{c_1(\spincs_z(\rho_2) + r\PD[Q]), n[R]} \\
&= \gen{c_1(\spincs_z(\rho_2)) + 2r\PD[Q], \frac{nm}{\nu} [P_\delta] + \frac{n(k+md)}{\nu}[Q]} \\
&= \frac{nm}{\nu} \gen{c_1(\spincs_z(\sigma')), \  [P_\delta]} + \frac{n(k+md)}{\nu} \gen{c_1(\spincs_z(\tau)), [Q]} + \frac{2nr(k+md)}{\nu}[Q]^2   \\
&= \frac{m}{\nu} \left( 2 \Al(\a) + 2(k+md)(n_{z_n}(\sigma') - n_z(\sigma')) + n^2 d\right) + \frac{n(k+md)}{\nu}(-m) +  \frac{2 nr(k+md)}{\nu}(-m)    \\
&= \frac{m}{\nu} \left( 2 \Al(\a) + 2(k+md)(n_{z_n}(\sigma') - n_z(\sigma') -nr) - n(k+md) + n^2 d \right)- n_z(\phi)   \\
&= \frac{m}{\nu} \left( 2 \Al(\a) + 2(k+md) \Big(n_{z_n}(\psi) - n_z(\psi) + n_{z_n}(\phi) -n_z(\phi) -nr \Big) + n(k+md) + n^2 d \right)  \\
&= \frac{m}{\nu} \left( 2 \Al(\a) + 2(k+md)\Big(n_{z_n}(\psi) - n_z(\psi) - \sum^n_{j=1} \left(n_{z_j} (\psi) - n_{u_j}(\psi) \right)\Big)   - n(k+md) + n^2 d  \right)
\end{align*}
thus proving \eqref{eq: agd-c1-a}. Finally, \eqref{eq: agd-c1-q} follows immediately from \eqref{eq: agd-alex} and \eqref{eq: agd-c1-a}.
\end{proof}

\begin{remark}
The first Chern class formulas in Proposition \ref{prop: alex-triangle} have their $\spincs_w$ version as well. For example, 
\begin{align*}
   & \gen{c_1(\spincs_w(\psi)), [P_\gamma]}\\
&=\gen{c_1(\spincs_z(\psi) - \PD[C]), [P_\gamma]}\\
&=\gen{c_1(\spincs_z(\psi) ) - 2\PD[C], [P_\gamma]}\\
&= 2 \Al(\x) + 2d n_w(\psi) - 2d n_z(\psi) - k, 
\end{align*}
which gives the $\spincs_w$ version of (\ref{eq: abg-c1-x}). The same computation applies to (\ref{eq: abg-c1-q}), (\ref{eq: adb-c1-a}) and (\ref{eq: adb-c1-x}) as well (using the core disk in each corresponding cobordism). For the $(\alpha,\delta,\gamma)$ triangles the two versions coincide due to the absence of the $\beta$ curves. 
\end{remark}

\begin{remark} \label{re: nprime}
Given a suitable diagram $(\Sigma, \bm\alpha, \bm\beta, \bm\gamma, \bm\delta, w,z,z_n)$ for a fixed $n,$ we can construct $P_{n',\gamma}$ and  $P_{n',\delta}$ for all the $n'$ with $1\leq n' \leq n.$ Simply by going through the same proof process, the results state for $n$ in the previous proposition holds (simultaneously) for all the $n'$ as well. Note that the Alexander gradings are dependent on $n$, even though it is not explicit in the notation. Adopting this point of view, in general the results stated for $n$ hold simultaneously for all  $n'$  with $1\leq n' \leq n$ for the rest of the paper, substituting $z_{n'} $ and $u_{n'}$ as appropriate. 

\end{remark}

\subsubsection{Spin$^c$ structures in cobordisms.} \label{sssec: spinc}
Proposition \ref{prop: alex-triangle} can help us understand the spin$^c$ structures in cobordisms. We focus on the cobordism $W_\lambda$, induced by $(\Sigma,\bm\alpha,\bm\beta,\bm\gamma).$ As before, let $K\subset Y$ be a knot of order $d>0$ in $H_1(Y;\Z)$ and $F$ a rational surface for $K$. Recall that $[\partial F]= d\lambda - k \mu$ in $H_1(\partial(Y-\text{nbh}(K))).$ Inside $W_\lambda$, let $C$ denote the core of the $2$-handle attached to $Y$, $C^*$  the cocore. Then $[C]$ (resp.~ $[C^*]$) generates $H_2(W,Y)$ (resp.~ $H_2(W,Y_\lambda)$). We abuse the notation and use $[C]$ and $[C^*]$ to denote the corresponding classes in $H_2(W)$ as well. Finally let $\hat F$ be the capped seifert surface in  $W_\lambda$, formed by capping a rational surface $F$ with $d$ parallel copies of $C$. Since $[\hat F]$ maps to $d[C]$ in $H_2(W,Y)$ and to $k[C^*]$ in $H_2(W, Y_\lambda(K))$, it follows that $[\hat F]^2=dk$.

In \cite[Section 2.2]{rational}, Ozsv\'ath and Szab\'o  defined the map
\[
G_{Y,K} \co \ul\Spin^c(Y,K) \to \Spin^c(Y),
\]
which is equivariant with respect to the restriction map
\[
H^2(Y,K;\Z) \to H^2(Y;\Z).
\]
The fibers of $G_{Y,K}$ are exactly the orbits of $\ul\Spin^c(Y,K)$ under the action of $\gen{\PD[\mu]} \subset H^2(Y,K;\Z)$. For any $\x \in \T_\alpha \cap \T_\beta,$  $G_{Y,K}(\ul\spincs_{w,z}(\x)) = \spincs_w(\x). $ They also construct a bijection
\[
E_{Y, \lambda, K} \co \Spin^c(W_\lambda) \to \ul\Spin^c(Y,K)
\]
characterized by the property that for $\psi \in \pi_2(\x, \Theta_{\beta\gamma}, \q)$,
\[
E_{Y,\lambda, K}(\spincs_w(\psi)) = \ul\spincs_{w,z}(\x) + (n_z(\psi) - n_w(\psi))\PD[\mu].
\]
As noted by Hedden and Levine, (\ref{eq: abg-c1-x}) in Proposition \ref{prop: alex-triangle} allows us to give an explicit and diagram-independent description of $E_{Y,\lambda, K}$ as follows
\begin{definition}\label{def: spincbij}
For any $\spincv \in \Spin^c(W_\lambda) $,   $E_{Y,\lambda, K}(\spincv)$ is the relative spin$^c$ structure satisfying
\[
G_{Y,K}(E_{Y,\lambda,K}(\spincv)) = \spincv|_Y \quad \text{and} \quad
\AlNorm_{Y,K}(E_{Y,\lambda, K}(\spincv)) = \frac{\gen{c_1(\spincv), [\hat F]} + k}{2d} .
\]
\end{definition}
We claim the two definitions coincide. Note that any relative spin$^c$ structure is determined by the pair $(G_{Y,K},\AlNorm_{Y,K})$. Given $\spincv \in \Spin^c(W_\lambda) $, suppose $\spincs_w(\psi)=\spincv$ for some $\psi \in \pi_2(\x, \Theta_{\beta\gamma}, \q)$. First we have  $G_{Y,K}(E_{Y,\lambda, K}(\spincv)) = G_{Y,K}(\ul\spincs_{w,z}(\x))$, since the action of $\PD[\mu]$ falls into the same orbit of $\Spin^c(Y)$.  Then according to   (\ref{eq: abg-c1-x}),
\begin{align*}
    \AlNorm_{Y,K}(E_{Y,\lambda, K}(\spincs_w(\psi)))&= \AlNorm(\x) + n_w(\psi) - n_z(\psi)\\
    &=\frac{\gen{c_1(\spincs_z(\psi)), [\hat F]} - k}{2d}\\
    &=\frac{\gen{c_1(\spincs_w(\psi))+2PD[C], [\hat F]} - k}{2d}\\
    & =\frac{\gen{c_1(\spincs_w(\psi)), [\hat F]} + k}{2d}\\
    & =\frac{\gen{c_1(\spincv), [\hat F]} + k}{2d}
\end{align*}
as required. 

Since $\ul\Spin^c(Y,K)$ and  $\ul\Spin^c(Y_\lambda,K_\lambda)$ (where $K_\lambda = K_{1,\lambda}$ is the dual knot) only depend on the knot complement, they are canonically identified. Therefore there is also a bijection between $\Spin^c(W_\lambda) $ and  $\ul\Spin^c(Y_\lambda,K_\lambda)$. Interestingly, Proposition \ref{prop: alex-triangle} further provides a bijection between $\Spin^c(W_\lambda) $ and  $\ul\Spin^c(Y_\lambda,K_{n,\lambda})$, and thus a bijection between $\ul\Spin^c(Y,K)$ and  $\ul\Spin^c(Y_\lambda,K_{n,\lambda})$, even though their knot complements are distinct.

\begin{definition} \label{def: Enbij}
For any $\spincv \in \Spin^c(W_\lambda) $,  suppose $\spincs_w(\psi)=\spincv$ for some $\psi \in \pi_2(\x, \Theta_{\beta\gamma}, \q)$.
Let 
\[
E_{Y_\lambda, K_{n,\lambda}} \co \Spin^c(W_\lambda) \to  \ul\Spin^c(Y_\lambda,K_{n,\lambda})
\]
be the map
 \[
E_{Y_\lambda, K_{n,\lambda}}(\spincs_z(\psi)) = \ul\spincs_{w,z_n}(\q) + (n_{z_n}(\psi) - n_w(\psi))\PD[K].
\]
Or equivalently, if $\spincs_z(\psi)=\spincv$,
 \[
\qquad E_{Y_\lambda, K_{n,\lambda}}(\spincs_w(\psi)) = \ul\spincs_{w,z_n}(\q) + (n_{z_n}(\psi) - n_w(\psi) - n)\PD[K].
\]
\end{definition}

And the following diagram-independent reinterpretation:

\begin{lemma}\label{le: spincbij}
For any $\spincv \in \Spin^c(W_\lambda) $, the map  $E_{Y_\lambda, K_{n,\lambda}} \co \Spin^c(W_\lambda) \to  \ul\Spin^c(Y_\lambda,K_{n,\lambda})$ is a bijection, characterized by
\[
G_{Y_\lambda,K_{n,\lambda}}(E_{Y_\lambda, K_{n,\lambda}}(\spincv)) = \spincv|_{Y_\lambda} \quad \text{and} \quad
\AlNorm_{Y_\lambda,K_{n,\lambda}}(E_{Y_\lambda, K_{n,\lambda}}(\spincv)) = \frac{\gen{c_1(\spincv), n[\hat F]} - n^2 d}{2k}.
\]
\\
Values of $\AlNorm_{Y_\lambda,K_{n,\lambda}}(E_{Y_\lambda, K_{n,\lambda}}(\spincv))$ for all the $\spincv$ with the same restriction in $\Spin^c(Y_\lambda(K))$  form a $\Q/n\Z$ coset.  
\end{lemma}

\begin{proof}
Fix $\spinct \in \Spin^c(Y_\lambda(K)),$ the  spin$^c$ structures in $\Spin^c(W_\lambda)$ that restricts to $\spinct$ form an orbit with the action of $\PD[C]$. Denote by $\spincv$ any such a spin$^c$ structure from this orbit.  Since
\[
\gen{c_1(\spincv + \PD[C]), n[\hat F]} = \gen{c_1(\spincv), n[\hat F]} + 2nk,
\]
the values of $\AlNorm_{Y_\lambda,K_{n,\lambda}}(E_{Y_\lambda, K_{n,\lambda}}(\spincv))$ in Lemma \ref{le: spincbij}
form a $\Q/n\Z$ coset. On the other hand, suppose $\spincs_z(\psi)=\spincv$ for some $\psi \in \pi_2(\x, \Theta_{\beta\gamma}, \q)$, then we have
\begin{align*}
    G_{Y_\lambda,K_{n,\lambda}}(E_{Y_\lambda, K_{n,\lambda}}(\spincs_z(\psi))) &= G_{Y_\lambda,K_{n,\lambda}}(\ul\spincs_{w,z_n}(\q))  = \spincv|_{Y_\lambda}\\
    \AlNorm_{Y_\lambda,K_{n,\lambda}}(E_{Y_\lambda, K_{n,\lambda}}(\spincs_z(\psi))) &= \AlNorm(\q) + n_{z_n}(\psi) - n_w(\psi)\\
    &=\frac{\gen{c_1(\spincv), n[\hat F]} - n^2 d}{2k}.
\end{align*}
as required. The last two equalities are according to Definition \ref{def: Enbij}, and (\ref{eq: abg-c1-q}) respectively. (Note that the image   $E_{Y_\lambda, K_{n,\lambda}}(\spincv)$ forms an orbit in $\ul\Spin^c(Y_\lambda,K_{n,\lambda})$ with the action of $\PD[K]$.)
\end{proof}

Lemma \ref{le: spincbij} (together with Definition \ref{def: spincbij})  concretely describe the bijection between $\ul\Spin^c(Y,K)$ and  $\ul\Spin^c(Y_\lambda,K_{n,\lambda})$. If we take $n=1$, it also recovers \cite[Corollary 4.5]{HL}. 

The spin$^c$ structures in $\Spin^c(W_\lambda)$ that has the same restriction in $\Spin^c(Y)$ form an orbit with the action given by $\PD[C^*]$. Their image under the bijection $E_{Y,\lambda,K}$ forms an orbit with the action given by  $\PD[\mu]$,  whose $\AlNorm_{Y,K}$ values form a $\Q/\Z$ coset. On the other hand, given $\spinct \in \Spin^c(Y_\lambda(K)),$ the spin$^c$ structures in $\Spin^c(W_\lambda)$ that restricts to $\spinct$ form an orbit with the action given by $\PD[C]$. Their image under the bijection $E_{Y,\lambda,K}$ forms an orbit with the action given by  $\PD[K]$, whose $\AlNorm_{Y,K}$ values have step length $k/d$ and $G_{Y,K}$ is of period $d$.

\subsubsection{Alexander grading shifts on holomorphic rectangles.}
Following Hedden-Levine's approach, we introduce a function $\MultComb$ to help with the computation. Note that the definition is adjusted to reflect the changes we made on the relative periodic domains.  For any domain $S$, define
\begin{align}\label{eq: mult-comb}
 \nonumber \MultComb(S) =   -nd n_w(S) - (k+md-nd) n_z(S) &+ (k+md) n_{z_n}(S)\\
&- md \sum_{j=1}^n\Big( n_{z_j}(S) - n_{u_j}(S) \Big).
\end{align}
Clearly, the definition of $\MultComb$ depends on $n$, even though we omit it in the notation.
Similar to the function defined by Hedden and Levine, for any multi-periodic domain $S$ (including those with nonzero multiplicity at $z$), we claim $\MultComb(S)=0$. To see this, first observe any domain is a linear combination of $P_\gamma$, $Q$, $\Sigma$, thin domains, and  $(\alpha,\beta)$ periodic domains with $n_z=0$. Then one can check that $\MultComb$ vanishes for each of these, proving the claim. Note that for different types of domains, the formula for $\MultComb(S)$ can be simplified  significantly, depending on which basepoints are in the same regions. We record the result as follows (compare with the table in \cite[Section 5.2]{HL}):
\begin{center}
  \begin{tabular}{|c|c|}
     \hline
     Type of domain & $\MultComb(S)$ \\ \hline
     $(\alpha,\beta), (\beta,\gamma)$ and $(\beta,\delta)$ & $nd(n_z(S) - n_w(S))$ \\
     $(\alpha,\gamma)$ & $k(n_{z_n}(S) - n_w(S))$ \\
     $(\alpha,\delta)$ & $(k+md)(n_{z_n}(S) - n_w(S))$ \\
     $(\gamma,\delta)$ & $-mdn(n_{z_n}(S) - n_u(S))$ \\
     $(\alpha,\beta,\gamma)$  & $-nd n_w(S) - (k-nd) n_z(S) + k n_{z_n}(S)$  \\
     $(\alpha,\gamma,\delta)$ & $- (k+md) n_z(S) + (k+md) n_{z_n}(S) - md \sum_{j=1}^n\Big( n_{z_j}(S) - n_{u_j}(S) \Big)$ \\
     $(\alpha,\delta,\beta)$ & $-nd n_w(S) - (k+md-nd) n_z(S) + (k+md) n_{z_n}(S) $  \\
     $(\beta,\gamma,\delta)$ & $-nd n_w(S) + nd n_z(S)   -mdn(n_{z_n}(S) - n_u(S))$  \\
     \hline
   \end{tabular}
\end{center}
Note that on $(\gamma,\delta)$ and $(\beta,\gamma,\delta)$ domains, due to the absence of $\alpha$ curves,  all $n_{u_j}$ takes the same value for $j=1,\cdots, n$, which we simply denote by $n_u$ in this context, and $n_{z_j}=n_z$($=n_w$, if further without $\beta$ curves). Now we are ready to compute the Alexander grading shifts on rectangles. Compare \cite[Proposition 5.6]{HL}.

\begin{proposition} \label{prop: alex-rectangle}
Let $\x \in \T_\alpha \cap \T_\beta$, $\q \in \T_\alpha \cap \T_\gamma$, and $\a \in \T_\alpha \cap \T_\delta$. Assume also $\a$ uses $p_l$, for some $l=b-m,\cdots,b-1$ in the winding region.
\begin{enumerate}
\item \label{it: rec1}
For any $\rho \in \pi_2(\x, \Theta_{\beta\gamma}, \Theta_{\gamma\delta}, \a)$,
\begin{align}
\label{eq: abgd-alex}
n\Al(\x) - \Al(\a) &= \MultComb(\rho)  -\frac{n(k+md) - n^2 d}{2} \\
\label{eq: abgd-c1-Pg}
\gen{c_1(\spincs_z(\rho)), [P_\gamma]} &= 2 \Al(\x) + 2d n_w(\rho) - 2d n_z(\rho) +k \\
\label{eq: abgd-c1-R}
\gen{c_1(\spincs_z(\rho)), n[R]}
&= \frac{m}{\nu}\Bigg( 2 \Al(\a)  +  2(k+md)\Big(n_{z_n}(\rho) - n_z(\rho)  -\sum_{j=1}^n\left( n_{z_j}(\rho) - n_{u_j}(\rho) \right)  \Big)\\
\nonumber &\qquad - n(k + md) + n^2 d \Bigg) \\
\label{eq: abgd-c1-Q}
\gen{c_1(\spincs_z(\rho)), n[Q]} &= -2m\sum_{j=1}^n\left( n_{z_j}(\rho) - n_{u_j}(\rho) \right) -nm.
\end{align}

\item \label{it: rec2}
For any $\rho \in \pi_2(\q, \Theta_{\gamma\delta}, \Theta_{\delta\beta}, \x)$,
\begin{align}
\label{eq: agdb-alex}
\Al(\q) - n\Al(\x) &= \MultComb(\rho) + \frac{nk- n^2d}{2} \\
\label{eq: agdb-c1-R}
\gen{c_1(\spincs_z(\rho)), n[R]} &= \frac{m}{\nu} \left( 2 \Al(\q) - 2k \sum_{j=1}^n\left( n_{z_j}(\rho) - n_{u_j}(\rho) \right)  - nk  + n^2 d \right) \\
\label{eq: agdb-c1-Pd}
\gen{c_1(\spincs_z(\rho)), [P_\delta]} &= 2\Al(\x) + 2d n_z(\rho) - 2d n_w(\rho)  + (k+md) \\
\label{eq: agdb-c1-Q}
\gen{c_1(\spincs_z(\rho)), n[Q]} &= 2m \left( -n_z(\rho) + n_{z_n}(\rho) -  \sum_{j=1}^n\left( n_{z_j}(\rho) - n_{u_j}(\rho) \right)  \right) -nm.
\end{align}

\item \label{it: rec3}
For any $\rho \in \pi_2(\a, \Theta_{\delta\beta}, \Theta_{\beta\gamma}, \q)$,
\begin{align}
\label{eq: adbg-alex}
\Al(\a) - \Al(\q) &=  \MultComb(\rho) - \frac{nmd}{2} \\
\label{eq: adbg-c1-Pd}
\gen{c_1(\spincs_z(\rho)), [P_{n,\delta}]}
&= 2\Al(\a)  +2(k+md) \left( n_z(\rho) - n_{z_n}(\rho) + \sum_{j=1}^n\left( n_{z_j}(\rho) - n_{u_j}(\rho) \right)   \right) + n^2 d \\
\label{eq: adbg-c1-Pg}
\gen{c_1(\spincs_z(\rho)), [P_{n,\gamma}]} &= 2 \Al(\q) + 2k \sum_{j=1}^n\left( n_{z_j}(\rho) - n_{u_j}(\rho) \right) + n^2 d \\
\label{eq: adbg-c1-Q}
\gen{c_1(\spincs_z(\rho)), [Q]} &= 2 n_w(\rho) - 2 n_z(\rho)  + m
\end{align}
\end{enumerate}
\end{proposition}

\begin{proof}
For part \eqref{it: rec1}, we consider $(\alpha, \beta, \gamma, \delta)$ rectangles.

For any $\x \in \T_\alpha \cap \T_\beta$, $\a \in \T_\alpha \cap \T_\delta$, and $\rho \in \pi_2(\x, \Theta_{\beta\gamma}, \Theta_{\gamma\delta}, \a)$, choose $\q \in \T_\alpha \cap \T_\gamma$, $\psi_1 \in \pi_2(\x, \Theta_{\beta\gamma}, \q)$, and $\psi_2 \in \pi_2(\q, \Theta_{\gamma\delta}, \a)$ such that $\spincs_z(\psi_1) = \spincs_z(\rho)|_{X_{\alpha\beta\gamma}}$ and $\spincs_z(\psi_2) = \spincs_z(\rho)|_{X_{\alpha\gamma\delta}}$. Moreover, by adding copies of $\Sigma$, which does not change the spin$^c$ structure condition, we may assume that $n_z(\rho) = n_z(\psi_1) + n_z(\psi_2)$. Hence, $S = \DD(\rho) - \DD(\psi_1 * \psi_2)$ is a quadruply periodic domain with $n_z(S)=0$. Since the function $\MultComb$ vanishes on all periodic domains, we have:
\begin{align*}
\Al(\x) - \Al(\a) &= (\Al(\x) - \Al(\q)) + (\Al(\q) - \Al(\a)) \\
&= \MultComb(\psi_1) - \frac{nk-n^2 d}{2} + \MultComb(\psi_2) - \frac{nmd}{2}
&= \MultComb(\rho) - \frac{n(k+md)-n^2 d}{2}
\end{align*}
which proves \eqref{eq: abgd-alex}.

Next, we consider the spin$^c$ evaluations. Up to thin domains, we have $S = x P_\gamma + y R$, where the decomposition is chosen considering fact that classes $[P_\gamma]$ and $[R]$ can be represented by disjoint surfaces in $X_{\alpha\beta\gamma\delta}$, so $[P_\gamma] \cdot [R] = 0$ in the intersection form on $X_{\alpha\beta\gamma\delta}$. We need to solve for $x$ and $y.$ Observe $n_w(R)=n_z(R)=0,$ from which we solve
\[
x =  \frac{n_w(S)}{k}.
\]
Similarly we have $-n_{z_n}(P_\gamma) + \sum_{j=1}^n\Big( n_{z_j}(P_\gamma) - n_{u_j}(P_\gamma) \Big)=0,$ so
\[
 y =  -\frac{\nu}{nk}\left(-n_{z_n}(S)+\sum_{j=1}^n\Big( n_{z_j}(S) - n_{u_j}(S) \Big)\right).
\]
Note that $x$ and $y$ need not be integers.  Using \eqref{eq: abgd-PD-self-int}, \eqref{eq: abgd-int-form}, and \eqref{eq: abg-c1-x}, we compute:
\begin{align*}
\gen{c_1(\spincs_z(\rho)), [P_\gamma]}
&= \gen{c_1(\spincs_z(\psi_1 * \psi_2) + \PD[S]), [P_\gamma]} \\
&= \gen{c_1(\spincs_z(\psi_1 * \psi_2)) + 2\PD[S], [P_\gamma]} \\
&= \gen{c_1(\spincs_z(\psi_1)), [P_\gamma]} + 2x [P_\gamma]^2 + 2y[R] \cdot [P_\gamma] \\
&= 2 \Al(\x) + 2d(n_w(\psi_1)-n_z(\psi_1)) + k  + 2dk\cdot \frac{ n_w(S)}{k} \\
&= 2 \Al(\x) + 2d(n_w(\psi_1)-n_z(\psi_1)) + k + 2d( n_w(\rho) - n_w(\psi_1) - n_w(\psi_2)) \\
&= 2 \Al(\x) + 2d(n_w(\rho) - n_z(\psi_1) -n_w(\psi_2)) +k  \\
&= 2 \Al(\x) + 2d(n_w(\rho) - n_z(\psi_1) -n_z(\psi_2)) +k  \\
&= 2 \Al(\x) + 2d(n_w(\rho) -n_z(\rho)) +k,
\end{align*}
which proves \eqref{eq: abgd-c1-Pg}. (Note the similarity with \eqref{eq: abg-c1-x}.) Formula \eqref{eq: abgd-c1-R} follows from a similar computation using  \eqref{eq: agd-c1-a}, as shown below. 
\begin{align*}
&\gen{c_1(\spincs_z(\rho)), n[R]}\\
&= \gen{c_1(\spincs_z(\psi_1 * \psi_2) + \PD[S]), [R]} \\
&= \gen{c_1(\spincs_z(\psi_2)), n[R]} + 2x n[P_\gamma] \cdot [R] + 2yn[R]^2 \\
&= \frac{m}{\nu}\Bigg( 2 \Al(\a)  +  2(k+md)\Big(n_{z_n}(\psi_2) - n_z(\psi_2)  -  \sum_{j=1}^n \left( n_{z_j}(\psi_2) - n_{u_j}(\psi_2) \right)\Big) - n(k + md) + n^2 d \Bigg) \\
&-\frac{2m(k+md)}{\nu}\left(-n_{z_n}(S)+\sum_{j=1}^n\Big( n_{z_j}(S) - n_{u_j}(S) \Big)\right) \\
&=  \frac{m}{\nu}\Bigg( 2 \Al(\a)  +  2(k+md)\Big(n_{z_n}(\psi_2) - n_z(\psi_2)  -  \sum_{j=1}^n \left( n_{z_j}(\psi_2) - n_{u_j}(\psi_2) \right)  +n_{z_n}(S) \\
&  -\sum_{j=1}^n\left( n_{z_j}(S) - n_{u_j}(S) \right)\Big) - n(k + md) + n^2 d \Bigg)\\
&=  \frac{m}{\nu}\Bigg( 2 \Al(\a)  +  2(k+md)\Big(n_{z_n}(\psi_2) - n_z(\psi_2) + n_{z_n}(\psi_1) - n_z(\psi_1)  + n_{z_n}(\rho) - n_{z_n}(\psi_1) \\
& - n_{z_n}(\psi_2)-\sum_{j=1}^n\left( n_{z_j}(\rho) - n_{u_j}(\rho) \right)  \Big)- n(k + md) + n^2 d \Bigg)\\
&=\frac{m}{\nu}\Bigg( 2 \Al(\a)  +  2(k+md)\Big(n_{z_n}(\rho) - n_z(\rho)  -\sum_{j=1}^n\left( n_{z_j}(\rho) - n_{u_j}(\rho) \right)  \Big)- n(k + md) + n^2 d \Bigg)
\end{align*}

Finally, to prove \eqref{eq: abgd-c1-Q}, we have:
\begin{align*}
&\gen{c_1(\spincs_z(\rho)), n[Q]} \\
&= \frac{\nu}{k} \gen{c_1(\spincs_z(\rho)), n[R]} - \frac{m}{k}\gen{c_1(\spincs_z(\rho)), n[P_\gamma]} \\
&= \frac{m}{k} \Bigg( 2 \Al(\a)  +  2(k+md)\Big(n_{z_n}(\rho) - n_z(\rho)  -\sum_{j=1}^n\left( n_{z_j}(\rho) - n_{u_j}(\rho) \right)  \Big) \\
& \qquad - n(k + md) + n^2 d \Bigg) - \frac{m}{k} \left( 2 n\Al(\x) + 2nd\left(n_w(\rho) -n_z(\rho)\right) +nk \right) \\
&=\frac{m}{k} \Bigg( -2\MultComb(\rho)  +  2(k+md)\Big(n_{z_n}(\rho) - n_z(\rho)  -\sum_{j=1}^n\left( n_{z_j}(\rho) - n_{u_j}(\rho) \right)  \Big)\\
& \qquad  - 2nd\left(n_w(\rho) -n_z(\rho)\right) -nk \Bigg) \\
&=-2m\sum_{j=1}^n\left( n_{z_j}(\rho) - n_{u_j}(\rho) \right) -nm
\end{align*}
as required. This concludes the proof for part \eqref{it: rec1}.

For part \eqref{it: rec2}, namely the case about $(\alpha,\gamma,\delta,\beta)$ rectangles,  consider surfaces inside $X_{\alpha\gamma\delta\beta}$. We use the decomposition $S = x R + y P_\delta,$ where
\[
x=\frac{\nu}{n(k+md)}\sum_{j=1}^n\left( n_{z_j}(S) - n_{u_j}(S) \right) \quad \text{and} \quad y=\frac{n_w(S)}{k+md}.\quad \qquad \qquad 
\]

 Similarly, part \eqref{it: rec3}  is the case about $(\alpha,\delta,\beta,\gamma)$ rectangles, where surfaces are inside $X_{\alpha\delta\beta\gamma}$. We  use the decomposition $S = x P_\delta + y P_\gamma,$ where
\[
x=\frac{1}{nd}\left(n_{z_n}(S)- \sum_{j=1}^n\left( n_{z_j}(S) - n_{u_j}(S) \right) \right) \quad \text{and} \quad y=\frac{1}{nd}\sum_{j=1}^n\left( n_{z_j}(S) - n_{u_j}(S) \right).
\]
The rest of the computations for both of these cases are parallel to part \eqref{it: rec1} and left for the reader.
 \end{proof}

\section{The filtered large surgery formula} \label{sec: largesurgery}
In this section, we prove the filtered large surgery formula for the $(n,1)$--cable of the knot meridian. The proof mostly follows the framework in \cite[Section 5.5]{HL}, with minor adjustments.

As before, let $K\subset Y$ be a knot of order $d>0$ in $H_1(Y;\Z)$ and $F$ a rational surface for $K$. Recall that $W_{\lambda+m\mu}$ denotes the cobordism from $Y$ to $Y_{\lambda+m\mu}(K)$, induced by $(\Sigma,\bm\alpha,\bm\beta,\bm\delta),$ and turning the cobordism around, $W'_{\lambda+m\mu}$ is a cobordism from $Y_{\lambda+m\mu}(K)$ to $Y$, induced by $(\Sigma,\bm\alpha,\bm\delta,\bm\beta).$ For the ease of notation, we write $W_m=W_{\lambda+m\mu}$ and  $W'_m=W'_{\lambda+m\mu}.$ We abuse the notation, and will denote by $C$  the core of the $2$-handle and $\hat F$ the capped seifert surface in the cobordism $W'_m$ as well.

Given $\spincu \in \Spin^c(Y_{\lambda+m\mu}(K))$, let $\spincv$ denote a spin$^c$ structure on $W'_m$ that extends $\spincu$, we have
\[
\gen{c_1(\spincv + \PD[C]), [\hat F]}  = \gen{c_1(\spincv), [\hat F]} + 2(k+md),
\]
so the values of $\gen{c_1(\spincv), [\hat F]}$ taken over all such $\spincv$ would form a coset in $\Z/2(k+md)$. We recall the following definition from \cite{HL}.

\begin{definition}[Definition 5.7 in \cite{HL}] \label{def: xu}
For each $\spincu \in \Spin^c(Y_{\lambda+m\mu}(K))$, let $\spincx_\spincu$ denote the unique spin$^c$ structure on $W'_m$ extending $\spincu$ such that
\begin{equation} \label{eq: c1(xu)}
-2(k+md) < \gen{c_1(\spincx_\spincu), [\hat F]} \le 0,
\end{equation}
 Let $\spincy_\spincu = \spincx_\spincu + \PD[C]$, so that
\begin{equation} \label{eq: c1(yu)}
0 < \gen{c_1(\spincy_\spincu), [\hat F]} \le 2(k+md).
\end{equation}
and let $\spincs_\spincu = \spincx_\spincu|_Y = \spincy_\spincu|_Y - \PD[K]$.
Define
\begin{align}\label{eq: su-def}
s_\spincu =& \frac{1}{2d}(\gen{c_1(\spincx_\spincu), [\hat F]} + k + md)\\
=&\frac{1}{2d}(\gen{c_1(\spincy_\spincu), [\hat F]} - k - md),
\end{align}
so that
\begin{equation} \label{eq: su-bound}
-\frac{k+md}{2d} < s_\spincu \le \frac{k+md}{2d}.
\end{equation}
Finally, define
\begin{equation} \label{eq: Delta-u-def}
\Delta_\spincu = \absgr(F^\infty_{W'_m, \spincx_u}) = -\frac{(2ds_\spincu-k-md)^2}{4d(k+md)} + \frac14.
\end{equation}
\end{definition}

We define a pair of filtrations $\II_\spincu, \JJ_\spincu$ on $\CFKi(Y,K,\spincs_\spincu)$ by the formula
\begin{align}
\label{eq: Iu-def} \II_\spincu([\x,i,j]) &= \max\{i, j-s_\spincu\} \\
\label{eq: Ju-def} \JJ_\spincu([\x,i,j]) &= \max\{i-n, j-s_\spincu\} + \frac{2nd(s_\spincu - n) + k+ md}{2(k+md)} .
\end{align}

The following theorem strengthens the main result of \cite{Truong}, as it computes the (absolute) Alexander grading of generators in $\CFKi(Y_{\lambda+m\mu}(K), K_{n,\lambda + m\mu})$, not only the filtration levels, and  the knot $K$ could be in any rational homology sphere, not only in $S^3$.  
Compare \cite[Theorem 5.8]{HL}.
\begin{theorem} \label{thm: large-surgery}
If $m$ is sufficiently large, then for every $\spincu \in \Spin^c(Y_{\lambda+m\mu}(K))$, there is a doubly-filtered quasi-isomorphism whose grading shift equal to $\Delta_\spincu$
\[
\Lambda^{\infty}_\spincu \co \CFKi(Y_{\lambda+m\mu}(K), K_{n,\lambda + m\mu}, \spincu) \to \CFKi(Y,K,\spincs_\spincu),
\]
where the latter is equipped with the filtrations $\II_\spincu$ and $\JJ_\spincu$, making the diagrams
\[
\xymatrix{
\CFKi(Y_{\lambda+m\mu}(K), K_{n,\lambda + m\mu}, \spincu) \ar[r]^-{\Lambda^{\infty}_\spincu} \ar[d]^{=} & \CFKi(Y,K,\spincs_\spincu) \ar[d]^{v^\infty} \\
\CFi(Y_{\lambda+m\mu}(K), \spincu) \ar[r]^-{F^\infty_{W'_m, \spincx_\spincu}} & \CFi(Y,\spincs_\spincu)
}
\]
and
\[
\xymatrix{
\CFKi(Y_{\lambda+m\mu}(K), K_{n,\lambda + m\mu}, \spincu) \ar[r]^-{\Lambda^{\infty}_\spincu} \ar[d]^{=} & \CFKi(Y,K,\spincs_\spincu)  \ar[d]^{h^\infty_{\spincs, s_\spincu}} \\
\CFi(Y_{\lambda+m\mu}(K), \spincu) \ar[r]^-{F^\infty_{W'_m, \spincy_\spincu}} & \CFi(Y,\spincs_\spincu+\PD[K])
}
\]
commute up to chain homotopy.
\end{theorem}
We make some remarks about the theorem statements before moving on to the proof.

\begin{remark} First, since the filtered chain homotopy type  of the knot Floer complexes is a topological invariant of the pair  $(Y,K)$, we drop the basepoints from the statements. Note that even though the definition of the map $\Lambda^{\infty}_\spincu$(spelled out in the proof later in the section) depends on the basepoints, the map itself does not. Next, observe that the statements here are almost the same as those in \cite[Theorem 5.8]{HL}, except that $ K_{n,\lambda + m\mu}$ is in place for $K_{\lambda + m\mu}$. Indeed, disregarding the second filtration, the map $\Lambda^{\infty}_\spincu$ here is identical to the one in \cite[Theorem 5.8]{HL}. The only difference is that  $\CFKi(Y_{\lambda+m\mu}(K), K_{n,\lambda + m\mu}, \spincu)$ is a  refiltering of $\CFKi(Y_{\lambda+m\mu}(K), K_{\lambda + m\mu}, \spincu)$, due to the different placement of a second basepoint. Reflecting this change, the $\JJ_\spincu$ filtration on $\CFKi(Y,K,\spincs_\spincu)$ is adjusted accordingly. Finally, the Maslov grading shift is induced by cobordism maps  $F^\infty_{W'_m, \spincx_\spincu}$ and $F^\infty_{W'_m, \spincy_\spincu}$, respectively, independent of the choice of a second basepoint. Therefore the statement about the Maslov grading follows directly from Hedden-Levine's proof.
\end{remark}

The bulk of Theorem \ref{thm: large-surgery} is proved by Ozsv\'ath and Szab\'o, we only need to show $\Lambda^{\infty}_\spincu$ preserves the second filtration that is defined above by (\ref{eq: Ju-def}).

We focus on the Heegaard diagram $(\Sigma, \bm\alpha, \bm\beta, \bm\delta^{m,b}, w, z, z_n)$, see Figure \ref{fig: twistpurple} for an example. Recall from the discussion in Section \ref{ssec: domains} that  $q$ is the unique intersecting point of $\alpha_g$ and $\beta_g$ and $p_l$ for $l=b-m,\cdots,b-1$ are the intersecting points of  $\alpha_g$ and $\delta_g$, following the orientation of $\delta_g$. For $\x \in \T_\alpha\cap \T_\beta$ and $l \in \{b-m, \dots, b-1\}$, we define $\x_l^{m,b} \in \Gens(\bm\alpha, \bm\delta^{m,b})$ to be the point obtained by replacing $q$ with $p_l$ and taking ``nearest points'' in thin domains. There is a \emph{small triangle} $\psi_{\x,l}^{m,b} \in \pi_2(\x_l^{m,b}, \Theta_{\delta\beta}, \x)$ in the winding region satisfying
\begin{align*}
    (n_w(\psi_{\x,l}),n_z(\psi_{\x,l}),n_{z_n}(\psi_{\x,l}))=&
    \begin{cases}
    (l,0,0) & l\geq 0  \\
    (0,-l,0) & -n<l<0  \\
    (0,-l,-l-n) & l\leq -n.
    \end{cases}\\
\end{align*}

A key ingredient of Ozsv\'ath and Szab\'o's large surgery theorem is the notion that every spin$^c$ structure can be represented by generators in the winding region. Hedden and Levine defined the following stronger version:

\begin{definition}[Definition 5.10 in \cite{HL}]
We say $\spincu \in \Spin^c(Y_{\lambda+m\mu}(K))$ is \emph{strongly supported in the winding region} if
\begin{itemize}
    \item every $\a \in \T_\alpha \cap \T_\delta$ with $\spincs_w(\a)=\spincu$ is of the form $\x_l^{m,b}$ for some $l$;
    \item  moreover $c_1(\spincs_w(\psi^{m,b}_{\x,l})) = \spincx_\spincu$, or equivalently, $c_1(\spincs_z(\psi^{m,b}_{\x,l})) = \spincy_\spincu$.
\end{itemize} 
\end{definition}

 Recall that $\spincs_w(\a)=\spincs_z(\a)$ since in the diagram $(\Sigma,\bm\alpha,\bm\delta)$ basepoints $w$ and $z$ are interchangeable. However, $\spincs_w(\psi)$ and $\spincs_z(\psi)$ differ by $PD[C]$  ; $\spincs_w(\x)$ and $\spincs_z(\x)$ differ by $\PD[K]$. Hedden and Levine proved that this improved condition can be satisfied.

\begin{lemma}[Lemma 5.12 in \cite{HL}]\label{lemma: strongly-supported}
There exists an $M$ such that for all $m\geq M$, for every  $\spincu \in \Spin^c(Y_{\lambda+m\mu}(K))$, there exists some $b$ such that $\spincu$ is strongly supported in the winding region of $(\Sigma, \bm\alpha, \bm\beta, \bm\delta^{m,b}, w, z, z_n)$. (Note that $b$ does depend on the choice of $\spincu$.)
\end{lemma}
As pointed out in \cite[Lemma 5.11]{HL}, a spin$^c$ structure $\spincu$ is strongly supported in the winding region of $(\Sigma, \bm\alpha,
\bm\delta)$ iff
\begin{equation*} \label{eq: strongly-supported}
\Gens(\bm\alpha, \bm\delta, w, \spincu) = \{\x_{\AlNorm(\x)-s_\spincu}  \mid \x \in
\Gens(\bm\alpha, \bm\beta, w, \spincs_\spincu)\},
\end{equation*}

where $\Gens(\bm\alpha, \bm\beta, w, \spincs) = \{\x \in \T_\alpha\cap \T_\beta \mid \spincs_w(\x) = \spincs\}.$ This can be seen from (\ref{eq: adb-c1-x}) combined with the definition of $s_\spincu$ by (\ref{eq: su-def}).

\begin{proof}[Proof of Theorem \ref{thm: large-surgery}]
When $m$ is large enough, for a given choice of $\spincu \in \Spin^c(Y_{\lambda+m\mu}(K))$, fix some $b$ to satisfy Lemma {\ref{lemma: strongly-supported}}.  

Define 
\begin{equation} \label{eq: Lambda-u}
\Lambda^\infty_\spincu \co \CFi(\Sigma, \bm\alpha, \bm\delta, w, \spincu) \to \CFKi(\Sigma, \bm\alpha, \bm\beta, w, z, \spincs_\spincu)
\end{equation}
by
\begin{equation} \label{eq: Lambda-u-def}
\Lambda^\infty_\spincu([\a,i]) =
\sum_{\substack{\x \in \T_\alpha \cap \T_\beta \\ \spincs_w(\x) = \spincs_\spincu }}
\sum_{\substack{\psi \in \pi_2(\a, \Theta_{\delta\beta}, \x) \\ \mu(\psi) = 0 \\ \spincs_w(\psi) = \spincx_\spincu}}
\#\MM(\psi) [\x, i-n_w(\psi), i-n_z(\psi) + s_\spincu].
\end{equation}
Precomposing with the identification of $\CFi(\Sigma, \bm\alpha, \bm\delta, w, \spincu)$ with $\CFKi(\Sigma, \bm\alpha, \bm\delta, w,z_n, \spincu)$,  $\Lambda^\infty_\spincu$ can also be seen as defined on $\CFKi(\Sigma, \bm\alpha, \bm\delta, w,z_n, \spincu)$. We will prove $\Lambda^\infty_\spincu$ is filtered with respect to $\JJ_\spincu$ filtration. The proof for $\II_\spincu $ filtration is easier and left for the reader. 

Since $\spincu$ is strongly supported in the winding region, every element of $\Gens(\bm\alpha, \bm\delta,w, \spincu)$ is of the form $\x_l$, where $\x \in \Gens(\bm\alpha,\bm\beta, w, \spincs_\spincu)$ and $l = \AlNorm(\x) - s_\spincu$. 
As in the proof of Proposition \ref{prop: alex-triangle}, we compute

\begin{align*}
    \Al(\x_l)=&\Al(\x) +\frac{n(k+md)-n^2d}{2}+
    \begin{cases}
    -lnd & l\geq 0\\
    l(k+md-nd) & -n<l<0 \\
    -n(k+md)-lnd & l\leq -n.
    \end{cases}\\
\end{align*}

So with $  \Al(\x) - lnd = nds_\spincu$, we have
\begin{align*}
    \AlNorm_{w,z_n}([\x_l,i])=i+&\frac{nd(2s_\spincu-n)}{2(k+md)}+\frac{n}{2}+
    \begin{cases}
    0 & l\geq 0\\
    l & -n<l<0 \\
    -n & l\leq -n.
    \end{cases}\\
\end{align*}

On the other hand, we first consider the small triangles $\psi_{\x,l} \in \pi_2(\x_l,\Theta_{\delta\beta},\x)$ in the winding region. According to Ozsv\'ath and Szab\'o, with respect to an energy filtration, $\psi_{\x,l}$ correspond to the main part of $\Lambda_\spincu ^\infty $.
\begin{align*}
    \JJ_{\spincu}([\x,i-n_w(\psi_{\x,l}),i-n_z(\psi_{\x,l})+s_\spincu])=&\text{ max}\{ -n_w(\psi_{\x,l}) -n, -n_z(\psi_{\x,l}) \} + i + \frac{nd(2s_\spincu - n)}{2(k+md)} + \frac{n}{2}\\
    =i + \frac{nd(2s_\spincu - n)}{2(k+md)} +& \frac{n}{2}+
    \begin{cases}
    0 & l\geq 0\\
    l & -n<l<0 \\
    -n & l\leq -n.
    \end{cases}
\end{align*}
Therefore the small triangles $\psi_{\x,l}$ preserves the $\JJ_\spincu$ filtration. Next, for an arbitrary triangle $\psi \in \pi_2(\x_l,\Theta_{\delta\beta},\x)$, $n_w(\psi) \geq n_w(\psi_{\x,l})$ and $n_z(\psi) \geq n_z(\psi_{\x,l})$. Thus $\psi$ decreases or preserves $\JJ_\spincu$ from the above computation.  

\end{proof}

\section{The surgery exact triangle} \label{sec: exacttri}

In this section, we will construct the surgery exact triangle relating the Floer homologies of $Y$, $Y_\lambda(K)$, and $Y_{\lambda+m\mu}(K)$ for $m$ large. The maps will be defined on the chain level. We start with a brief discussion about the Maslov grading.

For a diagram $(\Sigma, \bm\alpha, \bm\beta, \bm\gamma, \bm\delta^{b,m}, w, z, z_n)$, it is proved in \cite[Proposition 5.15]{HL} that for fixed large $m$, and $b$ within a small range of $\frac{m}{2}$, the Maslov grading of every generator $\a \in \T_\alpha \cap \T_{\delta^{m,b}}$ has constant lower and upper bound. We will adopt a choice of $m$ and $b$ satisfying the above condition, and henceforth drop them from the notation. Note that the Maslov grading is independent of the choice of $z_n$ basepoint, therefore the statements in \cite{HL} regarding the Maslov grading remain true in our set up as well, and we will restate them when needed. 

\subsection{Construction of the exact sequence}

We start by pointing out a slight difference in our set up compared to the construction in \cite{HL}. In this section we will define our complexes $\CF^\circ (\Sigma, \bm\alpha, \bm\beta, z)$, $\CF^\circ (\Sigma, \bm\alpha, \bm\gamma, z)$ and $\CF^\circ (\Sigma, \bm\alpha, \bm\delta, z)$ with the basepoint $z$ instead of $w$ as in \cite{HL}. And accordingly, we define the triangle, rectangle and pentagon counting maps with respect to the reference point $z$ instead of $w$. This in turn, coincides with the definition in \cite{integer,rational}. The reason why we need to make this modification, loosely speaking, is that we need to capture the information of the windings to the left of the $\beta_g$ curve in the winding region. Aside from this slight change, the proof largely follows from the one in \cite[Section 6]{HL}, with the only difference being computational details.   

The twisted complex $\CFit(\bm\alpha, \bm\beta, z; \GR)$ is generated over $\GR$ by all pairs $[\x,i]$ as usual, where $\GR$ denote the group ring $\F[\Z/m\Z]$. The differential is given by 
\begin{equation}
\partial(T^s \cdot [\x,i]) = \sum_{\y \in \T_\alpha \cap \T_\beta} \sum_{\substack{ \phi \in \pi_2(\x,\y) \\ \mu(\phi)=1}} \# \widehat\MM(\phi) \, T^{s+ n_w(\phi) - n_z(\phi)} [\y, i-n_z(\phi)].
\end{equation}

 Comparing to previous definitions, the $T$ exponent in \cite{integer,rational} was formulated  as the intersection number of the holomorphic disk with an extra basepoint on $\beta_g$, but this quantity is the same as $n_w(\phi) - n_z(\phi)$.  Comparing to the definition in \cite{HL}, we merely switched the reference point from $w$ to $z$.

We realize $\GR = \F[\Z/m\Z]$ as the quotient ring $\F[T]/(T^m-1)$, and often view it as a subring of $\F[\Q/m\Z]$. Recall that $\F[\Q/m\Z]$ is the ring of rational-exponent polynomials with variable $T$, where the coefficient is in $\F$. 

The complex $\CFit(\bm\alpha, \bm\beta, z; \GR)$ is isomorphic to $m$ copies of $\CFi(\Sigma, \bm\alpha, \bm\beta, z)$.
Let
\begin{equation} \label{eq: theta}
\theta\co \CFpt(\bm\alpha, \bm\beta, z; \GR) \to
\bigoplus_{\spincs \in \Spin^c(Y)} \CFp(\bm\alpha, \bm\beta, \spincs, z) \otimes T^{ - k/d -\AlNorm(\spincs)} \GR
\end{equation}
be the trivializing map defined by
\begin{equation} \label{eq: theta-def}
\theta(T^s [\x,i]) = [\x,i] \otimes T^{s - k/d - \AlNorm(\x) }.
\end{equation}
One can check that this gives an isomorphism between the two chain complexes.\footnote{This trivializing map differs from the one in \cite{HL} by a constant factor. This change amounts to shifting the sequence $s_l$ (defined in Section \ref{ssec: statement}) by one to the left, such that the terms in the sequence $(\spincs_l,s_l)$ in the mapping cone (see the proof in Section \ref{sec: proofcone}) have the same index. Otherwise the index is off by one, due to the choice of $z$ reference point instead of $w$.} 

 For each $\spincs \in \Spin^c(Y)$, recall that $\AlNorm(\spincs)$ forms a $\Q/Z$ coset. So there are $m$ different powers of $T$ occurring in $\CFp(\bm\alpha, \bm\beta, \spincs, z) \otimes T^{ -k/d-\AlNorm(\spincs)} \GR$, with exponents in $\Q/m\Z$. We regularly need to lift these exponents to $\Q$,  by choosing the $m$ values of $r$ satisfying
\begin{equation} \label{eq: r-bounds}
r \equiv -\frac{k}{d} -\AlNorm_{Y,K}(\spincs) \pmod \Z \quad \text{and} \quad \frac{-k-md}{2d} \le r < \frac{-k+md}{2d}.
 \end{equation}

Define chain maps
\begin{align}
\label{eq: f0}
f_0^+\co & \CFpt(\bm\alpha, \bm\beta, z; \GR) \to \CFp(\bm\alpha, \bm\gamma, z) \\
\label{eq: f1}
f_1^+ \co & \CFp(\bm\alpha, \bm\gamma, z) \to \CFp(\bm\alpha, \bm\delta, z) \\
\label{eq: f2}
f_2^+\co & \CFp(\bm\alpha, \bm\delta, z) \to \CFpt(\bm\alpha, \bm\beta, z; \GR)
\end{align}
by the following formulas:
\begin{align}
\label{eq: f0-def}
f_0^+(T^s \cdot [\x,i]) &= \sum_{\q \in \T_\alpha \cap \T_\gamma}  \sum_{\substack{\psi \in \pi_2(\x, \Theta_{\beta\gamma}, \q) \\ \mu(\psi)=0 \\ \mathclap{s + n_w(\psi) - n_z(\psi) \equiv 0 \pmod m}}}   \#\MM(\psi) \, [\q, i-n_z(\psi)] \\
\label{eq: f1-def}
f_1^+([\q,i]) &= \sum_{\a \in \T_\alpha \cap \T_\delta} \sum_{\substack{\psi \in \pi_2(\q, \Theta_{\gamma\delta}, \a) \\ \mu(\psi)=0}} \#\MM(\psi) \, [\a, i-n_z(\psi)] \\
\label{eq: f2-def}
f_2^+([\a,i]) &= \sum_{\x \in \T_\alpha \cap \T_\beta} \sum_{\substack{\psi \in \pi_2(\a, \Theta_{\delta\beta}, \x) \\ \mu(\psi)=0}} \#\MM(\psi) \, T^{n_w(\psi) - n_z(\psi)} \cdot [\x, i-n_z(\psi)].
\end{align}
Let $f_0^t$, $f_1^t$, $f_2^t$ denote the analogous maps on $\CF^t$. 

Following \cite{HeddenMarkFractional}, the quadrilateral-counting maps
\begin{align}
\label{eq: h0}
h_0^+\co & \CFpt(\bm\alpha, \bm\beta, z; \GR) \to \CFp(\bm\alpha, \bm\delta, z) \\
\label{eq: h1}
h_1^+ \co & \CFp(\bm\alpha, \bm\gamma, z) \to \CFpt(\bm\alpha, \bm\beta, z; \GR) \\
\label{eq: h2}
h_2^+\co & \CFp(\bm\alpha, \bm\delta, z) \to \CFp(\bm\alpha, \bm\gamma, z)
\end{align}
are defined by the following formulas:
\begin{align}
\label{eq: h0-def}
h_0^+(T^s \cdot [\x,i]) &= \sum_{\a \in \T_\alpha \cap \T_\delta}  \sum_{\substack{\rho \in \pi_2(\x, \Theta_{\beta\gamma}, \Theta_{\gamma\delta}, \a) \\ \mu(\rho)=-1 \\ \mathclap{s + n_w(\rho) - n_z(\rho) \equiv 0 \pmod m}}}   \#\MM(\rho) \, [\a, i-n_z(\rho)] \\
\label{eq: h1-def}
h_1^+([\q,i]) &= \sum_{\x \in \T_\alpha \cap \T_{\beta}} \sum_{\substack{\rho \in \pi_2(\q, \Theta_{\gamma\delta}, \Theta_{\delta\beta}, \x) \\ \mu(\rho)=-1}} \#\MM(\rho) \, T^{n_w(\rho) - n_z(\rho)} \cdot [\x, i-n_z(\rho)] \\
\label{eq: h2-def}
h_2^+([\a,i]) &= \sum_{\q \in \T_\alpha \cap \T_\gamma} \sum_{\substack{\rho \in \pi_2(\a, \Theta_{\delta\beta}, \Theta_{\beta\gamma}, \q) \\ \mu(\rho)=-1 \\ n_w(\rho) \equiv n_z(\rho) \pmod m}} \#\MM(\rho) \, [\q, i-n_z(\rho)]
\end{align}

A standard argument shows that for each $j \in \Z/3$, the following holds. (The second statement relies on pentagon-counting maps, which we will discuss in Section \ref{ssec: pent-filt}.)
\begin{itemize}
\item $h_j^+$ is a null-homotopy of $f^+_{j+1} \circ f^+_j$;
\item $h_{j+1}^+ \circ f_j^+ + f_{j+2}^+ \circ h_j^+$ is a quasi-isomorphism.
\end{itemize}

 Therefore, the exact triangle detection lemma \cite[Lemma 4.2]{branched} implies an exact sequence on homology. Using the same formulas, one can also define the maps on $\CF^t$, $\CFmc$, and $\CFic$.

Each of the three complexes come with an $\II$ filtration, and the maps $f_j$ and $h_j$ respect the $\II$ filtration by definition. We define a second filtration on each complex as follows.(Compare \cite[Definition 6.1]{HL}.)
\begin{definition} \label{def: J-filtrations}
\begin{itemize}
\item
The filtration $\JJ_{\alpha\gamma}$ on $\CFp(\bm\alpha, \bm\gamma, z)$ is simply the Alexander filtration:
\begin{equation} \label{eq: ag-filt}
\JJ_{\alpha\gamma}([\q,i]) = \AlNorm_{w,z_n}(\q) + i.
\end{equation}

\item
The filtration $\JJ_{\alpha\delta}$ on $\CFp(\bm\alpha, \bm\delta, z)$ is the Alexander filtration shifted by a constant on each spin$^c$ summand. For any spin$^c$ structure $\spincu$, and for any generator $\a$ with $\spincs_z(\a) = \spincu$, define 
\begin{equation} \label{eq: ad-filt}
\JJ_{\alpha\delta}([\a,i]) = \AlNorm_{w,z_n}(\a) + i + \frac{  nd^2 m (2 s_\spincu - n) }{2k(k+md)},
\end{equation}
where $s_\spincu$ is the number from Definition \ref{def: xu}.

\item
The filtration $\JJ_{\alpha\beta}$ on the twisted complex $\CFpt(\bm\alpha, \bm\beta, z; \GR)$ is defined via the trivialization map $\theta$. For any $\x \in \T_\alpha \cap \T_\beta$ with $\spincs_z(\x) = \spincs$, and any $r$ that satisfies \eqref{eq: r-bounds}, define
\begin{equation} \label{eq: ab-twisted-filt}
\JJ_{\alpha\beta}([\x,i] \otimes T^r) = i - \frac{2ndr +nk+n^2 d}{2k}.
\end{equation}
Namely, $\JJ_{\alpha\beta}$ is the trivial filtration shifted by a constant
which depends linearly on the exponent of $T$, and it does not depend on $\x$ except via its associated spin$^c$ structure. We transport this back to $\CFpt(\bm\alpha, \bm\beta, z; \GR)$ via the identification $\theta$.
\end{itemize}
\end{definition}

As is discussed in \cite{HL}, the maps $f^+_j$ are not filtered, due to the fact that the filtration shift of each map is a function of spin$^c$ evaluation. In order to get round this problem, note that it suffices to consider the maps $f^t_j$ to prove quasi-isomorphism. Recall that if we let $X_*$ denote one of the $4$-manifolds $X_{\alpha\beta\gamma}, X_{\alpha\gamma\delta}$ and $X_{\alpha\delta\beta}$ defined in Section \ref{ssec: cob} (either with three or four subscripts),  $\Spin^c_0(X_*)$ denotes the set of spin$^c$ structures that restricts to the canonical spin$^c$ structure represented by the generator $\Theta_{\beta\gamma}$ on $Y_{\beta\gamma}$, $\Theta_{\gamma\delta}$ on $Y_{\gamma\delta}$ and $\Theta_{\delta\beta}$ on $Y_{\delta\beta}$, for whichever applicable.  Each map $f^t_j$ decomposes over spin$^c$ structures as 
\[
f_j^t = \sum_{\spincv \in \Spin^c_0(X_*)} f^t_{j,\spincv},
\]
where  $X_*$ is the $4$-manifold corresponding to the cobordism and $f^t_{j,\spincv}$ counts only the triangles with $\spincs_z(\psi)=\spincv$. The Maslov grading shift of each $f^t_{j,\spincv}$ is given by a quadratic function of $c_1(\spincv)$, whereas for a fixed $t$, the grading of each generator in $\CF^t$ is bounded. As a result, only finitely many terms $f^t_{j,\spincv}$ may be nonzero, allowing us to have control over the $\JJ$ filtration shift within this range. We will prove the following proposition. Compare \cite[Proposition 6.2]{HL}.

\begin{proposition} \label{prop: tri-filt}
Fix $t \in \N$. For all $m$ sufficiently large, the maps $f_0^t$, $f_1^t$, and $f_2^t$ are all filtered with respect to the filtrations $\JJ_{\alpha\beta}$, $\JJ_{\alpha\gamma}$, and $\JJ_{\alpha\delta}$. Moreover, for any triangle $\psi$ contributing to any of these maps, the filtration shift of the corresponding term equals $n_{z_n}(\psi)$.
\end{proposition}

The similar problem applies to the rectangle-counting maps. Following Hedden-Levine's argument, we will define truncated versions $\tilde h^t_j$, as a sum of certain terms $h^t_{j,\spincv}$ satisfying specific constraints. In other words, we simply throw away the ``bad" terms. The resulting $\tilde h^t_j$ will be homotopic equivalent to the original $h^t_j$ maps but behave nicer with respect to the filtration. We will prove the following proposition, parallel to \cite[Proposition 6.3]{HL}.

\begin{proposition} \label{prop: rect-filt}
Fix $t \in \N$. For all $m$ sufficiently large, the maps $\tilde h^t_0$, $\tilde h^t_1$, and $\tilde h^t_2$ have the following properties:
\begin{itemize}
\item $\tilde h^t_j$ is a filtered null-homotopy of $f^t_{j+1} \circ f^t_j$.
\item $\tilde h^t_{j+1} \circ f^t_j +  f^t_{j+2} \circ \tilde h^t_j$ is a filtered quasi-isomorphism.
\end{itemize}
\end{proposition}

For the Maslov grading, observe that the underlying cobordisms  remain the same, independent of the choice of $z_n$ basepoint. It then follows from Hedden-Levine's argument that maps $f_j^t$ and $h^t_j$ are homogeneous and have the appropriate Maslov grading shifts (for the detailed statement see \cite[Proposition 6.5]{HL}). Combined with Proposition \ref{prop: rect-filt} and Proposition \ref{prop: tri-filt}, we have

\begin{theorem} \label{thm: CFt-cone-f2}
Fix $t \in \N$. For all $m$ sufficiently large, the map
\[
\begin{pmatrix} f_1^t \\ h_1^t \end{pmatrix}  \co \CF^t(\Sigma, \bm\alpha, \bm\gamma, z) \to \Cone(f_2^t)
\]
is a filtered homotopy equivalence that preserves the grading.
\end{theorem}

\subsection{Triangle maps} \label{ssec: tri-filt}

We will prove Proposition \ref{prop: tri-filt} in this section by looking at $f_0^t$, $f_1^t$, and $f_2^t$ individually. (Throughout, we will write $f_j^\circ$ when making statements that apply all the flavors of Heegaard Floer homology.) The proof follows the outline in \cite{HL} closely.

\subsubsection{The map $f_0^t$} \label{sssec: f0-filt}

For any $\x \in \T_\alpha \cap \T_\beta$ with $\spincs_z(\x) = \spincv|_Y$, and any $r \in \Q/m\Z$, we have:
\begin{align*}
f_{0,\spincv}^\circ(\theta^{-1}([\x,i] \otimes T^r))
&= f_{0,\spincv}^\circ (T^{r +\AlNorm(\x)} \cdot [\x,i])  \\
&= \sum_{\q \in \T_\alpha \cap \T_\gamma} \ \sum_{\substack{\psi \in \pi_2(\x, \Theta_{\beta\gamma}, \q) \\ \mu(\psi)=0 \\ \spincs_z(\psi) = \spincv \\ \mathclap{r + k/d + \AlNorm(\x) + n_w(\psi) - n_z(\psi) \equiv 0 \pmod m} }} \#\MM(\psi) \, [\q, i-n_z(\psi)].
\end{align*}
According to \eqref{eq: abg-c1-x},
\[
 \frac{k}{d} + \AlNorm(\x) + n_w(\psi) - n_z(\psi) = \frac{\gen{c_1(\spincv), [P_\gamma]}+k}{2d}.
\]
Thus, $f_{0,\spincv}^\circ \circ \theta^{-1}$ is nonzero only on the summand
\[
\CF(\Sigma, \bm\alpha, \bm\beta, \spincs, z) \otimes T^r \GR,
\]
where $\spincs = \spincv |_Y$ and
\[
r \equiv - \frac{1}{2d} (\gen{c_1(\spincv), [P_\gamma]} + k)  \pmod m.
\]
On this summand, if we neglect the power of $T$, the composition equals the untwisted cobordism map $F^\circ_{W_\lambda(K),\spincv}$. We also lift $r$ to $\Q$ with the constraint
\[
\frac{-k-md}{2d} \le r < \frac{-k+md}{2d}.
\]
In \cite[Lemma 6.7]{HL} it is proved that for fixed $t\in \N$ and all large enough  $m$, if $f_{0,\spincv}^t \neq 0$ over spin$^c$ structure $\spincv,$ then for any $\epsilon >0,$
\[
\abs{ \gen{c_1(\spincv), [P_\gamma]} } < \epsilon md.
\]
In particular, if we take $\epsilon < 1,$ this implies if $f_{0,\spincv}^\circ \circ \theta^{-1}$ is nonzero, then
\begin{equation} \label{eq: f0-c1-r}
\gen{c_1(\spincv), [P_\gamma]} = - 2dr - k.
\end{equation}
\begin{proposition}[Proposition 6.8 in \cite{HL}] \label{prop: f0-filt}
Fix $t \in \N$. For all $m$ sufficiently large, the map
\[
f_0^t \co \ul\CF^t(\bm\alpha, \bm\beta, z; \GR) \to \CF^t(\bm\alpha, \bm\gamma, z)
\]
is filtered with respect to the filtrations $\JJ_{\alpha\beta}$ and $\JJ_{\alpha\gamma}$.
\end{proposition}
\begin{proof}
Assume $m$ is large enough, and let $\spincv$ be a spin$^c$ structure for which $f^t_{0,\spincv} \ne 0$,  which satisfies \eqref{eq: f0-c1-r}.

For any $\x \in \T_\alpha \cap \T_\beta$ and $\q \in \T_\alpha\cap \T_\gamma$ with $\spincs_z(\x) =  \spincv|_{Y_{\alpha\beta}}$ and $\spincs_z(\q) = \spincv|_{Y_{\alpha\gamma}}$, and any triangle $\psi \in \pi_2(\x, \Theta_{\beta\gamma}, \q)$ contributing to $f^t_{0,\spincv}$, we compute
\begin{align*}
\JJ_{\alpha\beta} ([\x,i] &\otimes T^r) - \JJ_{\alpha\gamma}([\q,i-n_z(\psi)])   \\
&= - \frac{2ndr + nk +n^2 d}{2k} - \frac{\Al(\q)}{k}  + n_z(\psi)   \\
&= -\frac{2ndr + nk + n^2 d}{2k} - \frac{n\gen{c_1(\spincv), [P_\gamma]} - 2kn_{z_{n}}(\psi) + 2k n_z(\psi) - n^2 d}{2k} + n_z(\psi) \\
&= n_{z_{n}}(\psi) \\
&\ge 0
\end{align*}
as required. The last equality uses \eqref{eq: f0-c1-r}, and \eqref{eq: abg-c1-q}.
\end{proof}

\subsubsection{The map $f_1^t$} \label{sssec: f1-filt}The map $f_1^\circ$ decomposes as a sum
\begin{equation} \label{eq: f1-decomp}
f_1^\circ = \sum_{\spincv \in \Spin^c_0(X_{\alpha\gamma\delta})} f_{1,\spincv}^\circ.
\end{equation}
 We need the following lemma from \cite{HL}:
\begin{lemma}[Lemma 6.9 in \cite{HL}]\label{le: f1range}
Fix $t\in \N$, when $m$ is sufficiently large, if $f_{1,\spincv}^\circ\neq 0,$ then 
\begin{equation} \label{eq: f1-trunc}
\abs{ \gen{c_1(\spincv), [R]} } < \frac{m(k+md)}{\nu}.
\end{equation}
Moreover, for any $\spincu \in \Spin^c(Y_{\lambda+m\mu}(K))$, there is at most one nonzero term landing in $\CF^t(\Sigma, \bm\alpha, \bm\delta, \spincu, z)$, which satisfies
 \begin{equation} \label{eq: f1-trunc-su}
\gen{c_1(\spincv), [R]} = \frac{2 d m s_\spincu}{\nu} .
\end{equation}
\end{lemma}
 The second statement follows from the the range of $s_\spincu$ and the fact that  $\gen{c_1(\spincv), [R]}$ has step length $ 2m(k+md)/\nu$.  For each $\spincu \in \Spin^c(Y_{\lambda+m\mu}(K)),$ recall that $s_\spincu$ is the number appeared in Definition \ref{def: xu}.
\begin{proposition}[Proposition 6.10 in \cite{HL}] \label{prop: f1-filt}
For $m$ sufficiently large, the map
\[
f_1^t \co \CF^t(\bm\alpha, \bm\gamma, z) \to \CF^t(\bm\alpha, \bm\delta, z)
\]
is filtered with respect to the filtrations $\JJ_{\alpha\gamma}$ and $\JJ_{\alpha\delta}$.
\end{proposition}

\begin{proof}
    Assume $m$ is large enough, and suppose $f^t_{1,\spincv}$ is the only nonzero term landing in $\CF^t(\Sigma, \bm\alpha, \bm\delta, \spincu, z)$, where $\spincu = \spincv|_{Y_{\alpha\delta}}$. For any $\q$ with $\spincs_z(\q)=\spincv|_{Y_{\alpha\gamma}}$ and  $\a$ with $\spincs_z(\a)=\spincu$, and any   triangle $\psi \in \pi_2(\q, \Theta_{\gamma\delta}, \a)$ contributing to $f^t_{1,\spincv}$, compute
\begin{align*}
\AlNorm(\q) &- \AlNorm(\a) \\
&= \frac{\Al(\q)}{k} - \frac{\Al(\a)}{k+md}   \\
&= \frac{k(\Al(\q) - \Al(\a)) + md \Al(\q)}{k(k+md)} \\
&= \frac{1}{k+md} \left( (k+md) n_{z_n}(\psi)  - (k+md) n_z(\psi) -md \sum_{j=1}^n \Big( n_{z_j}(\psi) - n_{u_j}(\psi) \Big)- \frac{nmd}{2} \right) \\
 & \qquad + \frac{md}{k(k+md)} \left( \frac{n\nu}{2m} \gen{c_1(\spincv), [R]} + k \sum_{j=1}^n \Big( n_{z_j}(\psi) - n_{u_j}(\psi) \Big) + \frac{nk}{2} -  \frac{n^2 d}{2} \right) \\
&= n_{z_n}(\psi) - n_z(\psi)  - \frac{2nmd}{2(k+md)} + \frac{md}{k(k+md)} \left( nd s_\spincu  + \frac{nk}{2} -  \frac{n^2 d}{2} \right)\\
&= n_{z_n}(\psi) - n_z(\psi) + \frac{nd^2m ( 2s_\spincu -  n )}{2k(k+md)}.
\end{align*}
Thus,
\[
\JJ_{\alpha\gamma}([\q,i]) - \JJ_{\alpha\delta}([\a, i-n_z(\psi)]) = n_{z_n}(\psi) \ge 0
\]
as required. 
\end{proof}

\subsubsection{The map $f_2^t$} \label{sssec: f2-filt}

Let us first examine how the spin$^c$ decomposition of $f_2^\circ$ interacts with the trivializing map $\theta$. For any $\a \in \T_\alpha \cap \T_\delta$,  we have:
\begin{align*}
\theta \circ f_2^\circ([\a,i])
&= \sum_{\x \in \T_\alpha \cap \T_\beta} \sum_{\substack{\psi \in \pi_2(\a, \Theta_{\delta\beta}, \x) \\ \mu(\psi)=0}} \#\MM(\psi) \, [\x, i-n_z(\psi)] \otimes T^{-k/d + n_w(\psi) - n_z(\psi)  - \AlNorm_{w,z}(\x)}.
\end{align*}
By \eqref{eq: adb-c1-x}, the $T$ exponent is equal to
\[
-\frac{k}{d} + n_w(\psi) - n_z(\psi)  - \AlNorm_{w,z}(\x) = - \frac{\gen{c_1(\spincv), [P_\delta]} +k-md}{2d}.
\]
Therefore, the term $\theta \circ f_{2,\spincv}^\circ$ lands in the summand
\begin{equation} \label{eq: f2-t-target}
\CF^\circ(\bm\alpha, \bm\beta, \spincv|_Y, z) \otimes T^r,
\end{equation}
where $r \in \Q/m\Z$ is given by
\[
r \equiv - \frac{\gen{c_1(\spincv), [P_\delta]} +k+md}{2d}    \pmod {m}
\]

Neglecting the power of $T$, the composition map equals the untwisted cobordism map $F^\circ_{W'_m, \spincv}$ in the first factor. Fixing $t\in\N$ and $\epsilon>0$, according to \cite[Lemma 6.11]{HL}, for all $m$ sufficiently large, if $f_{2,\spincv}^t \ne 0$, then
\begin{equation} \label{eq: f2-trunc}
\abs{ \gen{c_1(\spincv), [P_\delta]} } < (1+\epsilon)(k+md).
\end{equation}
In particular, assuming that $(1+\epsilon)(k+md) < 2md$, then the only spin$^c$ structures that may contribute to $f^t_2$ are those denoted by $\spincx_\spincu$ and $\spincy_\spincu$ in Definition \ref{def: xu}.

\begin{proposition}[Proposition 6.12 in \cite{HL}] \label{prop: f2-filt}
Fix $t \in \N$. For all $m$ sufficiently large, the map
\[
f_2^t \co \CF^t(\bm\alpha, \bm\delta, z) \to \ul\CF^t(\bm\alpha, \bm\beta, z; \GR)
\]
is filtered with respect to the filtrations $\JJ_{\alpha\delta}$ and $\JJ_{\alpha\beta}$.
\end{proposition}

\begin{proof}
Assume  $m$ is sufficiently large  and   $(1+\epsilon)(k+md) < 2md$. Suppose $\spincv$ is a spin$^c$ structure on $W'_m$ for which $f^t_{2, \spincv} \ne 0$, and let $\spincu = \spincv|_{Y_{\lambda+m\mu}(K)}$ and $\spincs = \spincv|_Y$. By \eqref{eq: f2-trunc}, we have
\[
-2md < -(1+\epsilon)(k+md) < \gen{c_1(\spincv), [P_\delta]} < (1+\epsilon)(k+md) < 2md.
\]

Let $r$ denote the rational number satisfying
\[
-k-md \le 2dr < -k+md \quad \text{and} \quad 2dr \equiv -(\gen{c_1(\spincv), [P_\delta]}+k+md) \pmod {2md}.
\]
Note that $r$ is one of the exponents appearing in \eqref{eq: r-bounds}. By \eqref{eq: f2-t-target}, $f^t_{2,\spincv}$ lands in $\CF^\circ(\bm\alpha, \bm\beta, \spincv|_Y, z) \otimes T^r$. At the same time, by \eqref{eq: su-bound} and \eqref{eq: adb-c1-x}, the number $s_\spincu$ satisfies:
\[
-k-md < 2ds_\spincu \le k+md \quad \text{and} \quad 2ds_\spincu \equiv \gen{c_1(\spincv), [P_\delta]} + k+md \pmod {2(k+md)}.
\]

There are two possible cases to consider. 
\begin{enumerate}[label=(\roman*)]

    \item \label{it:c1f2}
    If $-(1+\epsilon)(k+md) < \gen{c_1(\spincv), [P_\delta]} \le 0$, in this case $\spincv=\spincx_\spincu,$ and the above inequalities and congruences imply that
\begin{equation} \label{eq: f2-filt-c1-neg}
\gen{c_1(\spincv), [P_\delta]} = -2dr-k-md = 2ds_\spincu - k-md.
\end{equation}
\item \label{it:c2f2}
Otherwise if $0 < \gen{c_1(\spincv), [P_\delta]} < (1+\epsilon)(k+md)$, in this case $\spincv=\spincy_\spincu,$ and we obtain
\begin{equation} \label{eq: f2-filt-c1-pos}
\gen{c_1(\spincv), [P_\delta]} = -2dr-k+md = 2ds_\spincu + k+md.
\end{equation}
\end{enumerate}

Suppose $\psi \in \pi_2(\a, \Theta_{\delta\beta}, \x)$ is any triangle that contributes to $f^t_{2,\spincv}$,
so that $\theta( f^t_{2,\spincv}([\a,i]))$ includes the term $[\x,i-n_z(\psi)] \otimes T^r$. We compute:
\begin{align*}
\JJ_{\alpha\delta}([\a,i]) &- \JJ_{\alpha\beta} ([\x,i-n_z(\psi)]  \otimes T^r)   \\
&= \AlNorm_{w,z_n}(\a) + \frac{2ndr + nk + n^2 d}{2k} +  \frac{nd^2m (2 s_\spincu-n)}{2k(k+md)} + n_z(\psi)  \\
&= \frac{n\gen{c_1(\spincv), [P_\delta]} - 2(k+md) n_z(\psi) + 2(k+md) n_{z_n}(\psi) -n^2 d}{2(k+md)}  \\
& \qquad + \frac{2ndr + nk + n^2 d}{2k} +  \frac{nd^2m (2 s_\spincu-n)}{2k(k+md)} + n_z(\psi)  \\
&= \frac{n\gen{c_1(\spincv), [P_\delta]} -n^2 d }{2(k+md)}  + \frac{2ndr + nk + n^2 d}{2k} +  \frac{nd^2m (2 s_\spincu-n)}{2k(k+md)} + n_{z_n}(\psi).
\end{align*}
Depending on case \ref{it:c1f2} or \ref{it:c2f2},  we will use either \eqref{eq: f2-filt-c1-neg} or \eqref{eq: f2-filt-c1-pos} to express the terms in first two fractions with $s_\spincu$, which gives us:
\begin{align*}
\JJ_{\alpha\delta}([\a,i]) &- \JJ_{\alpha\beta} ([\x,i-n_z(\psi)]  \otimes T^r)   \\
&= \frac{2nds_\spincu -n^2 d }{2(k+md)}  - \frac{2nds_\spincu  - n^2 d}{2k} +  \frac{nd^2m (2 s_\spincu-n)}{2k(k+md)} + n_{z_n}(\psi) \\
&= n_{z_n}(\psi) \ge 0
\end{align*}
as required.
\end{proof}

\subsection{Rectangle maps} \label{ssec: rect-filt}
In this section we analyze rectangle-counting maps. Following Hedden-Levine's argument, we will introduce the truncated maps $\tilde h^t_0$, $\tilde h^t_1$, and $\tilde h^t_2$, and prove the first part of Proposition \ref{prop: rect-filt} with these maps. The proof follows completely from the recipe in \cite{HL}, replacing numerical values of Alexander filtration and spin$^c$ evaluation as appropriate. We write down the whole process for the completeness.

\subsubsection{The map $h_0^t$} \label{sssec: h0-filt} 
Similar to the triangle-counting maps, $h^\circ_0$ splits over spin$^c$ structures $\spincv \in \Spin^c_0(X_{\alpha\beta\gamma\delta})$. The composition map $h_{0,\spincv}^\circ \circ \theta^{-1}$ is nonzero only on the summand $\CF^\circ(\bm\alpha, \bm\beta, \spincs, z) \otimes T^r \GR$, where $\spincs = \spincv |_Y$ and
\[
r \equiv - \frac{1}{2d} (\gen{c_1(\spincv), [P_\gamma]} + k)  \pmod m,
\]
and on this summand  $h_{0,\spincv}^\circ \circ \theta^{-1}$ is equal to an untwisted count of rectangles. Recall the general strategy for the rectangle-counting maps is to throw away bad terms, and to prove the remaining terms still constitute a null-homotopy of the triangle-counting maps. The following definition and lemma are due to Hedden and Levine:

\begin{definition} [Definition 6.14 in \cite{HL}] \label{def: h0-trunc}
For given $\epsilon>0$, let $\tilde h_0^t$ be the sum of all terms $h^t_{0,\spincv}$ corresponding to spin$^c$ structures $\spincv$ which either satisfy both
\begin{align}
\label{eq: h0-filt-Pg-bound} \abs{\gen{c_1(\spincv), [P_\gamma]}} &< \epsilon md \\
\label{eq: h0-filt-R-bound} \abs{ \gen{c_1(\spincv), [R]}} &< \frac{m(k+md)}{\nu}
\end{align}
or satisfy
\begin{equation}
\label{eq: h0-filt-Q-bound} \abs{ \gen{c_1(\spincv), [Q]}} = \pm m.
\end{equation}
\end{definition}

\begin{lemma}[Lemma 6.15 in \cite{HL}] \label{lemma: h0-trunc-nulhtpy}
Fix $t \in \N$ and $\epsilon>0$. For all $m$ sufficiently large, $\tilde h^t_0$ is a null-homotopy of $f^t_1 \circ f^t_0$.
\end{lemma}

We require an extra restraint on the spin$^c$ evaluation. The following lemma is given by an analysis on the absolute grading, which applies to our case as well.

\begin{lemma} [Lemma 6.13 in \cite{HL}]\label{lemma: h0-trunc}
Fix $t \in \N$ and $\epsilon>0$. For all $m$ sufficiently large, suppose $\spincv$ is a spin$^c$ structure with $\abs{\gen{c_1(\spincv), [Q]}} = m$, and $h^t_{0,\spincv} \ne 0$, then
\[
\abs{\gen{c_1(\spincv), [P_\delta]}} < (1+\epsilon)(k+md).
\]
\end{lemma}

 With the restricted spin$^c$ evaluation, the filtration shifts on  the truncated map can be much more effectively controlled. The following is parallel to \cite[Proposition 6.16]{HL}.

\begin{proposition} \label{prop: h0-filt}
Fix $t \in \N$ and $0 <\epsilon <1$. For all $m$ sufficiently large, the map $\tilde h^t_{0}$ is filtered with respect to $\JJ_{\alpha\beta}$ and $\JJ_{\alpha\delta}$.
\end{proposition}

\begin{proof}

We will start by looking at the filtration shift of a general term in $h^t_{0,\spincv},$ before specializing to the case of $\tilde h^t_{0}$.

Suppose $\x \in \T_\alpha \cap \T_\beta$ and $\a \in \T_\alpha \cap \T_\delta$ are generators such that $\spincs_z(\x) = \spincv|_Y = \spincs$ and $\spincs_z(\a) = \spincv|_{Y_{\lambda+m\mu}(K)}= \spincu$. Similar as before, $h^t_{0,\spincv} \circ \theta^{-1}$ in nonzero only on the summand $[\x,i] \otimes T^r$, where $r\in \Q$ is given by
\[
\frac{-k-md}{2d} \le r < \frac{-k+md}{2d} \quad \text{and} \quad r \equiv -\frac{1}{2d}( \gen{c_1(\spincv), [P_\gamma]} +k) \pmod m.
\]
So for some $p \in \Z$, we can write
\begin{equation} \label{eq: h0-filt-p-def}
\gen{c_1(\spincv), [P_\gamma]} = -2dr-k + 2pmd.
\end{equation} 
In other words, $p$ is the unique integer for which
\begin{equation} \label{eq: h0-filt-p-bounds}
(2p-1) md < \gen{c_1(\spincv), [P_\gamma]} \le (2p+1)md.
\end{equation}
In particular, if $\spincv$ satisfies \eqref{eq: h0-filt-Pg-bound}, then $p=0$.

On the other hand, consider the number $s_\spincu$ associated with the spin$^c$ structure $\spincu$. By its definition \eqref{eq: su-def}, combined with \eqref{eq: agd-c1-a} and \eqref{eq: adb-c1-a}, we have
\begin{align*}
    2nd s_\spincu &\equiv \gen{c_1(\spincv),[P_{n,\delta}]} + n(k+md)   \pmod{2n(k+md)}    \\
    &\equiv 2\Al(\a) + n^2 d - n(k+md)   \pmod{2n(k+md)} \\
    &\equiv \frac{\nu}{m} \gen{c_1(\spincv),n[R]}. \pmod{2n(k+md)}
\end{align*}
Hence for some $q \in \Z$, we can write
\begin{equation} \label{eq: h0-filt-q-def}
\frac{\nu}{m} \gen{c_1(\spincv), n[R]} = 2nd s_\spincu + 2nq(k+md),
\end{equation}
so that
\begin{equation} \label{eq: h0-filt-q-bounds}
(2q-1)(k+md) < \frac{\nu}{m} \gen{c_1(\spincv), [R]} \le (2q+1)(k+md).
\end{equation}
Again, suppose $\spincv$ satisfies \eqref{eq: h0-filt-R-bound}, then $q=0$.


For a   rectangle $\rho \in \pi_2(\x,\Theta_{\beta\gamma}, \Theta_{\gamma\delta}, \a)$  that contributes to $h^\circ_{0,\spincv}$, compute:
\begin{align*}
\JJ_{\alpha\delta}([\a, &i-n_z(\rho)]) - \JJ_{\alpha\beta}([\x,i] \otimes T^r) \\
&= \frac{\Al(\a)}{k+md} -n_z(\rho) + \frac{nd^2m(2s_\spincu-n)}{2k(k+md)} + \frac{2ndr + nk+n^2 d}{2k} \\
&= \frac{\frac{\nu}{m} \gen{c_1(\spincv), n[R]}  +n(k+md) -n^2 d}{2(k+md)}  + \frac{md (\frac{\nu}{m} \gen{c_1(\spincv), n[R]} - 2nq(k+md)) }{2k(k+md)}   \\
& \qquad  + \frac{-n\gen{c_1(\spincv), [P_\gamma]} - nk + 2npmd}{2k} - \frac{nd^2m}{2k(k+md)} + \frac{nk+n^2 d}{2k}\\
&\qquad- n_{z_n}(\rho) + \sum_{j=1}^n\left( n_{z_j}(\rho) - n_{u_j}(\rho) \right)\\
&= \frac{n}{2k} \gen{c_1(\spincv), \frac{\nu}{m}[R]- [P_\gamma]} + \frac{(p-q)nmd}{k} - \frac{n^2 d}{2k} + \frac{nk+n^2 d}{2k} \\
& \qquad - n_{z_n}(\rho) +  \sum_{j=1}^n\left( n_{z_j}(\rho) - n_{u_j}(\rho) \right) \\
&= \frac{1}{2m} \gen{c_1(\spincv),  n[Q] } + \frac{(p-q)nmd}{k}  + \frac{n}{2}   - n_{z_n}(\rho) +  \sum_{j=1}^n\left( n_{z_j}(\rho) - n_{u_j}(\rho) \right) \\
&= \frac{1}{2m} \left(-2m \sum_{j=1}^n\left( n_{z_j}(\rho) - n_{u_j}(\rho) \right) -nm \right)  + \frac{(p-q)nmd}{k}  + \frac{n}{2}  \\
&\qquad - n_{z_n}(\rho) +  \sum_{j=1}^n\left( n_{z_j}(\rho) - n_{u_j}(\rho) \right) \\
&= - n_{z_n}(\rho) + \frac{(p-q)nmd}{k}
\end{align*}
Therefore, in order to show that $\tilde h^t_{0}$ is filtered, we only need to show that $p-q = 0$ whenever $\spincv$ satisfies the conditions from Definition \ref{def: h0-trunc}.

If $\spincv$ satisfies \eqref{eq: h0-filt-Pg-bound} and \eqref{eq: h0-filt-R-bound}, we immediately deduce that $p=q=0$. This leaves us with only one case, when $\spincv$ satisfies \eqref{eq: h0-filt-Q-bound}, namely $\gen{c_1(\spincv), [Q]} = em$ where $e=\pm 1$. By Lemma \ref{lemma: h0-trunc}, we may also assume that $\abs{c_1(\spincv), [P_\delta]} < (1+\epsilon)(k+md)$. Recall that $[P_\gamma] = [P_\delta] + d[Q]$ and $\frac{\nu}{m}[R] = [P_\delta] + \frac{k+md}{m}[Q]$. Therefore,
\begin{align*}
(e-1-\epsilon) md -(1+\epsilon)k  < \gen{c_1(\spincv), [P_\gamma]} &<  (e+1+\epsilon)md + (1+\epsilon)k \\
(e-1-\epsilon)(k+md) < \frac{\nu}{m} \gen{c_1(\spincv), [R]} &< (e+1+\epsilon)(k+md).
\end{align*}
Comparing with \eqref{eq: h0-filt-p-def} and \eqref{eq: h0-filt-q-def} respectively, for sufficiently large $m$, this implies that $p,q \in \{0,e\}$. Also through \eqref{eq: h0-filt-p-bounds}and \eqref{eq: h0-filt-q-bounds}, we have 
\[
2(q-p-1)md + (2q-1)k < ke < 2(q-p+1)md + (2q+1)k.
\]
It then follows that $p=q$, as required.

\end{proof}

\subsubsection{The map $h_1^t$} \label{sssec: h1-filt}

We recall the definition of the truncated map $h^t_{1}$ from Hedden-Levine's argument.

\begin{definition}[Definition 6.17 in \cite{HL}] \label{def: h1-trunc}
Let $\tilde h^t_1$ denote the sum of all terms $h^t_{1,\spincv}$ for which $\spincv$ satisfies
\begin{equation} \label{eq: h1-trunc-Q-bound}
\gen{c_1(\spincv), [Q]} = \pm m.
\end{equation}
\end{definition}

Similar to the previous case, the truncated map is enough to fulfill the required condition in Proposition \ref{prop: rect-filt}, as indicated by the following lemma.

\begin{lemma}[Lemma 6.19 in \cite{HL}] \label{lemma: h1-trunc-nulhtpy}
Fix $t \in \N$. For all $m$ sufficiently large, $\tilde h_1^t$ is a null-homotopy of $f_2^t \circ f_1^t$.
\end{lemma}

An extra spin$^c$ constraint is needed in the proof, given by the Maslov grading bound.

\begin{lemma} \label{lemma: h1-trunc}
Fix $t \in \N$. For all $m$ sufficiently large, if $\spincv$ is any spin$^c$ structure with $\abs{\gen{c_1(\spincv), [Q]}} = m$, and $h^t_{1,\spincv} \ne 0$, then $\abs{\gen{c_1(\spincv), [P_\gamma]}} < md$.
\end{lemma}

We are ready to prove the following proposition, parallel to \cite[Proposition 6.22]{HL}.

\begin{proposition} \label{prop: h1-filt}
Fix $t \in \N$. For all $m$ sufficiently large, the map $\tilde h_1^t$ is filtered with respect to the filtrations $\JJ_{\alpha\gamma}$ and $\JJ_{\alpha\beta}$.
\end{proposition}

\begin{proof}

Suppose $\spincv$ is a spin$^c$ structure for which $h^t_{1,\spincv} \ne 0$, and assume $\gen{c_1(\spincv), [Q]} = em$, where $e=\pm 1.$  Let $\q \in \T_\alpha \cap \T_\gamma$ and $\x \in \T_\alpha \cap \T_\beta$ and  be generators such that   $\spincs_z(\q)=\spincv|_{Y_\lambda}$ and $\spincs_z(\x)=\spincv|_Y$. 

Similar as before, note that $\theta \circ h^\circ_{1,\spincv}$  lands in the summand $[\x,i] \otimes T^r \GR$, where $r\in \Q$ is given by
\[
\frac{-k-md}{2d} \le r < \frac{-k+md}{2d} \quad \text{and} \quad r \equiv -\frac{1}{2d}( \gen{c_1(\spincv), [P_\delta]} +k+md) \pmod m,
\]
so for some $p\in \Z$, we can write
\[
\gen{c_1(\spincv), [P_\delta]} = -2dr-k + (2p-1) md,
\]
which implies that
\begin{equation} \label{eq: h1-filt-p-bounds}
(2p-2)md < \gen{c_1(\spincv), [P_\delta]} \le (2p)md.
\end{equation}

By Lemma \ref{lemma: h1-trunc}, we may assume that $\abs{\gen{c_1(\spincv), [P_\gamma]}} < md$. Recall that $[P_\gamma] = [P_\delta] + d[Q]$, and therefore
\[
\gen{c_1(\spincv), [P_\delta]} = \gen{c_1(\spincv), [P_\gamma]} - emd.
\]
If $e = 1$, this implies $p=0$; if $e = -1$, this implies $p=1$. Either way, it holds that 
$2p-1=-e$.
By \eqref{eq: agdb-c1-Q}, we have
\begin{align*}
    (2p-2)nm & =  -\gen{c_1(\spincv), n[Q]} -nm\\
    &=-2m \left( -n_z(\rho) + n_{z_n}(\rho) -  \sum_{j=1}^n\left( n_{z_j}(\rho) - n_{u_j}(\rho) \right)  \right) +nm -nm\\
    &=2m n_z(\rho) - 2m n_{z_n}(\rho) + 2m \sum_{j=1}^n\left( n_{z_j}(\rho) - n_{u_j}(\rho) \right).
\end{align*}

Now, we compute:
\begin{align*}
&\JJ_{\alpha\gamma}([\q,i]) - \JJ_{\alpha\beta}([\x,i-n_z(\rho)] \otimes T^r) \\
&= \AlNorm_{w,z_n}(\q) +  \frac{2nd r +nk+n^2 d}{2k} +n_z(\rho) \\
&= \frac{1}{2k} \left( 2\Al(\q) + 2ndr + nk + n^2 d \right) +n_z(\rho) \\
&= \frac{1}{2k} \left( 2\Al(\q)  - n\gen{c_1(\spincv),[P_\delta]}  + (2p-1)nmd + n^2 d \right) +n_z(\rho) \\
&= \frac{1}{2k} \left( 2\Al(\q) -2\Al(\x) + 2d n_w(\rho) - 2dn_z(\rho) +  (2p-1)nmd -n(k+md) + n^2 d \right) +n_z(\rho) \\
&= \frac{1}{2k} \Big( 2\MultComb(\rho) +  2d n_w(\rho) - 2dn_z(\rho) + (2p-2)nmd \Big) +n_z(\rho) \\
&= \frac{1}{k} \Big( -nd n_w(\rho) - (k+md -nd) n_z(\rho) + (k+md)n_{z_n}(\rho) -md\sum_{j=1}^n\left( n_{z_j}(\rho) - n_{u_j}(\rho) \right)\\
&\qquad +  d n_w(\rho) - dn_z(\rho) +  md n_z(\rho) -md n_{z_n}(\rho) +md  \sum_{j=1}^n\left( n_{z_j}(\rho) - n_{u_j}(\rho) \right)   \Big) +n_z(\rho) \\
&= n_{z_n}(\rho) \ge 0.
\end{align*}

\end{proof}

\subsubsection{The map $h_2^t$} \label{sssec: h2-filt}

Let us start by recalling the definition of the truncated map $\tilde h^t_2$ from \cite{HL}, which is similar to the previous case. 
For any $\spincv \in \Spin^c_0(X_{\alpha\delta\beta\gamma})$,  let $\tilde h^t_2$ denote the sum of all terms $h^t_{2,\spincv}$ such that  
\begin{equation} \label{eq: h2-trunc-Q-bound}
\gen{c_1(\spincv), [Q]} = \pm m.
\end{equation}

As before, the next two lemmas show that $\tilde h^t_2$ satisfies the first condition in Proposition \ref{prop: rect-filt} and give an extra spin$^c$ evaluation bound that we will use in the proof, respectively.

\begin{lemma}[Lemma 6.24 in \cite{HL}] \label{lemma: h2-trunc-nulhtpy}
Fix $t \in \N$. For all $m$ sufficiently large, $\tilde h_2^t$ is a null-homotopy of $f_0^t \circ f_2^t$.  \qed
\end{lemma}

\begin{lemma}[Lemma 6.25 in \cite{HL}] \label{lemma: h2-trunc}
Fix $t \in \N$ and $\epsilon > 0$. For all $m$ sufficiently large, if $\spincv$ is any spin$^c$ structure with $\abs{\gen{c_1(\spincv), [Q]}} = m$ (where $e$ is an odd integer), and $h^t_{1,\spincv} \ne 0$, then
\begin{equation} \label{eq: h2-trunc-R-bound}
\abs{\gen{c_1(\spincv), [R]}} < \frac{m(k+md)}{\nu}. 
\end{equation}
\end{lemma}

The following proposition is parallel to \cite[Proposition 6.27]{HL}.

\begin{proposition} \label{prop: h2-filt}
Fix $t \in \N$. For all $m$ sufficiently large, the map $\tilde h_2^t$ is filtered with respect to the filtrations $\JJ_{\alpha\delta}$ and $\JJ_{\alpha\gamma}$.
\end{proposition}

\begin{proof}
Let $\a \in \T_\alpha \in \T_\delta$ and $\q \in \T_\alpha \cap \T_\gamma$ be generators such that $\spincs_z(\q)=\spincv|_{Y_\lambda}$ and $\spincs_z(\a)=\spincv|_{Y_{\lambda+m\mu}}=\spincu$, and suppose $\rho \in \pi_2(\a, \Theta_{\delta\beta}, \Theta_{\beta\gamma}, \q)$ is a rectangle that contributes to $h^t_{2,\spincv}([\a,i])$. Associated with $\spincu$ is a 
number $s_\spincu$, which by its definition \eqref{eq: su-def} satisfies
\[
-(k+md) < 2d s_\spincu \le k+md \quad \text{and} \quad  2nd s_\spincu \equiv  \gen{c_1(\spincv),[P_{n,\delta}]} + n (k+md)   \pmod {2n(k+md)}.
\]
 
Therefore, for some $q\in \Z$ we have
\[
\gen{c_1(\spincs_z(\rho)), [P_{n,\delta}]} = 2nd s_\spincu + (2q-1)n(k+md),
\]
so that
\[
(2q-2)(k+md) < \gen{c_1(\spincs_z(\rho)), [P_\delta]} \le 2q(k+md).
\]

Next we will assume that $\gen{c_1(\spincv), [Q]} = em$, where $e=\pm 1$, and $h^t_{2,\spincv} \ne 0$, so that $\spincv$ satisfies \eqref{eq: h2-trunc-R-bound}. We have
\begin{gather*}
\gen{c_1(\spincv), [P_\delta]} = \frac{\nu}{m} \gen{c_1(\spincv), [R]} - \frac{k+md}{m} \gen{c_1(\spincv), [Q]}. 
\end{gather*}
Using \eqref{lemma: h2-trunc}, compare the range of both sides of the equation. If $e=1$, this implies $q=0$; if $e=-1$, this implies $q=1$. In either case, it holds that $2q-1 = -e$.

Thus
\begin{align*}
    -(2q-1)m&= em=\gen{c_1(\spincs_z(\rho)), [Q]}\\
    &=2n_w(\rho) - 2n_z(\rho) + m.
\end{align*}

We now compute:
\begin{align*}
&\JJ_{\alpha\delta}([\a,i]) - \JJ_{\alpha\gamma}([\q,i-n_w(\rho)] ) \\
&= \frac{\Al(\a)}{k+md}  + \frac{nd^2m(2s_\spincu-n)}{2k(k+md)} - \frac{\Al(\q)}{k}  + n_z(\rho) \\
&= \frac{2k \Al(\a) - 2(k+md)\Al(\q) + nd^2m(2s_\spincu-n)}{2k(k+md)}  + n_z(\rho) \\
&= \frac{2(k+md)( \Al(\a) - \Al(\q)) - 2md \Al(\a) + nd^2m(2s_\spincu-n)}{2k(k+md)}  + n_z(\rho) \\
&= \frac{1}{k} \left( \Al(\a) - \Al(\q) \right) + \frac{nd^2m(2s_\spincu-n)}{2k(k+md)} + n_z(\rho)  - \frac{md\Al(\a)}{2k(k+md)}\\
&= \frac{1}{k} \left( \MultComb(\rho) - \frac{nmd}{2} \right) + \frac{md\gen{c_1(\spincs_z(\rho)), [P_{n,\delta}]} -(2q-1)m\cdot nd(k+md) -n^2 d^2 m}{2k(k+md)}  + n_z(\rho) \\
& - \frac{ md\left( \gen{c_1(\spincs_z(\rho)), [P_{n,\delta}]} -2(k+md)\left( n_z(\rho) - n_{z_n}(\rho) + \sum_{j=1}^n\left( n_{z_j}(\rho) - n_{u_j}(\rho) \right) \right) - n^2 d \right) } {2k(k+md)} \\
&= \frac{1}{k} \left( \MultComb(\rho) - \frac{nmd}{2} \right)  + \frac{nd(k+md)(2n_w(\rho) - 2n_z(\rho) + m )}{2k(k+md)}   + n_z(\rho) \\
&+ \frac{ 2md(k+md)\left( n_z(\rho) - n_{z_n}(\rho) + \sum_{j=1}^n\left( n_{z_j}(\rho) - n_{u_j}(\rho) \right) \right)}{2k(k+md)}\\
&= \frac{1}{k} \left( -nd n_w(\rho) - (k+md -nd) n_z(\rho) + (k+md) n_{z_n}(\rho) -md\sum_{j=1}^n\left( n_{z_j}(\rho) - n_{u_j}(\rho) \right)\right)      \\
& + \frac{ nd\left(n_w(\rho) - n_z(\rho)\right) + md \left( n_z(\rho) - n_{z_n}(\rho) + \sum_{j=1}^n\left( n_{z_j}(\rho) - n_{u_j}(\rho) \right)    \right) } {k} + n_z(\rho) \\
&= \frac{1}{k}\Big( -nd n_w(\rho) - (k -nd) n_z(\rho) + k n_{z_n}(\rho)  +  nd\left(n_w(\rho) - n_z(\rho)\right) \Big)       + n_z(\rho)\\
&= n_{z_n}(\rho) \\
&\ge 0.
\end{align*}
\end{proof}

\subsection{Pentagon maps} \label{ssec: pent-filt}
In this section we aim to prove the second part of Proposition \ref{prop: rect-filt}: $\tilde h^t_{j+1} \circ f^t_j +  f^t_{j+2} \circ \tilde h^t_j$ (where $j \in \Z/3$) are filtered quasi-isomorphisms. We will only focus on the case of $j=0$, which is the most technically difficult because of the twisted coefficients. The arguments for $j=1$ and $j=2$ are similar.
Again, our argument here is completely parallel to the argument presented in \cite[Section 6.4]{HL}, and we will leave out some of the technical details and focus on the adjustment needed due to a different setting. For a more detailed read, see \cite[Section 6.4]{HL}.

To begin, let $\tilde{\bm\beta} = (\tilde \beta_1, \dots, \tilde\beta_g)$ denote a small Hamiltonian isotopy of $\bm\beta$, such that each $\tilde \beta_i$ meets $\beta_i$ in a pair of points. We further require that $v$ is an extra reference point such that $v$ is in the same region of $\Sigma \minus (\bm\alpha \cup \bm\beta)$ as $z$ and in the same region of $\Sigma \minus (\bm\alpha \cup \tilde{\bm\beta})$ as $w$.  Finally, let $\Theta_{\beta\tilde\beta} \in \T_\beta \cap \T_{\tilde\beta}$ denote the canonical top-dimensional generator. \footnote{This is the same setting as depicted in \cite[Figure 6]{HL}, although in their text description  the order of $w$ and $z$ is switched by mistake.}

Define 
\[
\tilde \Psi^t_0  = \tilde h^t_{1} \circ f^t_0 +  f^t_{2} \circ \tilde h^t_0.
\]

 The fact that $\tilde \Psi^t_0$ is filtered follows immediately  from the previous sections, since it is a sum of filtered maps. In order to show it is a quasi-isomorphism, we relate it to the map
 \begin{equation} \label{eq: Phi0}
\Phi_0^t \co \ul\CF^t(\Sigma, \bm\alpha, \bm\beta, z; \GR) \to \ul\CF^t(\Sigma, \bm\alpha, \tilde{\bm\beta}, z; \GR)
\end{equation}
given by
\begin{equation} \label{eq: Phi0-def}
\Phi_0^t(T^s \cdot [\x, i]) =
\sum_{\tilde\y \in \T_\alpha \cap \T_{\tilde\beta}} \sum_{\substack{\psi \in \pi_2(\x, \Theta_{\beta\tilde\beta}, \tilde\y) \\ \mu(\psi)=0 }} \#\MM(\psi) \, T^{s + n_w(\psi)-n_z(\psi)} \cdot [\tilde\y, i-n_z(\psi) ].
\end{equation}

By the work of \cite{HeddenMarkFractional}, $\Phi_0^t$ is a chain isomorphism. Moreover, for any $\psi \in \pi_2(\x, \Theta_{\beta\tilde\beta}, \tilde\y)$, we have $\AlNorm(\x) - \AlNorm(\tilde\y) = n_z(\psi) - n_w(\psi)$. It follows that $\Phi_0^\circ$ is a filtered isomorphism with respect to $\JJ_{\alpha\beta}$ and $\JJ_{\alpha\tilde\beta}$.

The aim is then to show  $\tilde \Psi^t_0$ is filtered homotopy equivalent to  $\Phi_0^t$.   The argument hinges on the following pentagon-counting map, defined in \cite{HL}:
\begin{equation} \label{eq: g0-Phi0}
g_0^\circ \co \ul\CF^\circ(\Sigma, \bm\alpha, \bm\beta, z; \GR) \to \ul\CF^\circ(\Sigma, \bm\alpha, \tilde{\bm\beta}, z; \GR)
\end{equation}
 given by
\begin{equation}
\label{eq: g0-def}
g_0^\circ(T^s \cdot [\x,i]) = \sum_{\y \in \T_\alpha \cap \T_{{\tilde \beta}}}  \sum_{\substack{\sigma \in \pi_2(\x, \Theta_{\beta\gamma}, \Theta_{\gamma\delta}, \Theta_{\delta{\tilde \beta}}, \y) \\ \mu(\sigma)=-2 \\ \mathclap{s + n_w(\sigma) - n_v(\sigma) \equiv 0 \pmod m}}}   \#\MM(\sigma) \, T^{n_v(\sigma) - n_z(\sigma)} [\y, i-n_z(\sigma)].
\end{equation}
Hedden and Levine defined a truncated version of $g_0^t$, again by throwing away bad terms, and they proceed to show the truncated map $\tilde g^t_0$ establishes the required filtered homotopy equivalence between $\tilde \Psi^t_0$ and $\Phi_0^t$. This argument works for our case as well, except that the grading shift on $\tilde g^t_0$ needs to be recalculated.

Before we discuss the filtration shifts, let us state analogues of the results of Section \ref{sec: alex} for pentagons. To begin, let $V$ be the $(\beta,\tilde\beta)$ periodic domain with $\partial V = \beta_g - \tilde\beta_g$. In other word, $V$ is the thin domain between  $\beta_g $ and $\tilde\beta_g$. We have  $n_v(V) = 1$, and $n_w(V) = n_z(V) = n_{z_j}(V) = n_{u_j}(V) = 0$ for $j=1,\cdots, n$. Let $\tilde P_\gamma$, $\tilde P_\delta$, and $\tilde Q$ be the analogues of $P_\gamma$, $P_\delta$, and $Q$ with $\beta$ circles replaced by ${\tilde \beta}$ circles: up to thin domains, we have
\[
[\tilde P_\gamma] = [P_\gamma] + k[V], \quad [\tilde P_\delta] = [P_\delta]+(k+md)[V], \quad \text{and} \quad
[\tilde Q] = [Q]-m[V].
\]

The Heegaard diagram determines a $4$-manifold $X_{\alpha\beta\gamma\delta\tilde\beta}$, which admits various decompositions into the pieces described in Section \ref{ssec: cob}; for instance, we have
\[
X_{\alpha\beta\gamma\delta\tilde\beta} = X_{\alpha\beta\gamma\delta} \cup_{Y_{\alpha\delta}} X_{\alpha\delta\tilde\beta} = X_{\alpha\beta\gamma} \cup_{Y_{\alpha\gamma}} X_{\alpha\gamma\delta\tilde\beta} .
\]
In the intersection pairing form on $H_2(X_{\alpha\beta\gamma\delta\tilde\beta})$, we have
\[
[V] \cdot [V] = [V] \cdot [Q] = [V] \cdot [\tilde Q] = 0,
\]
and all other intersection numbers can be deduced accordingly. Compare the following with \cite[Lemma 6.31]{HL}.

\begin{lemma} \label{le: pentagonspinc}
For any $\x \in \T_\alpha \cap \T_\beta$, $\tilde\y \in \T_\alpha \cap \T_{\tilde\beta}$, and $\sigma \in \pi_2(\x, \Theta_{\beta\gamma}, \Theta_{\gamma\delta}, \Theta_{\delta\tilde\beta}, \tilde \y)$, we have:
\begin{align}
\label{eq: abgdb-alex}
 n\Al(\x) - n\Al(\tilde \y) &= -nd n_w(\sigma) - (k+md - nd) n_z(\sigma) + (k+md) n_{z_n}(\sigma) \\
\nonumber & \qquad \qquad - md \sum_{j=1}^n\left( n_{z_j}(\sigma) - n_{u_j}(\sigma) \right)   \\
\label{eq: abgdb-c1-Pg}
\qquad \qquad \gen{c_1(\spincs_z(\sigma)), [P_\gamma]} &= 2 \Al(\x) + 2dn_w(\sigma) - 2d n_v(\sigma) + k \\
\label{eq: abgdb-c1-Pdt}
\qquad \qquad \gen{c_1(\spincs_z(\sigma)), [\tilde P_\delta]} &= 2 \Al(\tilde\y) + 2dn_z(\sigma) - 2d n_v(\sigma) + (k+md) \\
\label{eq: abgdb-c1-Q}
\qquad \qquad \gen{c_1(\spincs_z(\sigma)), n[Q]} &= -2m \sum_{j=1}^n\left( n_{z_j}(\sigma) - n_{u_j}(\sigma) \right) -nm \\
\label{eq: abgdb-c1-Qt}
\qquad \qquad \gen{c_1(\spincs_z(\sigma)), n[\tilde Q]} &= 2m\Big( -n_z(\sigma) + n_{z_n}(\sigma) - \sum_{j=1}^n\left( n_{z_j}(\sigma) - n_{u_j}(\sigma) \right) \Big) -nm \\
\label{eq: abgdb-c1-V}
\qquad \qquad \gen{c_1(\spincs_z(\sigma)), n[V]} &= 2n_z(\sigma) - 2n_{z_n}(\sigma)
\end{align}
\end{lemma}

\begin{proof}
This proof  resembles the one for Proposition \ref{prop: alex-rectangle}, building upon the calculations made in Propositions \ref{prop: alex-triangle} and \ref{prop: alex-rectangle}.

For any $\x \in \T_\alpha \cap \T_\beta$, $\tilde\y \in \T_\alpha \cap \T_{\tilde\beta}$, and $\sigma \in \pi_2(\x, \Theta_{\beta\gamma}, \Theta_{\gamma\delta}, \Theta_{\delta\tilde\beta}, \tilde \y)$, choose $\a \in \T_\alpha \cap \T_\delta$, $\rho \in \pi_2(\x, \Theta_{\beta\gamma}, \Theta_{\gamma\delta},\a)$ and $\psi \in \pi_2(\a,\Theta_{\delta\tilde\beta}, \tilde \y)$ such that $\spincs_z(\rho) = \spincs_z(\sigma)|_{X_{\alpha\beta\gamma\delta}}$ and $\spincs_z(\psi) = \spincs_z(\sigma)|_{X_{\alpha\delta{\tilde\beta}}}$. Up to adding copies of $\Sigma,$ which does not change the spin$^c$ structure condition, we may assume that $n_z(\sigma) = n_z(\rho) + n_z(\psi)$. Hence, $S = \DD(\sigma) - \DD(\rho * \psi)$ is a quintuply periodic domain with $n_z(S)=0$. Since the function $\MultComb$ vanishes on all periodic domains, we have:
\begin{align*}
n\Al(\x) - n\Al(\tilde \y) &= (n\Al(\x) - \Al(\a)) + (\Al(\a) - n\Al(\tilde \y)) \\
&= \MultComb(\rho) - \frac{n(k+md)-n^2 d}{2} + \MultComb(\psi) + \frac{n(k+md)-n^2 d}{2}\\
&= \MultComb(\sigma) 
\end{align*}
which proves \eqref{eq: abgdb-alex}.

Next, up to thin domains, we have $S = x P_\gamma + y R$, similar as before, we can solve
\[
x =  \frac{n_w(S)}{k}.
\]
Using \eqref{eq: abgd-PD-self-int}, \eqref{eq: abgd-int-form}, and \eqref{eq: abgd-c1-Pg}, we compute:
\begin{align*}
\gen{c_1(\spincs_z(\sigma)), [P_\gamma]}
&= \gen{c_1(\spincs_z(\rho)), [P_\gamma]} + 2x [P_\gamma]^2  \\
&= 2 \Al(\x) + 2d(n_w(\rho)-n_z(\rho)) + k + 2d( n_w(\sigma) - n_w(\rho) - n_w(\psi)) \\
&= 2 \Al(\x) + 2d(n_w(\sigma) - n_z(\rho) -n_w(\psi)) +k  \\
&= 2 \Al(\x) + 2d(n_w(\sigma) - n_v(\rho) -n_v(\psi)) +k  \\
&= 2 \Al(\x) + 2d(n_w(\sigma) -n_v(\sigma)) +k,
\end{align*}
which proves \eqref{eq: abgdb-c1-Pg}.
The computations for the rest of the items are similar and left for the reader.



\end{proof}

\begin{remark} \label{re: veven}
We can adopt the perspective in Remark \ref{re: nprime}: for a fixed $n$, the relations in Lemma \ref{le: pentagonspinc} hold for all $n'$ with $1\leq n' \leq n$. In particular, taking $n'=1,$ then \eqref{eq: abgdb-c1-V} implies that $\gen{c_1(\spincs_z(\sigma)), [V]}$ is always an even integer, and \eqref{eq: abgdb-c1-Qt} implies that $\gen{c_1(\spincs_z(\sigma)), [\tilde Q]}=em$ for some odd integer $e.$

Moreover, by varying $n'$, we obtain 
$n_z(\sigma) - n_{z_1}(\sigma)=n_{z_j}(\sigma) - n_{z_{j+1}}(\sigma)$ for $j=1,\cdots,n-1.$ This is true because the pentagon class $\sigma$ has no $(\alpha,\gamma)$ or $(\alpha,\delta)$ endpoints.
\end{remark}

\begin{definition}[Definition 6.33 in \cite{HL}] \label{def: g0-trunc}
Fix $0 < \epsilon < \frac13$. Let $\tilde g^t_0$ denote the sum of all terms $g^t_{0,\spincv}$ for which $\spincv$  satisfies either both
\begin{align}
\label{eq: g0-trunc-Pg} \abs{\gen{c_1(\spincv), [P_\gamma]}} &< \epsilon md  \\
\label{eq: g0-trunc-Qt} \gen{c_1(\spincv), [\tilde Q]} &= \pm m, \\
\intertext{or it satisfies both}
\label{eq: g0-trunc-Pdt} \abs{\gen{c_1(\spincv), [\tilde P_\delta]}} &< (1+\epsilon)(k+md)  \\
\label{eq: g0-trunc-Q} \gen{c_1(\spincv), [Q]} &= \pm m,  \\
\intertext{or it satisfies}
\label{eq: g0-trunc-V} \gen{c_1(\spincv), [V]} &= 0.
\end{align}
\end{definition}

The following lemma puts further constraints on the spin$^c$ evaluations. It is stated and proved in \cite{HL}.

\begin{lemma} [Lemma 6.32 in \cite{HL}] \label{lemma: g0-trunc}
Fix $t \in \N$ and $0<\epsilon < \epsilon' < 1$. For all $m$ sufficiently large, if $\spincv$ is any spin$^c$ structure for which $g^t_{0,\spincv} \ne 0$, then the following implications hold:
\begin{enumerate}
\item \label{it: g0-trunc-Pg-Pgt}
If $\abs{\gen{c_1(\spincv), [P_\gamma]}} < \epsilon md$ and $\gen{c_1(\spincv), [\tilde Q] } = \pm m$, then
\begin{equation} \label{eq: g0-trunc-Pgt}
\abs{\gen{c_1(\spincv), [\tilde P_\gamma]}} < \epsilon' md.
\end{equation}
\item \label{it: g0-trunc-Pdt-Pd}
If $\gen{c_1(\spincv), [Q]} = \pm m$ and $\abs{\gen{c_1(\spincv), [\tilde P_\delta]}} < (1+\epsilon)(k+md)$, then \begin{equation} \label{eq: g0-trunc-Pd}
\abs{\gen{c_1(\spincv), [P_\delta]}} < (1+\epsilon')(k+md).
\end{equation}
\item \label{it: g0-trunc-V-Q}
If $\gen{c_1(\spincv), [V]}=0$, then
\[
\gen{c_1(\spincv), [Q]} = \gen{c_1(\spincv), [\tilde Q]}  = \pm m.
\]
\end{enumerate}
\end{lemma}

We are ready to compute the filtration shifts. The following lemma is parallel to \cite[Lemma 6.34]{HL}.

\begin{lemma} \label{lemma: g0-filt}
Fix $t \in \N$ and $\epsilon>0$. For all $m$ sufficiently large, the map $\tilde g^t_0$ is filtered with respect to the filtrations $\JJ_{\alpha\beta}$ and $\JJ_{\alpha\tilde\beta}$.
\end{lemma}

\begin{proof}
For any $\x \in \T_\alpha \cap \T_\beta$, we have 
\begin{align*}
(\theta_{\alpha\tilde\beta} &\circ  g^t_0  \circ \theta_{\alpha\beta}^{-1}) ([\x,i] \otimes T^r) \\
&=
\sum_{\tilde\y \in \T_\alpha \cap \T_{\tilde\beta}}  \sum_{\substack{\sigma \in \pi_2(\x, \Theta_{\beta\gamma}, \Theta_{\gamma\delta}, \Theta_{\delta\tilde\beta}, \tilde\y) \\ \mu(\sigma)=-2 \\ \mathclap{r + k/d +\AlNorm_{w,z}(\x) + n_z(\sigma) - n_v(\sigma) \equiv 0 \pmod m}}}   \#\MM(\sigma) \, [\tilde\y, i-n_z(\sigma)] \otimes T^{-k/d -\AlNorm_{w,z} (\tilde\y) + n_v(\sigma) - n_z(\sigma)} 
\end{align*}
So as before, the composition $(\theta_{\alpha\tilde\beta} \circ  g^t_{0,\spincv}  \circ \theta_{\alpha\beta}^{-1})$ is only nonzero on the summand $\CF^t(\bm\alpha, \bm\beta, \spincs, z) \otimes T^r \GR $ and lands in $\CF^t(\bm\alpha, \bm\beta, \tilde \spincs, z) \otimes T^s \GR $, where $\spincs = \spincv |_{Y_{\alpha\beta}}$ , $\tilde\spincs = \spincv |_{Y_{\alpha\tilde\beta}}$ and $r,s \in \Q$ are given by
\begin{equation} \label{eq: g0-r-s-def}
\begin{aligned}
\frac{-k-md}{2d} &\le r < \frac{-k+md}{2d} & r &\equiv -\frac{1}{2d}( \gen{c_1(\spincv), [P_\gamma]} +k) \pmod m \\
\frac{-k-md}{2d} &\le s < \frac{-k+md}{2d} & s &\equiv -\frac{1}{2d}( \gen{c_1(\spincv), [\tilde P_\delta]} +k+md) \pmod m.
\end{aligned}
\end{equation}
Therefore, for some $p,q\in \Z$ we can write
\begin{align*}
\gen{c_1(\spincv), [P_\gamma]} &= -2dr -k + 2pmd \\
\gen{c_1(\spincv), [\tilde P_\delta]} &= -2ds -k + (2q-1) md,
\end{align*}
so that 

\begin{align}
\label{eq: g0-filt-Pg-bounds} (2p-1)md &< \gen{c_1(\spincv), [P_\gamma]} \le (2p+1)md \\
\label{eq: g0-filt-Pdt-bounds} (2q-2)md &< \gen{c_1(\spincv), [\tilde P_\delta]} \le 2q md.
\end{align}

Next, assume that $\gen{c_1(\spincv), [\tilde Q]} = em$, where $e$ is an odd integer.     For any $\spincv$ that appears in the definition of $\tilde g^t_0$,      we now claim 
\[
e=2p-2q+1.
\]
Following the same argument in \cite{HL}, we will first establish that $\gen{c_1(\spincv), [V]}$ is small compared to $m$. To be precise, we will show $\abs{k \gen{c_1(\spincv), [V]}} < 2md$ in all three cases in Definition \ref{def: g0-trunc}.
\begin{enumerate}[label=(\roman*)]
    \item If $\spincv$ satisfies \eqref{eq: g0-trunc-Pg} and \eqref{eq: g0-trunc-Qt}, hence \eqref{eq: g0-trunc-Pgt} by Lemma \ref{lemma: g0-trunc}\eqref{it: g0-trunc-Pg-Pgt},where we take $\epsilon' = 2\epsilon$. Then
\begin{align*}
\abs{k \gen{c_1(\spincv), [V]}} &= \abs{ \gen{c_1(\spincv), [P_\gamma] - [\tilde P_\gamma]}} \\
&\le \abs{\gen{c_1(\spincv), [P_\gamma]}} + \abs{\gen{c_1(\spincv), [\tilde P_\gamma]}}  \\
&< 3\epsilon md \\
&< 2md
\end{align*}
as required.
     \item Suppose that $ \spincv$ satisfies \eqref{eq: g0-trunc-Pdt} and \eqref{eq: g0-trunc-Q} and hence \eqref{eq: g0-trunc-Pd} by Lemma \ref{lemma: g0-trunc}\eqref{it: g0-trunc-Pdt-Pd}, again taking $\epsilon'=2\epsilon$.
Therefore,
\begin{align*}
(k+md) \abs{\gen{c_1(\spincv), [V]}} = \abs{\gen{c_1(\spincv), [\tilde P_\delta] - [P_\delta]}} &< (2+3\epsilon)(k+md).
\end{align*}
For $m$ sufficiently large, we again obtain $\abs{k \gen{c_1(\spincv), [V]}} <2md$, as required.
      \item If $ \spincv$ satisfies \eqref{eq: g0-trunc-V}, then the requirement is met immediately. 
\end{enumerate}
We proceed by observing since
\[
[P_\gamma] - [P_\delta] = d[Q], \quad [\tilde P_\delta] = [P_\delta] + (k+md) [V] \quad \text{and} \quad[\tilde Q] = [Q] -m[V],
\]
it follows
\begin{equation}
   \gen{c_1(\spincv),[\tilde P_\delta ] - [P_\gamma] + d [\tilde Q]} = \gen{c_1(\spincv),k[V]}.  
\end{equation}
Comparing the range of both sides of the equation, using \eqref{eq: g0-filt-Pg-bounds} and \eqref{eq: g0-filt-Pdt-bounds}, we obtain 
\begin{equation} \label{eq: g0-filt-V-bounds}
2q-2p-3 + e <     \frac{k\gen{c_1(\spincv), [V]}}{md} < 2q-2p+1 + e.
\end{equation}
When $\abs{\gen{c_1(\spincv), [V]}}$ is sufficiently small, this implies $e=2p-2q+1,$ proving the claim.

For any $\sigma \in \pi_2(\x, \Theta_{\beta\gamma}, \Theta_{\gamma\delta}, \Theta_{\delta\tilde\beta}, \tilde \y)$ contributing to $\tilde g^t_0$, we compute:
\begin{align*}
\JJ_{\alpha\beta}([\x,i] \otimes T^r) &- \JJ_{\alpha\tilde\beta}([\tilde\y,i-n_z(\sigma)] \otimes T^s) \\
&= i - \frac{2nd r +nk+n^2 d}{2k} - (i-n_z(\sigma)) + \frac{2nd s +nk+n^2 d}{2k} \\
&= \frac{nd(s-r)}{k} + n_z(\sigma) \\
&= \frac{\gen{c_1(\spincv), n[P_\gamma]-n[\tilde P_\delta]} + (2q-2p-1)nmd }{2k} + n_z(\sigma) \\
&= \frac{\gen{c_1(\spincv), n[P_\delta] + nd[Q] -n[\tilde P_\delta]} -enmd }{2k} + n_z(\sigma) \\
&= \frac{\gen{c_1(\spincv), -(k+md)n[V] + nd[Q]} - \gen{c_1(\spincv),   nd[\tilde Q]} }{2k} + n_z(\sigma) \\
&= \frac{\gen{c_1(\spincv), -(k+md)n[V] + nmd[V]} }{2k} + n_z(\sigma) \\
&= \frac{-k\gen{c_1(\spincv), n[V]} }{2k} + n_z(\sigma) \\
&= n_{z_n}(\sigma) - n_z(\sigma) + n_z(\sigma)\\
&= n_{z_n}(\sigma)\\
&\geq 0.
\end{align*}
\end{proof}

\begin{lemma}[Lemma 6.35 in \cite{HL}] \label{lemma: g0-trunc-htpy}
Fix $t \in \N$. For all $m$ sufficiently large, the map $\tilde g^t_0$ is a (filtered) chain homotopy between $\tilde \Psi^t_0$ and $\Phi^t_0$.
\end{lemma}

\begin{proof}
The proof in \cite{HL} can proceed without modulation, except that in the last item (P-$5$) where $\sigma$ has the decomposition $\sigma = \rho * \psi$, where $\rho \in \pi_2(\Theta_{\beta\gamma}, \Theta_{\gamma\delta}, \Theta_{\delta\tilde\beta}, \tilde \Theta)$ and $\psi \in \pi_2(\x, \Theta, \tilde\y)$ with $\tilde \Theta \in \T_\beta \cap \T_{\tilde\beta}$, now $z$ and $z_n$ are in the same region of both $(\Sigma, \bm\beta, \bm\gamma, \bm\delta, \bm{\tilde \beta})$ and $(\Sigma, \bm\alpha, \bm\beta, \bm{\tilde \beta})$. As a result, $n_z(\sigma) = n_{z_n}(\sigma)$.  By \eqref{eq: abgdb-c1-V}, $\gen{c_1(\spincs_z(\sigma)), [V]}=0$, so $\spincv$ satisfies the condition \eqref{eq: g0-trunc-V} and must appear in the definition of $\tilde g^t_0$. 

\end{proof}

This concludes the proof of Proposition \ref{prop: rect-filt}, and hence of Theorem \ref{thm: CFt-cone-f2}.

\section{Proof of the filtered mapping cone formula}\label{sec: proofcone}

We are ready to prove the filtered mapping cone formula. According to Lemma \ref{lemma: cft}, it suffices to prove the mapping cone formula for  $\CF^t$. The proof follows closely the recipe from \cite[Section 7]{HL}, and we will write down the steps for completeness. 

\begin{proof}[Proof of Theorem \ref{thm: mapping-cone}]
We will only prove the case when $k>0$. Given $t\in \N$, by Theorem \ref{thm: CFt-cone-f2} , for a diagram $(\Sigma, \bm\alpha, \bm\beta, \bm\delta, w, z, z_n)$ and large enough $m$, we have doubly-filtered homotopy equivalence
\[
\begin{pmatrix} f_1^t \\ h_1^t \end{pmatrix}  \co \CF^t(\bm\alpha, \bm\gamma, z) \to \operatorname{Cone}(f_2^t).
\]


Fixing a spin$^c$ structure $\spinct \in \Spin^c(Y_\lambda(K))$, consider all its extension in $\Spin^c(W_\lambda).$  As discussed at the end of Section \ref{sssec: spinc}, these spin$^c$ structures form an orbit with the action of $\PD[C]$, where $C$ denotes the $2$-handle attached to $Y$. Their image under the bijection $E_{Y,\lambda,K}$ form an orbit in $\ul\Spin^c(Y,K)$ with the action of $\PD[K];$ call them $\xi_l$. The Alexander gradings  $A_{Y,K}(\xi_l)$ with $l \in \Z$ are precisely the arithmetic sequence $s_l$ with step length $k/d$, as defined  in Section \ref{ssec: statement}, where the index is fixed by
\[
  \frac{(2l-1)k}{2d}  < s_l \le \frac{(2l+1)k}{2d}.
\]
Under the bijection $E_{Y_\lambda, K_{n,\lambda}}\circ  (E_{Y,\lambda,K})^{-1} \co \ul\Spin^c(Y,K) \to \ul\Spin^c(Y_\lambda,K_{n,\lambda})$,  the set $\{\xi_l\}$ is identified with a sequence of spin$^c$ structures in  $\ul\Spin^c(Y_\lambda,K_{n,\lambda})$ (also with the action of $\PD[K]$). Their  $A_{Y_\lambda,K_{n,\lambda}}$ values form a $\Q/n\Z$ coset, which we denote by $A_{Y_\lambda,K_{n,\lambda}}(\spinct).$ By \eqref{eq: abg-alex}, we have
\[
nds_l \equiv kA_{Y_\lambda,K_{n,\lambda}}(\spinct) + \frac{-nk+n^2 d}{2}  \pmod{\Z}
\]


 Next, define $\spincv_l \in \Spin^c_0(\bar X_{ \alpha\gamma\delta})$ to be  the spin$^c$ structure extending $\spinct$ with
\[
\frac{(2l-1) mk}{\nu} < \gen{c_1(\spincv_l), [R]} \le \frac{(2l+1)mk}{\nu},
\]
and let  $\spincu_l = \spincv_l | _{Y_{\alpha\delta}}$. We want to show $s_{\spincu_l}=s_l$. (For the definition of $s_{\spincu_l}$ see (\ref{eq: su-def}). By \eqref{eq: agd-c1-q}, we have
\begin{align*}
    \frac{n\nu}{2m}\gen{c_1(\spincv_l),[R]}&=\Al(\q) - k \sum_{j=1}^n \left( n_{z_j}(\psi) - n_{u_j}(\psi) \right) +\frac{-nk + n^2 d}{2}\\
    &\equiv kA_{Y_\lambda,K_{n,\lambda}}(\spinct) +\frac{-nk + n^2 d}{2}   \pmod{k}\\
    &\equiv nd s_l.  \pmod{k}
\end{align*}
Therefore 
\[
\gen{c_1(\spincv_l),[R]}= \frac{2mds_l}{\nu}.
\]
On the other hand, according to Lemma \ref{le: f1range},
\[
\gen{c_1(\spincv_l),[R]}= \frac{2mds_{\spincu_l}}{\nu},
\]
which shows $s_{\spincu_l}=s_l$, as required. Lemma \ref{le: f1range} further shows that $f^t_{1,\spincv_l}$ is only non zero when $-L\leq l \leq L$ for some $L$.  In other words, the  image of $f^t_1$ restricted to $\CF^t(\bm\alpha,\bm\gamma,\spinct)$  is contained in the direct sum of $\CF^t(\bm\alpha,\bm\delta,\spincu_l)$ for $-L\leq l \leq L$.

Recall that spin$^c$ structures $\spincx_{\spincu_l}, \spincy_{\spincu_l}$ on $W_{\alpha\delta\beta}$, both restricting to $\spincu_l$ on $Y_{\alpha\delta}$,  are characterized by
\begin{align*}
\gen{c_1(\spincx_{\spincu_l}), [P_\delta]} &= 2ds_l - k-md \\
\gen{c_1(\spincy_{\spincu_l}), [P_\delta]} &= 2ds_l + k+md.
\end{align*}
We will write $\spincx_l = \spincx_{\spincu_l}$ and $\spincy_l = \spincy_{\spincu_l}$ to simplify the notation. The proof of Proposition \ref{prop: f2-filt} shows that $\spincx_{l}$ and  $\spincy_{l}$ are the only two spin$^c$ structures that may contribute to the map $f^t_{2}$ on $\CF^t(\bm\alpha,\bm\delta, \spincu_l)$. 
Observe $\spincx_{l}$ and $\spincy_{l-1}$ restrict to the same spin$^c$ structure  $\spincs_{\spincu_l}$ on $Y$. We will write  $\spincs_l =\spincs_{\spincu_l}.$ 

Moreover,  the images of the maps $\theta \circ f^t_{2, \spincx_{l}}$ and $\theta \circ f^t_{2, \spincy_{l-1}}$ both lie in the summand $\CF^t(\bm\alpha, \bm\beta, \spincs_l) \otimes T^{-s_l}$. 

Therefore, the complex $\CF^t(\Sigma, \bm\alpha, \bm\gamma, w, \spinct)$ equipped with filtrations $\II_{\alpha\gamma}$ and $\JJ_{\alpha\gamma}$, is doubly-filtered quasi-isomorphic to the doubly-filtered complex
\begin{equation} \label{eq: mapping-cone-CFt}
\xymatrix@R=0.6in{
\cdots \ar[dr] &\CF^t(\bm\alpha, \bm\delta, \spincu_{l-1}) \ar[d]|{F^t_{W'_m, \spincx_{l-1}}} \ar[dr]|{F^t_{W'_m, \spincy_{l-1}}} & \CF^t(\bm\alpha, \bm\delta, \spincu_{l}) \ar[d]|{F^t_{W'_m, \spincx_{l}}} \ar[dr]|{F^t_{W'_m, \spincy_{l}}} & \cdots  \ar[d]  \\
&\CF^t(\bm\alpha, \bm\beta, \spincs_{l-1}) \otimes T^{-s_{l-1}} & \CF^t(\bm\alpha, \bm\beta, \spincs_{l}) \otimes T^{-s_{l}} & \cdots 
}
\end{equation}
where the filtrations are inherited from those on $\CF^t(\bm\alpha, \bm\delta, z)$ and $\ul\CF^t(\bm\alpha, \bm\beta, z; \GR)$, respectively.

By Theorem \ref{thm: large-surgery}, there are doubly-filtered quasi-isomorphisms
\[
\Lambda^t_{\spincu_l}\co \CF^t(\Sigma, \bm\alpha, \bm\delta, \spincu_l) \to A^t_{\spincs_l, s_l}
\]
where the Alexander filtration on $\Lambda^t_{\spincu_l}$ is identified with the filtration $\JJ_{\spincu_l}$ from \eqref{eq: Ju-def}. Whereas each $\CF^t(\bm\alpha, \bm\beta, \spincs_{l}) \otimes T^{-s_{l}}$ can be identified with $B^t_{\spincs_l} = C_{\spincs_l}\{0 \le i \le t\}$, so that the complex in \eqref{eq: mapping-cone-CFt} is quasi-isomorphic to
\begin{equation} \label{eq: mapping-cone-AtBt}
\xymatrix@C=0.6in@R=0.6in{
\cdots \ar[dr] & A^t_{\xi_{l-1}} \ar[d]|{v^t_{\xi_{l-1}}} \ar[dr]|{h^t_{\xi_{l-1}}} & A^t_{\xi_{l}} \ar[d]|{v^t_{\xi_{l}}} \ar[d]|{v^t_{\xi_{l}}} \ar[dr]|{h^t_{\xi_{l}}}  &\cdots  \ar[d] \\
& B^t_{\spincs_{l-1}}   & B^t_{\spincs_{l}}& \cdots
}
\end{equation}
This is  the complex $X^t_{\lambda,\spinct,n,-L,L}$ from Section \ref{ssec: statement} by definition.

Disregarding the second filtration, the underlying mapping cone is isomorphic to the one defined in \cite{HL}, and with the mapping cone in \cite{rational}, for this purpose. The maps induced by cobordisms are independent of the choice of basepoints. Thus the absolute grading shift on our mapping cone matches with Hedden-Levine's. Placing the $z_n$ basepoint differently amounts to imposing a different $\JJ$ filtration to the mapping cone. To complete the proof, we just need to check the $\JJ$ filtration agrees with the descriptions (\ref{eq: Jt-def-A}) and (\ref{eq: Jt-def-B}) in Section \ref{ssec: statement}. 
\begin{itemize}
\item
On each summand $\CF^t(\bm\alpha, \bm\delta, \spincu_l)$, $\JJ_{\alpha\delta}$ is defined in \eqref{eq: ad-filt} as the Alexander filtration plus $\frac{nd^2m(2s_{\spincu_l}-n)}{2k(k+md)}$, thus $\JJ$ on $A^t_{\xi_l}$ would be $\JJ_{\spincu_l}$ plus the same shift:
\begin{align*}
\JJ([\x, i, j]) &= \JJ_{\spincu_l}([\x,i,j]) + \frac{nd^2m(2s_{\spincu_l}-n)}{2k(k+md)} \\
&= \max\{i-n, j-s_l\} + \frac{nd(2s_l - n)}{2(k+md)}  + \frac12 + \frac{nd^2m(2s_l-n)}{2k(k+md)} \\
&= \max\{i-n, j-s_l\} + \frac{nds_l +nk-n^2 d}{2k}
\end{align*}
as required.
\item
On each summand $\CF^t(\bm\alpha, \bm\beta, \spincs_l) \otimes T^{-s_l}$, according to \eqref{eq: ab-twisted-filt} we have
\begin{align*}
\JJ([\x, i, j]) &= \JJ_{\alpha\beta}([\x, i] \otimes T^{-s_l}) \\
&= i - \frac{2nd(-s_l) + nk + n^2 d}{2k} \\
&= i-n + \frac{2nds_l + nk - n^2 d}{2k}
\end{align*}
as required.
\end{itemize}

We have proved that $\CFK^t(Y_\lambda, K_{n,\lambda}, \spinct)$ is doubly-filtered quasi-isomorphic to  $X^t_{\lambda,\spinct,n,-L,L}$. According to Lemma \ref{lemma: cft}, this implies the doubly-filtered quasi-isomorphism between   $\CFKm(Y_\lambda, K_{n,\lambda}, \spinct)$ and $X^-_{\lambda,\spinct,n,-L,L}$. To pass to the infinity version, tensor both complexes by $\F[U,U^{-1}]$, and we conclude that $\CFKi(Y_\lambda, K_{n,\lambda}, \spinct)$ is doubly-filtered quasi-isomorphic to $X^\infty_{\lambda,\spinct,n,-L,L}$, as required.

\end{proof}

\section{Rational surgery} \label{sec: rational}

Following the same approach spelled out in \cite[Section 8]{HL}, our formula also supports a generalization to the rational surgery. For the simplicity, in this section we only demonstrate the process for $1/p$ surgery on a null-homologous knot $K$ inside a rational homology sphere $Y$, where $p$ is a positive integer. But this can be generalized to the case of arbitrary rational surgeries on any knot $K$ in a rational homology sphere without too much trouble.

Let $K_{n,1/p}$ denote the knot in $Y_{1/p}(K)$ obtained from $(n,1)$--cable of a left-handed meridian of $K$. Note that the left-handed meridian is no longer isotopic to the surgery solid torus as in the case of integral surgery.  Following Hedden and Levine's notation, let $Y'$ denote $Y\conn -L(p,1)$ and $K'=K\conn O_p$, where $O_p \subset -L(p,1)$ is obtained from the Hopf link by performing a $-p$ surgery on one of the unknot components. Then   $Y_{1/p}(K)$ is obtained from $Y'$ via a $2$-handle cobordism along $K';$ let $W$ denote this cobordism.

For $q\in \{0, \cdots, p-1 \}$, the complex $\HFKa(-L(p,1), O_p, \spincu_q)$ is generated by a single generator in Alexander grading $-\frac{p-2q-1}{2p}$. For any $\spincs \in \Spin^c(Y)$, let  $\spincs_{q} = \spincs \conn \spincu_q \in \Spin^c(Y \conn -L(p,1))$. The K\"unneth principle for connected sums (see \cite[Theorem 7.1]{OSknot}) implies that $\CFKi(Y',K', \spincs_q)$ is isomorphic to $\CFKi(Y,K)$, with the Alexander grading shifted by $-\frac{p-2q-1}{2p}$. Moreover, all the  $\spincs_{q}$ are cobordant to the same spin$^c$ structure on $Y_{1/p}(K)$ through $W$; denoted this spin$^c$ structure by $\spinct \in \Spin^c(Y_{1/p}(K))$.
 
In order to compute $\CFKi(Y_{1/p}(K),K_{n,1/p},\spinct),$ we apply the filtered mapping cone formula to $(Y',K')$. Using the notation from Section \ref{ssec: statement}, we take $d=p$ and the framing on $K'$ corresponds to $k=1$. Consider all the spin$^c$ structures on $W$ that extends $\spinct$. Through the bijection between $\Spin^c(W)$ and $\ul\Spin^c(Y',K'),$ they form a sequence $\xi_l \in \ul\Spin^c(Y',K')$ for $l\in \Z$, where
\[
\frac{2l-1}{2p} < \AlNorm_{Y',K'}(\xi_l) \le \frac{2l+1}{2p}.
\]
We set $s_l = \AlNorm_{Y',K'}(\xi_l)$. Note that each $\AlNorm_{Y',K'}(\xi_l)$ satisfies 
\[
\AlNorm_{Y',K'}(\xi_l) \equiv \frac{-p+2q+1}{2p} \pmod \Z,
\]
for some $q=0,1,\cdots,p-1.$ So we can write
\[
s_l = \frac{-p+2q+1}{2p} + r
\]
for some $r \in \Z.$ The above arithmetic constraint is enough to pin down $s_l$ and $r$, as demonstrated by the following computation from \cite{HL}: Observe $2l-1 < (2r-1)p + 2q+1 \le 2l+1$, so we have $(2r-1)p \le 2(l-q) < (2r-1)p + 2$. There are two cases:
\begin{enumerate}[label=(\roman*)]
    \item If $p$ is even, we deduce that $2(l-q) = (2r-1)p$, which implies
    \[
    s_l = \frac{2l+1}{2p};
    \]
    \item if $p$ is odd, then $2(l-q) = (2r-1)p+1$, and therefore 
    \[
    s_l= \frac {l}{p}. 
    \]
\end{enumerate}
In both cases  we have
\[
r=\floor {\frac{2l+p}{2p}}.
\]

In the mapping cone, the complexes $A^\infty_{\xi_l}$ and $B^\infty_{\xi_l}$ are each copies of $\CFKi(Y',K', \spincs_q)$ by definition. As noted above, each of these complexes is in turn isomorphic to $\CFKi(Y,K,\spincs)$, but with the $j$ coordinate shifted by $-\frac{p-2q-1}{2p}$. Now we can directly compute the filtrations on $A^\infty_{\xi_l}$ and $B^\infty_{\xi_l}$ using the formulas \eqref{eq: It-def-A}, \eqref{eq: Jt-def-A}, \eqref{eq: It-def-B}, and \eqref{eq: Jt-def-B}. We collect the results as follows.
\begin{align}
\intertext{On $A^\infty_{\xi_l}$,}
\label{eq: It-ratl-A} \II_\spinct([\x,i,j]) &= \max\{i,j-r\} \\
\label{eq: Jt-ratl-A} \JJ_\spinct([\x,i,j]) &= \max\{i-n,j-r\} + nl - 
\begin{cases}
\frac{n}{2}(np-1)  \quad p \text{ odd},  \\
\frac{n}{2}(np-2)  \quad p \text{ even},  \\
\end{cases}
\intertext{On $B^\infty_{\xi_l}$,}
\label{eq: It-ratl-B} \II_\spinct([\x,i,j]) &= i \\
\label{eq: Jt-ratl-B} \JJ_\spinct([\x,i,j]) &= i-n + nl - 
\begin{cases}
\frac{n}{2}(np-1)  \quad p \text{ odd},  \\
\frac{n}{2}(np-2)  \quad p \text{ even}.  \\
\end{cases}
\end{align}
In particular, the above formula takes the same form as the mapping cone defined by Ozsv\'ath and Szab\'o, namely each $A_s$ and $B_s$ complex appears $p$ times in the mapping cone (though in our case the $\JJ_t$ filtration of each copy is shifted accordingly).

\section{Examples and applications}\label{sec: examples}

In this section, we compute a series of examples that lead us to the proof of Proposition \ref{prop: phi} and other applications. For the reader's convenience, we restate the mapping cone formula for any knot  $K\subset S^3$ under the $+1$-surgery, where the $(n,1)$--cable of the meridian is denoted by $K_{n,1}$. Take $k=d=1$, and $s_l=s \in \Z$ in this case, since the Alexander grading for any relative spin$^c$ structure of $K$ is an integer. 

According to the main theorem, the knot floer complex  $\CFKi(S^3_{1}(K),K_{n,1})$ is given by $X^\infty_{1,n,a,b}(K)$ for some $a\ll 0$ and $b\gg 0$. Specifically, let us consider the  mapping cone $X^\infty_{1,n,-g+1,g+n-1}(K)$ (we will see this gives suitable values for $a$ and $b$ later): 
\begin{align}
    \bigoplus^{g+n-1}_{s=-g+1}A^\infty_s \xrightarrow{v^\infty_s+h^\infty_s} \bigoplus^{g+n-1}_{s=-g+2}B^\infty_s,
\end{align}
where each $A^\infty_s$ and $B^\infty_s$ is isomorphic to $\CFKi(S^3,K)$, the map $v^\infty_s\co A^\infty_s \to B^\infty_s $ is the identity and $h^\infty_s\co A^\infty_s \to B^\infty_{s+1} $ is the reflection map precomposed with $U^s$. The $\II$ and  $\JJ$  filtrations are given by
\begin{align}
\intertext{For $[\x,i,j] \in A^\infty_{s}$,}
\label{eq: filtration_s3_1}
 \II([\x,i,j]) &= \max\{i,j-s\} \\
 \label{eq: filtration_s3_2}
 \JJ([\x,i,j]) &= \max\{i-n,j-s\} + ns - \frac{n(n-1)}{2} \\
\intertext{For $[\x,i,j] \in B^\infty_{s}$,}
 \label{eq: filtration_s3_3}
 \II([\x,i,j]) &= i \\
  \label{eq: filtration_s3_4}
 \JJ([\x,i,j]) &= i-n + ns - \frac{n(n-1) }{2}.
\end{align}
 It is straightforward to check that for $s<-g+1$, the map $h^\infty_s$ induces an isomorphism on the homology; for $s>g+n-1,$ 
the map $v^\infty_s(K)$ induces an isomorphism on the homology. Thus the filtered quasi-isomorphism type of $X^\infty_{1,n,a,b}(K)$ does not depend on $a,b$ if $a\leq -g+1$ and $b\geq g+n-1$. We will write 
\[
X_n^\infty(K) = X^\infty_{1,n,-g+1,g+n-1}(K). 
\]
For the purpose of computations,   looking at the associated graded complex is usually helpful. We will describe the  associated graded complex with $\II=0$, and the associated graded complex with other  $\II$ values can be obtained from translations. The complex $X_n^\infty (K)\{i=0\}$ is given by 
\begin{align}
    \bigoplus^{g+n-1}_{s=-g+1}A^\infty_s \{\text{max}(i,j-s)=0 \} \xrightarrow{v^\infty_s+h^\infty_s} \bigoplus^{g+n-1}_{s=-g+2}B^\infty_s\{i=0\},
\end{align}
where each $A^\infty_s \{\text{max}(i,j-s)=0 \}$ admits exactly $n+1$ filtration levels with respect to $\JJ$, namely $\JJ= ns+\frac{n(n-1)}{2}, 1+ns+\frac{n(n-1)}{2},\cdots, n+ns+\frac{n(n-1)}{2}, $ while   $B^\infty_s\{i=0\}$ corresponds to the filtration level $-n+ns+\frac{n(n-1)}{2}.$ See Figure \ref{fig: filtration}.

\begin{figure}
    \labellist

\pinlabel $ns+n(n-1)/2$  at 100 290
\pinlabel $-1+ns+n(n-1)/2$  at 193 268
\pinlabel $\cdot$  at 143 249
\pinlabel $\cdot$  at 143 245
\pinlabel $\cdot$  at 143 241
\pinlabel $-n+ns+n(n-1)/2$  at 193 215

\pinlabel $A_{s-1}$  at 110 314

\pinlabel $h_{s-1}$  at 200 164
\pinlabel $v_s$  at 292 164

\pinlabel $n+ns+n(n-1)/2$  at 280 288
\pinlabel $n-1+ns+n(n-1)/2$  at 382 268
\pinlabel $\cdot$  at 330 248
\pinlabel $\cdot$  at 330 244
\pinlabel $\cdot$  at 330 240
\pinlabel $ns+n(n-1)/2$  at 389 215

\pinlabel $A_s$  at 290 314

\pinlabel $ns+n(n-1)/2$  at 200 100
\pinlabel $B_s$  at 290 110

\endlabellist
    \includegraphics{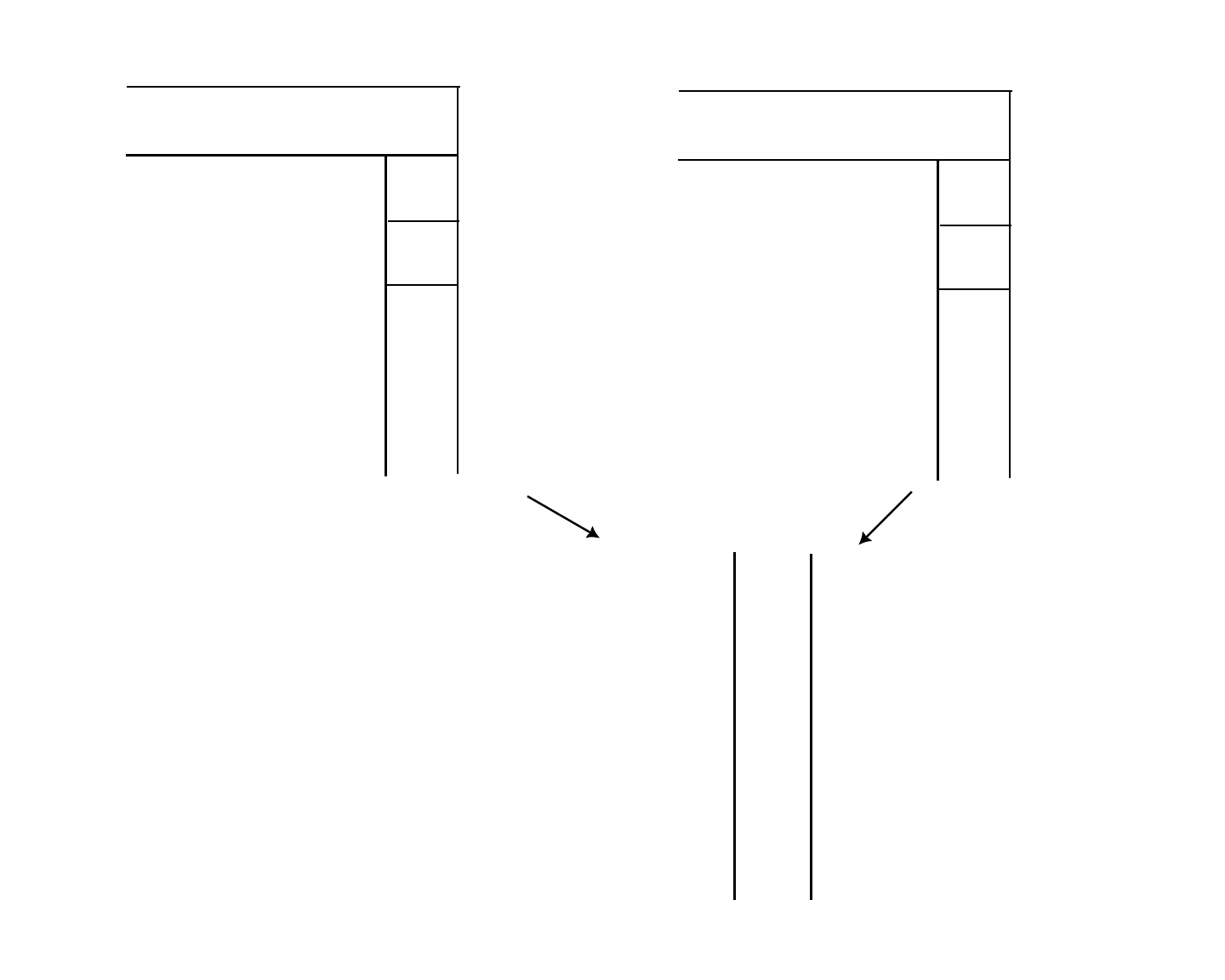}
    \caption{A portion of the complex $X_n^\infty(K) \{i=0\}$, where the vertical bar of each complex is at $i=0$ and the horizontal bar of $A_{s}$ (resp.~ $A_{s-1}$) is at $j=s$ (resp.~ $j=s-1$).  }
    \label{fig: filtration}
\end{figure}
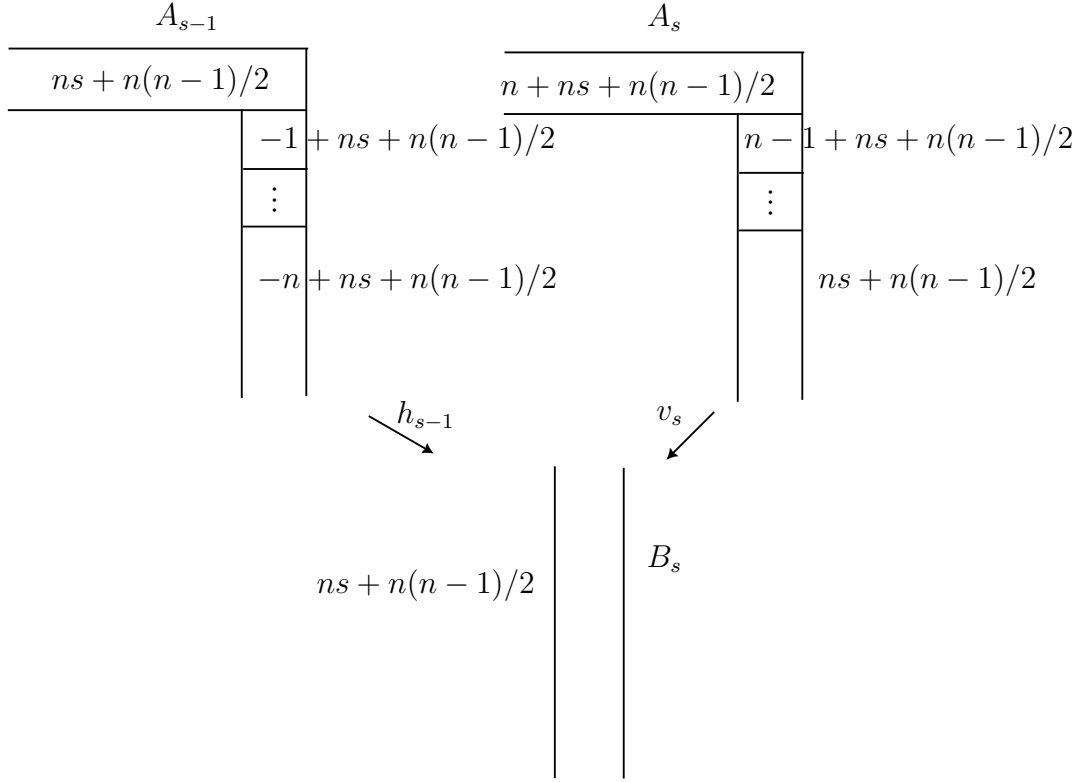

The general strategy for computation involves finding a \emph{reduced} basis for $X^\infty_n(K),$ where every term in the differential strictly lowers at least one of the filtrations. This can be achieved through a cancellation process (see for example \cite{Bordered}) as follows: suppose $\partial x_i = y_i$ $+ $ lower filtration terms, where the double filtration of $y_i$ is the same as $x_i$, then the subcomplex of  $X^\infty_n(K)$ generated by all such $\{x_i,  \partial x_i\}$ is acyclic, and $X^\infty_n(K)$ quotient by this complex is reduced.

In the rest of the section, we wish to shed a light on the richness of examples coming from the new filtered mapping cone formula. But at first we have to go on a detour for some preliminaries of Heegaard Floer concordance invariants. The experts can skip the next subsection.  
\subsection{Concordance invariants and other rings}\label{ssec: ring}
The concordance invariants in our interests are defined using Heegaard Floer homology over other rings. Pointing the reader to the sources, we would not go in depth about the detailed construction. Instead, we will review some results that are useful to us and use an example to explain some of the intuition. 

To start, Ian Zemke provided a reinterpretation of Heegaard Floer homology over the ring $\F[U,V]$ (see for example \cite{zemkegrading}).   In the knot Floer complex over $\F[U,V]$,  there is a bigrading $\gr=(\gr_U,\gr_V)$ assigned to each generator, where $\gr(U)=(-2,0)$ and $\gr(V)=(-2,0)$.  Two knot Floer complexes are \emph{locally equivalent} over $\F[U,V]$ if the two knots are (smoothly) concordant (\cite[Section 2.3]{ianconnect}). The local equivalence takes on slightly different meanings when defined over different rings, but is defined to always respect the knot concordance. We are generally interested in, within different contexts,  the morphism between the monoid of knots over concordance and the set of knot Floer complexes over the local equivalence relation.  

No matter over which ring it is defined, the knot Floer complex always possesses certain symmetric property. For example, over the ring  $\F[U,V]$ this entails that after exchanging the role of $U$ and $V$, the resulting complex is homotopy equivalent to the original knot Floer complex.

In \cite{Moreconcor},  Dai-Hom-Stoffregen-Truong studied  knot Floer complexes over the ring $\F[U,V]/(UV).$ The results were initially for knots in $S^3,$ but easily generalizable to knots in integer homology sphere $L$--spaces. According to \cite[Definition 4.3]{Moreconcor}, a \emph{standard complex} over $\F[U,V]/(UV)$ is freely generated by $\{x_0, x_1, \cdots, x_{2l} \}$ for some $l\in \N,$ where each pair of generators $x_{2i}$ and $x_{2i+1}$ is connected by some ``U-arrow''  and each pair of generators $x_{2i+1}$ and $x_{2i+2}$ is connected by some ``V-arrow''  for $i=0,\cdots, l-1$. All information of a standard complex can be encoded using a sequence of signed integers describing  the length and the direction of the arrows in order. Thus we shall use this sequence to indicate a standard complex. It turns out that the complexes over the ring   $\F[U,V]/(UV)$ are quite nice. Indeed, we have
\begin{theorem*}[Theorem 6.1 in \cite{Moreconcor}]
The knot Floer complex of any knot in an integer homology sphere $L$--space is locally equivalent to a unique standard complex over $\F[U,V]/(UV).$
\end{theorem*}
Therefore, the sequence of signed integers is a concordance invariant through this identification. In \cite[Definition 7.1]{Moreconcor}, the invariant $\varphi_i$ is defined to be a signed count of the times number $i$ appear in the sequence in odd position, and  subsequently  in \cite[Theorem 7.2]{Moreconcor}, is shown to be a homomorphism from the knot concordance group $\mathcal{C}$ to $\Z$. Equivalently we can define $\varphi_i$ by counting only in even position (inverting the sign for each integer since we want the two definitions to be consistent); either way only half of the sequence is needed due to the symmetry. The two definitions correspond respectively to  counting all the ``U-arrows'' or  ``V-arrows'' of a certain length.

 Setting $U=0$ in the ring $\F[U,V]/(UV)$ makes the knot Floer complex a module over $\F[V]$; call the homology of the resulting complex the \emph{vertical homology.} Since $\F[V]$ is a PID,   the vertical homology of any knot Floer complex  has the decomposition
 \begin{align*}
     \F[V] \oplus \big( \bigoplus_j \F[V]/(V^j)  \big).
 \end{align*}
A basis on the knot Floer complex that realizes this decomposition is called a \emph{vertical basis;} namely, a vertical basis consists of generators that pairwise generate the  $V$--torsions and a standalone generator (i.e. in the kernel but not the image of $\partial$) that generates the single copy of $\F[V]$. Furthermore by setting $V=1,$ the resulting complex computes $\HFa(Y) \cong \F,$ where $Y$ is the ambient manifold. Therefore the  standalone generator in a vertical basis has $\gr_V=d(Y),$ or we say the $d$--invariant is supported in said generator.

In \cite{Homoconcor}, the above construction is generalized by the same group of authors. They devised the following ring 
\begin{align*}
    \X=\frac{\F[U_B,\{W_{B,i}\}_{i\in \Z},V_T,\{W_{T,i}\}_{i\in \Z}]}{(U_B V_T, U_B W_{B,i}-W_{B,i+1},V_T W_{T,i}-W_{T,i+1})}
\end{align*}
with a bigrading $(\gr_1,\gr_2) \in \Z \times \Z,$ where
\begin{gather*}
    \begin{aligned}
     \hspace{8em} \gr(U_B)=(-2,0) &\hspace{2em} \text{and} \hspace{2em}&  \gr(W_{B,i})=(-2i,-2) &\hspace{3em}
      \shortintertext{and}
      \gr(V_T)=(0,-2) &\hspace{2em} \text{and} \hspace{2em}&  \gr(W_{T,i})=(-2,-2i) &\hspace{3em}.
    \end{aligned}
\end{gather*}
Inside $\X$ there are two subrings $\mathcal{R}_U$ and  $\mathcal{R}_V$  given by
\begin{align*}
    \mathcal{R}_U &= \F[U_B,\{W_{B,i}\}_{i\in \Z}]/(U_B W_{B,i}-W_{B,i+1})
    \intertext{and}
     \mathcal{R}_V &= \F[V_T,\{W_{T,i}\}_{i\in \Z}]/(V_T W_{T,i}-W_{T,i+1}).
\end{align*}
Note that elements in $\mathcal{R}_U$ can be uniquely written down as $U^i_B W^j_{B,0},$ where 
\begin{align*}
    (i,j)\in (\Z \times \Z^{\geq 0}) - (\Z^{<0} \times \{0\}).
\end{align*}
Namely, $i$ (the power of $U_B$) can be negative when $j$ (the power of $W_{B,0}$) is positive, but is required to be non-negative when $j=0.$  This is analogous to the ring $\F[U,V]$ localized at the maximal ideal $(U)$, denoted by $\F[U,U^{-1},V]$. Here $U_B$ plays the role of $U$ and $W_{B,0}$ plays the role of $V$, such that we are allow to invert $U$ but not $V$. However, a difference is that in $\mathcal{R}_U$  when there is no $V$ power (when $j=0$) we ``remember'' the power of $U$ ($i$ needs to be non-negative).

Similarly, the elements in $\mathcal{R}_V$ can be written down  uniquely as well. Whenever talking about the elements in $\mathcal{R}_U$ and $\mathcal{R}_V$, we canonically use this expression.

The complexes over $\F[U,V]$ can be translated to $\X$ using the map
\begin{align*}
    U&\xmapsto[\hspace{2em}]{} U_B + W_{T,0}\\
    V&\xmapsto[\hspace{2em}]{} V_T + W_{B,0},
\end{align*}
so this allows us to talk about knot Floer complexes over the ring $\X.$

Similar to the case of $\F[U,V]/(UV)$, there is the notion of the standard complex over $\X$ (\cite[Definition 5.1]{Homoconcor}). A standard complex is freely generated by $\{x_0, x_1, \cdots, x_{2l} \}$ for some $l\in \N,$ where each pair of generators $x_{2i}$ and $x_{2i+1}$ is connected by a differential using elements in $\mathcal{R}_U$  and each pair of generators $x_{2i+1}$ and $x_{2i+2}$ is connected by a differential using elements in $\mathcal{R}_V$,  for $i=0,\cdots, l-1$. Again, quite remarkably, we have the following. 
\begin{theorem*}[Theorem 7.1 in \cite{Homoconcor}]
Every knot Floer complex over the ring $\X$ is locally equivalent to a unique standard complex.
\end{theorem*}
Therefore one can  define the concordance invariants $\varphi_{i,j}$ by either counting  $\mathcal{R}_U$  edges or $\mathcal{R}_V$  edges in the standard complex that is locally equivalent to the given complex (\cite[Definition 8.1]{Homoconcor}). For $(i,j)\in (\Z \times \Z^{\geq 0}) - (\Z^{<0} \times \{0\})$, $\varphi_{i,j}$ are indeed homomorphisms according to \cite[Theorem 8.2]{Homoconcor}. When $j=0$, the map $\varphi_{i,0}$ agrees with the previously defined homomorphism $\varphi_i$, and when $j \neq 0$, all $\varphi_{i,j}$ vanish for any knot in an integer homology sphere $L$--space.

Note that for a knot Floer complex over $\X$, suppose we set $V_{T}=1$. This forces $\mathcal{R}_U=0$ due to the relation $U_B V_T=0$. The resulting complex is a graded module over $\F[W_{T,0}]$, and once again we obtain that the $d$--invariant is supported in the non-torsion element in this complex. 

We hope that the following example provides some intuition in dealing with complexes over the ring $\X.$

\begin{figure}
\begin{minipage}{.5\linewidth}
     \centering
      \subfloat[The complex $C$ in Example \ref{example:C}. ]{
       \begin{tikzpicture}
      \begin{scope}[thin, black!20!white]
    \foreach \i in {-2,...,3}
     {\draw  (\i-0.5, 3) -- (\i-0.5, -3);
	  \draw  (3, \i-0.5) -- (-3, \i-0.5); }
      \end{scope}
      	\filldraw (2, 0) circle (2pt) node[] (x1){};
         \filldraw (2, 1) circle (2pt) node[] (y2){};
      	\filldraw (2, 2) circle (2pt) node[] (z){};
      	\filldraw (1, 2) circle (2pt) node[] (x2){};
      	\filldraw (0, 2) circle (2pt) node[] (y1){};
      	\filldraw (-1, -2) circle (2pt) node[] (w1){};
      	\filldraw (-2, -1) circle (2pt) node[] (w2){};
 
     \draw  (2,2) -- (2,1);
      \draw  (2,2) -- (1,2);
       \draw  (2,1) -- (-2,-1);
          \draw  (2,1) -- (-1,-2);
        \draw  (1,2) -- (-1,-2);
        \draw  (1,2) -- (-2,-1);
        \draw  (0,2) -- (-2,-1);
         \draw  (2,0) -- (-1,-2);
      
      	\node  [right] at (x1) {$x_1$};
      	 	\node  [right] at (z) {$z$};
      	 	 	\node  [right] at (y2) {$y_2$};
      	 	 		\node  [above] at (y1) {$y_1$};
      	 	 		\node  [above] at (x2) {$x_2$};
      	 	 			\node  [below] at (w1) {$w_1$};
      	 	 				\node  [below] at (w2) {$w_2$};
    \end{tikzpicture}
    }
\end{minipage}%
\begin{minipage}{.5\linewidth}
     \centering
      \subfloat[{The  complex $C$ after a change of basis in $\F[U,V,V^{-1}]$.}]{
       \begin{tikzpicture}[scale=0.9]
       \begin{scope}[thin, black!0!white]
	  \draw  (5, 0) -- (-5,0);
      \end{scope}
      \begin{scope}[thin, black!20!white]
    \foreach \i in {-2,...,3}
     {\draw  (\i-0.5, 3) -- (\i-0.5, -3);
	  \draw  (3, \i-0.5) -- (-3, \i-0.5); }
      \end{scope}
      	\filldraw (2, 0) circle (2pt) node[] (x1){};
         \filldraw (2, 1) circle (2pt) node[] (y2){};
      	\filldraw (2, 2) circle (2pt) node[] (z){};
      	\filldraw (1, 2) circle (2pt) node[] (x2){};
      	\filldraw (0, 2) circle (2pt) node[] (y1){};
      	\filldraw (-1, -2) circle (2pt) node[] (w1){};
      	\filldraw (-2, -1) circle (2pt) node[] (w2){};
 
     \draw  (2,2) -- (2,1);
  
        \draw  (1,2) -- (-1,-2);
       
        \draw  (0,2) -- (-2,-1);
 \tiny{

      	 	\node  [right] at (z) {$z$};
      	 	 	\node  [right] at (y2) {$y_2+UV^{-1}x_2$};
      	 	 		\node  [left] at (y1) {$y_1$};
      	 	 		\node  [above] at (x2) {$x_2+Uy_1$};
      	 	 			\node  [below] at (w1) {$w_1$};
      	 	 				\node  [below] at (w2) {$w_2$};
     \node  at (2.1, -0.25) {$x_1+UV^{-1}x_2+U^2V^{-2}y_1$};
     
     }    
    \end{tikzpicture}
     }
\end{minipage}
    \caption{}
    \label{fig: complex_example}
\end{figure}
\begin{example} \label{example:C}
Consider the following complex $C$ over $\F[U,V]$ generated by $x_1,x_2,y_1,y_2,$ $w_1,w_2$ and $z$ with differentials as follows (see Figure \ref{fig: complex_example})
\begin{gather*}
    \begin{aligned}
      \partial x_1 &= U^3V^2w_1  \hspace{4em}   &\partial x_2 = U^2V^4w_1 + U^3V^3w_2 \\
     \partial y_1 &= U^2V^3w_2              &\partial y_2 = U^3V^3w_1 + U^4V^2w_2\\
      \partial z &= Ux_2 + Vy_2. &
    \end{aligned}
\end{gather*}
In fact, $C$ is locally equivalent to the complex $CFK(S^3_1(T_{2,11}),(T_{2,11})_{2,1}).$ (This follows from Lemma \ref{le: localequi} and the proof of Proposition \ref{prop: phi}.) We aim to compute $\varphi_{i,j}(C)$ for some suitable $i$ and $j,$ which necessitates studying $C$ over the ring $\X.$ To provide some intuition, we would like to first examine $C$ over $\F[U,V]$.

Inverting $V$, in the localized ring  $\F[U,V,V^{-1}]$, we can perform a change of basis by replacing $y_2$ with $y_2+UV^{-1}x_2.$ This simplifies the differentials:
\begin{gather}\label{eq: example_middlestep_changebasis}
    \begin{aligned}
          \partial x_1 &= U^3V^2w_1  \hspace{4em} &\partial x_2 = U^2V^4w_1 + U^3V^3w_2 \\
     \partial y_1 &= U^2V^3w_2      &\partial z = V(y_2 + UV^{-1}x_2 ).  \hspace{0.8em}
    \end{aligned}
\end{gather}
Observing that $x_1, x_2$ and $y_1$  and their differentials make up  ``a connected series of line segments", we further perform the following change of basis:
\begin{align*}
    x_1 &\xmapsto[\hspace{1em}]{} x_1 + UV^{-2}x_2 + U^2V^{-2}y_1\\
    x_2 &\xmapsto[\hspace{1em}]{} x_2 + Uy_1.
\end{align*}
Under this change of basis, (the image of) $y_1$ and $w_2$, $x_2$ and $w_1$, $z$ and $y_2$ pairwise generate the torsion in the homology while the standalone generator $x_1$  generates the single copy of the base ring, realizing the decomposition 
\[
H_*(C\otimes \F[U,V,V^{-1}]) \cong \F[U,V,V^{-1}] \oplus \big(\bigoplus_j \F[U,V,V^{-1}]/(U^j)\big).
\]
If we instead invert $U$, then one could also write down a similar change of basis over the ring $\F[U,U^{-1},V],$ such that there is one standalone generator while all other generators appear in pairs forming the torsion in the homology. Indeed, 
\begin{gather*}
\begin{aligned}
     y_1 &\xmapsto[\hspace{1em}]{} y_1 + U^{-2}Vy_2 + U^{-2}V^2x_1,  & \hspace{2em}
    y_2 \xmapsto[\hspace{1em}]{} y_2 + Vy_1, \\
    x_2 &\xmapsto[\hspace{1em}]{} x_2 + U^{-1}Vy_2  &
\end{aligned}
\end{gather*}
    accomplishes the job.
   
Next, denote by $C_\X$ the complex $C$ over the ring $\X.$ Translating using the language of ring $\X$, the differentials become 
\begin{align*}
    \partial x_1 &= (U_B^3 W_{B,0}^2 + V_T^2W_{T,0}^3)w_1\\
    \partial x_2 &= (U_B^2 W_{B,0}^4 + V_T^4 W_{T,0}^2 )w_1 + (U_B^3W_{B,0}^3 + V_T^3 W_{T,0}^3  )w_2 \\
    \partial y_1 &= (U_B^2 W_{B,0}^3 + V_T^3 W_{T,0}^2  )w_2\\
     \partial y_2 &= (U_B^3 W_{B,0}^3 + V_T^3 W_{T,0}^3 )w_1 + (U_B^4W_{B,0}^2 + V_T^2 W_{T,0}^4) w_2\\
      \partial z &= (U_B + W_{T,0}) x_2 + (W_{B,0} + V_T ) y_2.
\end{align*}
 We wish to formulate a change of basis, under which $C_\X$ is a standard complex. To motivate such a change of basis, consider the following. Inside the differential of each generator, there are terms in the subring $\mathcal{R}_V$ and terms in the subring $\mathcal{R}_U.$ As we have seen, the ring $\mathcal{R}_V$ is the analogue of inverting $V$ inside $\F[U,V]$ and the ring $\mathcal{R}_U$ is the analog of inverting $U$ inside $\F[U,V]$. If we ignore all the terms in $\mathcal{R}_U$ and focus on the terms in $\mathcal{R}_V$, there clearly exists a set of basis analogous to the case of $\F[U,V,V^{-1}]$. Similar for $\mathcal{R}_U$. Thus the intuition would be to perform the two sets of the change of basis simultaneously.
 
 First, if we set 
 \begin{gather*}
\begin{aligned}
      \hspace{2.5em} \widetilde{x}_2 &= x_2 + U^{-1}_B W_{B,0} y_2    \hspace{5.5em} &\widetilde{y}_2 &= y_2 + V^{-1}_T W_{T,0} x_2,
\end{aligned}
\end{gather*}
the differentials can be simplified to
 \begin{gather*}
\begin{aligned}
  \hspace{4em}    \partial x_1 &= (U_B^3 W_{B,0}^2 + V_T^2W_{T,0}^3)w_1 \hspace{2em}    &\partial \widetilde{x}_2 &=  V_T^4 W_{T,0}^2 w_1 +  V_T^3 W_{T,0}^3  w_2 \hspace{0.3em}\\
       \partial y_1 &= (U_B^2 W_{B,0}^3 + V_T^3 W_{T,0}^2  )w_2      &\partial \widetilde{y}_2 &= U_B^3 W_{B,0}^3 w_1 + U_B^4W_{B,0}^2  w_2\\
        \partial z &= U_B \widetilde{x}_2 + V_T  \widetilde{y}_2.
\end{aligned}
\end{gather*}
Compare the $\mathcal{R}_V$ terms in the differential of each generator with the differentials in \eqref{eq: example_middlestep_changebasis} and notice the similarity.   With this observation in mind, we further perform the following change of basis
 \begin{gather*}
\begin{aligned}
     \hspace{6em}\widehat{x}_1 &= x_1 + V_T^{-2}W_{T,0} \widetilde{x}_2 + V_T^{-2}W^{2}_{T,0} y_1 \hspace{1em}
    &\widehat{x}_2 &= \widetilde{x}_2 + W_{T,0} y_1 \hspace{0.3em}\\
     \widehat{y}_1 &= y_1 + U_B^{-2}W_{B,0} \widetilde{y}_2 + U_B^{-2}W^{2}_{B,0} x_1
    &\widehat{y}_2 &= \widetilde{y}_2 + W_{B,0} x_1,
\shortintertext{making the differentials become}
      \partial \widehat{x}_1 &= U_B^3 W_{B,0}^2 w_1  &\partial \widehat{x}_2 &=  V_T^4 W_{T,0}^2 w_1  \\
       \partial  \widehat{y}_1 &=  V_T^3 W_{T,0}^2 w_2    &\partial \widetilde{y}_2 &=  U_B^4W_{B,0}^2  w_2\\
        \partial z &= U_B \widehat{x}_2 + V_T  \widehat{y}_2.
\end{aligned}
\end{gather*}
We conclude that $C_\X$ under this final basis is a standard complex given by
\[
(-(3,2),(4,2),(1,0),-(1,0),-(4,2),(3,2)).
\]
In particular, we have
\begin{align*}
    \varphi_{i,j}(C)=\begin{cases} -1 \quad &\text{if} \hspace{0.5em} (i,j)=(3,2) \text{ or }(4,2)\\
    1 &\text{if} \hspace{0.5em} (i,j)=(1,0)\\
    0 &\text{otherwise.}
    \end{cases}
\end{align*}
\end{example}

In the following subsection, we will perform most of the computations (especially when we are invoking the filtered mapping cone formula) over the original ring $\F[U,U^{-1}].$ Only after collecting enough information and ready to compute the concordance invariants, we then move on to either the ring $\F[U,V]/(UV)$ or $\X$ depending on the case at hand. We will always specify the ring we are working on whenever the base ring is not  $\F[U,U^{-1}]$ or $\F[U]$. It is our hope that the current approach causes minimal confusion.

\subsection{Examples and applications} \label{ssec: examples}

We begin by describing a family of concrete examples by applying the filtered mapping formula to the $+1$-surgery on the torus knot $T_{2,3}$  with varying $n$.


Using the notion of the standard complex from either \cite{Homoconcor} or \cite{Moreconcor}, we have 
\begin{proposition} \label{prop: t23_n1}
 The complex $\CFKi(S^3_1(T_{2,3}),(T_{2,3})_{n,1})$ is locally equivalent to the standard complex $(-1,n,1,-1,-2,\cdots,n-i+1,1,-1,-i-1,\cdots,2,1,-1,-n,1)$ for some fixed $n\geq 1$, where $i=1,\cdots, n-1.$
\end{proposition}

The standard complex described in the proposition is simply the concatenation of $(n-i+1,1,-1,-i-1)$ for $i=1,\cdots,n-1$ in that order, with $-1$ and $1$ in the front and back.

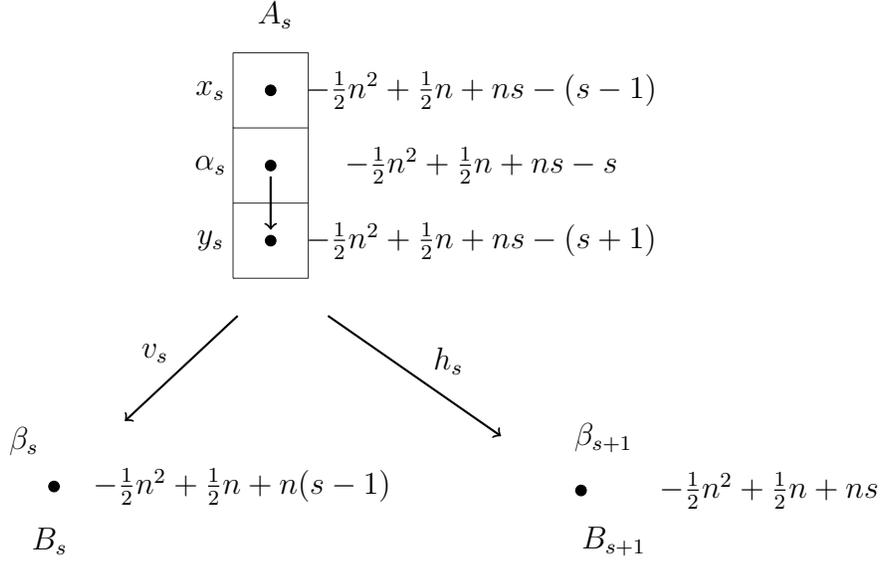
\begin{figure}

\begin{tikzpicture}

	\node   at (-2,6) {$A_s$};
	
\node[]() at (0,4) {

\begin{tikzpicture}
	\draw[step=1, black!80!white, thin] (0, 0) grid (1, 3);

	\filldraw (0.5, 2.5) circle (2pt) node[] (b){};
		\filldraw (0.5, 1.5) circle (2pt) node[] (a){};
		\filldraw (0.5, 0.5) circle (2pt) node[] (c){};

	\draw [thick, ->] (a) -- (c);

	\node   at (-0.3,1.5) {$\alpha_s$};
		\node   at (-0.3,2.5) {$x_s$};
			\node   at (-0.3,0.5) {$y_s$};
			
				\node   at (3.3,1.5) {$-\frac{1}{2}n^2  + \frac{1}{2}n + ns -s$};
		\node   at (3.3,2.5) {$-\frac{1}{2}n^2  + \frac{1}{2}n + ns -(s-1)$};
			\node   at (3.3,0.5) {$-\frac{1}{2}n^2  + \frac{1}{2}n + ns -(s+1)$};

\end{tikzpicture}};

	\node   at (-5,-1) {$B_s$};
\node[]() at (-3,0) {
\begin{tikzpicture}
\filldraw (0, 0) circle (2pt) node[] {};
	\node   at (-0.4,0.6) {$\beta_s$};
		\node   at (2.5,0) {$-\frac{1}{2}n^2  + \frac{1}{2}n + n(s -1)$};
\end{tikzpicture}};

	\node   at (2.5,-1) {$B_{s+1}$};
\node[]() at (4,0) {
\begin{tikzpicture}
\filldraw (0, 0) circle (2pt) node[] {};
	\node   at (0.3,0.7) {$\beta_{s+1}$};
		\node   at (2.5,0) {$-\frac{1}{2}n^2  + \frac{1}{2}n + ns$};
\end{tikzpicture}};

	\draw [thick, ->] (-2.5,2) -- (-4,0.6);
	\draw [thick, ->] (-1.3,2) -- (1,0.4);

	\node   at (-3.6,1.5) {$v_s$};
	\node   at (0.3,1.4) {$h_s$};

\end{tikzpicture}
\caption{A part of the complex $X^\infty_n(T_{2,3})\{i=0\}$ under the chosen basis.}
\label{fig: portion1}

\end{figure}

\begin{proof}
The case of $n=1$ follows from a simple calculation. Now consider when $n>1.$

Let us denote the generators of $\CFKi(S^3,T_{2,3})$ by $a,b$ and $c$ with coordinates $(0,0),(0,1)$ and $(0,-1)$ respectively, where $b$ has Maslov grading $0$. They satisfy
\[
\partial a=Ub + c.
\]
Via the isomorphism with $\CFKi(S^3,T_{2,3})$, we denote the generators in $A_s$ (resp.~ $B_s$) by $a_s,b_s$ and $c_s$  (resp.~ $a'_s,b'_s$ and $c'_s$) for $0\leq s \leq n$. We will abuse the notation and use $\partial$ for the differential in $X^\infty_n(T_{2,3})$ from now on. For each $s=1,\cdots, n-1$, We have
\begin{align*}
    \partial a_s &= Ub_s + c_s + a'_i + U^s a'_{s+1} \qquad \qquad \null \\
    \partial b_s &= b'_s + U^{s-1} c'_{s+1} \\
    \partial c_s &= c'_s + U^{s+1} b'_{s+1}\\
    \partial a'_s &= Ub'_s +  c'_s.
\end{align*}
 We will establish a set of basis for the reduced complex of $X^\infty_n(T_{2,3})$.   First we can let $c'_s + Ub'_s$ replace $c'_s$ and quotienting out the acyclic summand generated by $a'_s$ and $c'_s + Ub'_s$. Then considering the first three relations, it is natural to perform the following change of basis (for all $0\leq s\leq n$):
\begin{gather*}
    \begin{aligned}
    \alpha_s &= a_s    \qquad \qquad   &   x_s &= b_s + U^{s-1} a'_{s+1}\\
     \beta_s &=c'_s    \qquad  \qquad  &       y_s &= c'_s + a_s,
    \end{aligned}
\end{gather*}
such that under the new basis, the relations are simplified.  See Figure \ref{fig: portion1}. For each $s=1,\cdots, n-1$,
\begin{align*}
    \partial \alpha_s &= Ux_s + y_s \qquad \qquad \qquad \qquad \qquad \null \\
     \partial x_s &= \beta_s + U^s \beta_{s+1} \\
    \partial  y_s &= U\beta_s + U^{s+1} \beta_{s+1}. 
\end{align*}
When considered as a generator of the complex $B_s\{ i =0\},$  $\beta_s$ is the unique generator of homology under the chosen basis. The $(\II,\JJ)$ filtrations of the generators can be computed using \eqref{eq: filtration_s3_1},\eqref{eq: filtration_s3_2},\eqref{eq: filtration_s3_3} and \eqref{eq: filtration_s3_4}. All $\alpha_s, \beta_s, x_s$ and $y_s$ have $\II=0$ and their $\JJ$ filtrations are as follows:
\begin{align*}
    \JJ(\alpha_s)&=-\frac{1}{2}n^2  + \frac{1}{2}n + ns -s\\
    \JJ(\beta_s)&=-\frac{1}{2}n^2  + \frac{1}{2}n + n(s -1)\\
    \JJ(x_s)&=-\frac{1}{2}n^2  + \frac{1}{2}n + ns -(s-1)\\
      \JJ(y_s)&=-\frac{1}{2}n^2  + \frac{1}{2}n + ns -(s+1).
\end{align*}
The differential of $\alpha_s$ consists of $Ux_s$ and $y_s$, which differ by $(1,-1)$ in  the $(\II,\JJ)$ grading. Next    consider the differentials of $x_s$ and $y_s$. The $(\II,\JJ)$ grading shift is equal to $(0,n-s+1)$  from $x_s$ to $\beta_s$ and equal to $(s,1)$  from $x_s$ to $U^s\beta_{s+1}$. Note that $ \partial  y_s = \partial Ux_s$, from this immediately follows that  the $(\II,\JJ)$ grading shifts in the differential of $y_s$ are $(1,n-s)$ and $(s+1,0)$ respectively.

 We then consider the generators in $A_0$. After performing a change of basis 
 \begin{align*}
      x_0 &\xmapsto[]{}  x_0 + U^{-1}y_0\\
 \alpha_0 &\xmapsto[]{} \alpha_0,
 \end{align*}
and quotienting out the resulting acyclic summand, the only remaining generator in $A_0$ is $y_0$, with the differential given by 
 \[
 \partial y_0 = U\beta_1.
 \]
 The $\JJ$ filtrations of $y_0 $ and $ U\beta_1$ are equal.

With the above computations ready,  the rest of the proof  assesses the complex  over the ring $\F[U,V]/(UV)$,  using the language of \cite{Moreconcor}. The sub-quotient complex generated by $\alpha_s, \beta_s x_s$ and $y_s$  corresponds to the sequence $(n-s+1,1,-1,-s-1)$, starting from $\beta_s$. See Figure \ref{fig: portion2}.  Next observe that each $\beta_s$ is in the image of $x_s$ and  $y_{s+1}$ (with some appropriate $U,V$ decoration), thus as $s$ ranges from $1$ to $n-1$,  the sub-quoitent complex generated by all $\alpha_s, \beta_s, x_s$ and $y_s$  corresponds to the  concatenation of $(n-s+1,1,-1,-s-1)$.

 Finally, it is easy to see that by setting $U=0$, each pair of generators $\alpha_s$ and $y_s$,  $\beta_s$ and  $x_s$ for $n\geq s\geq 1$ makes up a torsion summand in the vertical homology while $y_0$ generates the single copy of $\F[V]$. Therefore by further setting $V=1$, $\HFa(S^3_1(T_{2,3}))$ (and thus the $d$-invariant) is supported in $y_0$ as required.  The differential of $y_0$ gives the $-1$ in the beginning of the sequence. Following from the symmetry, we obtain the $1$ at the end of the sequence.
\end{proof}

\begin{remark}
In fact, the proof shows that over the ring $\F[U,V],$ $\CFK(S^3,T_{2,3})/(UV)$ is isomorphic to the standard complex $(-1,\cdots,n-i+1,1,-1,-i-1,\cdots,1)$, with $i=1,\cdots, n-1.$ Furthermore, this information completely determines $\CFKi(S^3,T_{2,3})$ (via the condition $\partial^2 = 0$).
\end{remark}

\begin{figure}[!htb]
  \begin{tikzpicture}

		\draw  [thick] (2, 2) -- (1, 2);
		\draw  [thick] (2, 2) -- (2, 1);
			\draw  [thick] (1, 2) -- (1, -2);
			\draw  [thick] (2, 1) -- (-2, 1);

\node[] at (-2.5,1.3) {\small $\beta_{s+1}$};
\node[] at (2.4,0.7) {\small $y_s$};
\node[] at (2.4,2.3) {\small $\alpha_s$};
\node[] at (0.7,2.3) {\small $x_s$};
\node[] at (1.3,-2) {\small $\beta_{s}$};

\node[] at (1.5,2.2) {\tiny $1$};
\node[] at (2.2,1.6) {\tiny $1$};

\node[] at (-0.5,1.3) {\small $s+1$};
\node[] at (0,-0.5) {\small $n-s+1$};
\end{tikzpicture}
    \caption{The portion of $X^\infty_n (T_{2,3})$ that corresponds to the sequence \newline $(n-s+1,1,-1,-s-1)$ in a standard complex using vertical basis. The generators are denoted abstractly, without their $U,V$ decoration.}
    \label{fig: portion2}
    
\end{figure}
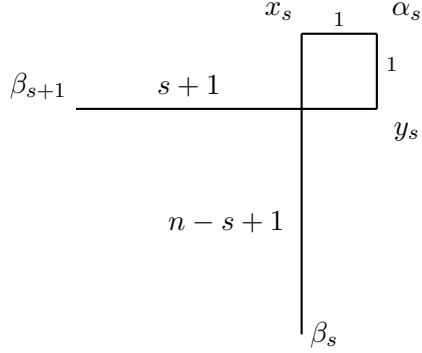

 We are ready to prove the computational result Proposition \ref{prop: phi} which leads to Theorem \ref{thm: phi}.
 Recall that torus knots $T_{2,4k+3}$ are $L$--space knots with genus $g=2k + 1$, where $k\in \Z_{\geq 0}$. (We have so far reserved the letter $k$ for the surgery coefficient. Since in this section the surgery is mostly fixed, hopefully this new assignment of letter $k$ does not cause confusion.)  The knot Floer complex $\CFKi(S^3,T_{2,4k+3})$ is generated by $a_i$ for $i=1,\cdots,g$ with coordinate $(0,-g+2i-1)$ and $b_i$ for $i=1,\cdots,g+1$ with coordinate $(0,-g+2i-2)$. The Maslov grading is supported in $b_{g+1}$ and the differentials are given by
 \[
 \partial a_i = Ub_{i+1} + b_i, \quad \text{for} \hspace{0.5em} i=1,\cdots,g.
 \]
 As before, let $(T_{2,4k+3})_{n,1}$ denote the $(n,1)$--cable of the meridian inside $+1$-surgery on $T_{2,4k+3}$ and let $J_{n,k}$ denote $(T_{2,4k+3})_{n,1}$ connected sum with the unknot in $-S^3_1(T_{2,4k+3})$. The ambient manifold  $S^3_1(T_{2,4k+3}) \conn -S^3_1(T_{2,4k+3})$ is homology cobordant to $S^3$. 

\propphi*
 
 This immediately implies Theorem \ref{thm: phi}. In fact, we can prove a stronger version.
 \begin{theorem}
 For any integers $i$ and $j$ such that $i>j\geq 0,$ and a given sequence of arbitrary  integers $(h_{j+1},\cdots, h_i)$,
 there exists some knot $K\in \CZhat$ such that $\varphi_{t,j}(K)=h_t$ for $t=j+1,\cdots, i,$ and $\varphi_{t,j}(K)=0$ for $t>i.$
 \end{theorem}
 \begin{proof}
 The concordance invariant $\varphi_{i,j}$  is additive under the group action of $\CZhat$. By Proposition \ref{prop: phi}, we can first choose a knot  with $\varphi_{i,j}=h_i$. Assuming $i-j>1,$ then by adding copies of $J_{i-1,j}$ to the chosen knot, we can make it such that $\varphi_{i-1,j}=h_{i-1}$ while the value of $\varphi_{i,j}$ remains unchanged. Repeat this process until  a knot with desired property is obtained.
 \end{proof}
 
  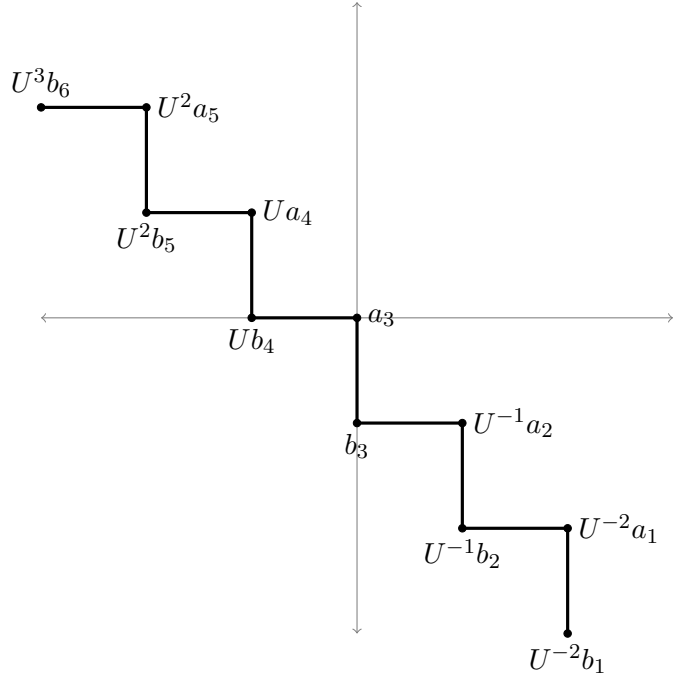
\begin{figure}[!htb]

\begin{tikzpicture}[scale=0.7]

\begin{scope}[thin, black!50!white]
		\draw [<->] (-6, 0) -- (6, 0);
	\draw [<->] (0, -6) -- (0, 6);
	\end{scope}
		\filldraw (-6,4) circle (2pt) node (b_{6}) {};
			\node[above] at (b_{6}) {\small $U^3b_{6}$};
    \foreach \i in {-2,...,2}
    {
	\filldraw (-2*\i,2*\i) circle (2pt) node (a_{\i}) {};
	\filldraw (-2*\i,-2+2*\i) circle (2pt) node (b_{\i}) {};

		\draw [very thick] (-2*\i,2*\i)--(-2*\i-2,2*\i);
		\draw [very thick] (-2*\i,2*\i)--(-2*\i,-2+2*\i);
	}
	\node[below] at (b_{2}) {\small $U^2b_{5}$};
	\node[below] at (b_{1}) {\small $Ub_{4}$};
	\node[below] at (b_{0}) {\small $b_{3}$};
	\node[below] at (b_{-1}) {\small $U^{-1}b_{2}$};
	\node[below] at (b_{-2}) {\small $U^{-2}b_{1}$};
	
		\node[right] at (a_{2}) {\small $U^2a_{5}$};
	\node[right] at (a_{1}) {\small $Ua_{4}$};
	\node[right] at (a_{0}) {\small $a_{3}$};
	\node[right] at (a_{-1}) {\small $U^{-1}a_{2}$};
	\node[right] at (a_{-2}) {\small $U^{-2}a_{1}$};
	\end{tikzpicture}

	\caption{The knot Floer complex $\CFKi(S^3,T_{2,11})$. The torus knot $T_{2,11}$ is an $L$--space knot with genus equal to $5$.}
	\label{fig:lspace}
	\end{figure}

 Before moving on to the proof of Proposition \ref{prop: phi}, let us study the complex $X^\infty_n (T_{2,4k+3})$ closely. The case when $k=0$ has been studied, so we assume $k>0$ from now on. Via the isomorphism with $\CFKi(S^3,T_{2,4k+3})$, denote the generators in  $A_s$  by $(a_i)_s$ for $i=1,\cdots, g$ and $(b_i)_s$ for $i=1,\cdots, g+1$ and the generators in $B_s$ by $(a_i)'_s$ and $(b_i)'_s$, where $s=-g+1,\cdots, g+n-1.$ The differential on the mapping cone is given by
 \begin{align*}
     \partial (a_i)_s &= (b_i)_s + U(b_{i+1})_s + (a_i)'_s + U^{s+g+1-2i} (a_{g-i+1})'_{s+1} \\
     \partial (b_i)_s &= (b_i)'_s + U^{s+g+2-2i} (b_{g-i+2})'_{s+1} \\
     \partial (a_i)'_s &= (b_i)'_s + U(b_{i+1})'_s.
 \end{align*}
 We first aim to choose a reduced basis for the complex of $X^\infty_n (T_{2,4k+3})$. Note that $(a_i)'_s $ and $ (b_i)'_s \in \partial  (a_i)'_s$ are in the same $(\II,\JJ)$ filtration, so we may quotient out the acyclic complex generated by $\{(a_i)'_s, \partial  (a_i)'_s \}$. Under this identification, we have 
 \[
    (b_{g+1})'_s = U(b_{g})'_s = \cdots = U^g(b_{1})'_s. \quad \null
 \]
 Setting
 \begin{align*}
     (\alpha_i)_s &=(a_i)_s   \qquad \qquad \qquad \qquad \qquad \null\\
     (x_i)_s &=(b_i)_s       \\
     \beta_s &=(b_{g+1})'_s,  
 \end{align*}
 the differentials are now simplified considerably:
 \begin{align}
   \label{eq: diff_a}  \partial (\alpha_i)_s &= (x_i)_s + U(x_{i+1})_s \\
   \label{eq: diff_x}  \partial (x_i)_s &= U^{g+1-i}\beta_s + U^{s+g+1-i} \beta_{s+1}. 
 \end{align}
 Next, notice for certain values of $i$ and $s$,  at least one term in $\partial (\alpha_i)_s$ preserve the $(\II,\JJ)$ filtration, and we shall quotient out the acyclic summands generated by these $\{(a_i)_s, \partial  (a_i)_s \}$  as well. See Figure \ref{fig: reduced}. When $-g\leq s\leq g$, we claim for $i \in \{1,\cdots, g\}$  the above condition is satisfied if and only
\begin{gather*}
\begin{aligned}
 i\geq \frac{s+g+1}{2} & \quad \text{or} \quad i\leq \frac{s+g-1}{2}-\floor{\frac{n-1}{2}}  & \text{ if } s \text{ is even;} \\
 i\geq \frac{s+g+2}{2} & \quad  \text{or} \quad  i\leq \frac{s+g}{2}-\ceil{\frac{n-1}{2}}  & \text{ if } s \text{ is odd},
\end{aligned}  
\end{gather*}
 and when $s\geq g$, if and only if
\[
  i\leq \frac{g+1}{2} - \ceil{\frac{n-s}{2}}.
\]
 We will show the claim when  $-g\leq s\leq g$ and $s$ is even. The other two cases are similar and left for the reader to verify. The first half of the claim is clear, and we focus on the second half.

 \begin{figure}[!htb]
 \centering
     \subfloat[The complex $A_s \{\text{max}(i,j-s)=0 \}$ when $-g \leq s \leq g$ and $s$ is even. ]{
     \begin{tikzpicture}[scale=0.6]
       \begin{scope}[thin, black!0!white]
      	\draw (-10, 0) -- (12, 0);
      \end{scope}
     \begin{scope}[thin, black!40!white]
		\draw (0, 0) -- (0, -10);
	\draw  (2, 2) -- (2, -10);
	\draw  (-4, 0) -- (0, 0);
		\draw  (-4, 2) -- (2, 2);
			\draw (0, 0) -- (2, 0);
					\draw (0, -2) -- (2, -2);
						\draw (0, -4) -- (2, -4);
	\end{scope}

	\filldraw (-1,1) circle (2pt) node () {};
	
	  \foreach \i in {0,...,2}
    {
	\filldraw (1,1-4*\i) circle (2pt) node () {};
	\filldraw (1,-1-4*\i) circle (2pt) node () {};

		\draw [very thick] (1,1-4*\i)--(1,-1-4*\i);
	}
	\draw [very thick] (-1,1)--(1,1);
	
		\node[] at (-2.1,0.8) {$x_{\frac{s+g+3}{2}}$};
		\node[right] at (1.2,1) {$\alpha_{\frac{s+g+1}{2}}$};
		\node[right] at (1.2,-1) {$x_{\frac{s+g+1}{2}}$};
			\node[right] at (1.2,-3) {$\alpha_{\frac{s+g-1}{2}}$};

		\node[right] at (2,-5) {$\cdot$};
			\node[right] at (2,-5.5) {$\cdot$};
				\node[right] at (2,-4.5) {$\cdot$};
				
					\node[right] at (1.2,-7) {$\alpha_i$};
					\node[right] at (1.2,-9) {$x_i$};

     \end{tikzpicture}
     }\par\medskip
     \begin{minipage}{.5\linewidth}
     \centering
      \subfloat[The complex $A_s \{\text{max}(i,j-s)=0 \}$ when $-g \leq s \leq g$ and $s$ is odd. ]{
      \begin{tikzpicture}[scale=0.5]
      \begin{scope}[thin, black!0!white]
      	\draw (-10, 0) -- (8, 0);
      \end{scope}
     \begin{scope}[thin, black!40!white]
		\draw (0, 0) -- (0, -10);
	\draw  (2, 2) -- (2, -10);
	\draw  (-4, 0) -- (0, 0);
		\draw  (-4, 2) -- (2, 2);
			\draw (0, 0) -- (2, 0);
					\draw (0, -2) -- (2, -2);
						\draw (0, -4) -- (2, -4);
	\end{scope}
	
	 \filldraw (-1,1) circle (2pt) node () {};
	  \filldraw (-3,1) circle (2pt) node () {};
		\draw [very thick] (-1,1)--(-3,1);
	
	  \foreach \i in {0,...,1}
    {
	\filldraw (1,-1-4*\i) circle (2pt) node () {};
	\filldraw (1,-3-4*\i) circle (2pt) node () {};

		\draw [very thick] (1,-1-4*\i)--(1,-3-4*\i);
	}
   \filldraw (1,1) circle (2pt) node () {};
	
		\node[right] at (1.2,1) {$x_{\frac{s+g+2}{2}}$};
		\node[right] at (1.2,-1) {$\alpha_{\frac{s+g}{2}}$};
			\node[right] at (1.2,-3) {$x_{\frac{s+g}{2}}$};
			
		\node[right] at (2,-5) {$\cdot$};
		\node[right] at (2,-5.5) {$\cdot$};
	\node[right] at (2,-6) {$\cdot$};

     \end{tikzpicture}
     }
     \end{minipage}%
        \begin{minipage}{.5\linewidth}
     \centering
       \subfloat[The complex $A_s \{\text{max}(i,j-s)=0 \}$ when $ s \geq g$.]{
      \begin{tikzpicture}[scale=0.5]
        \begin{scope}[thin, black!0!white]
      	\draw (-10, 0) -- (8, 0);
      \end{scope}
     \begin{scope}[thin, black!40!white]
		\draw (0, 6) -- (0, -4);
        \draw (2, 6) -- (2, -4);
        \foreach \i in {0,...,2}
    {
     \draw (0, 2*\i) -- (2, 2*\i);
	}
	\end{scope}

	\filldraw (1,5) circle (2pt) node () {};
    \filldraw (1,3) circle (2pt) node () {};
    \filldraw (1,1) circle (2pt) node () {};
      \filldraw (1,-1) circle (2pt) node () {};
        \filldraw (1,-3) circle (2pt) node () {};

		\draw [very thick] (1,3)--(1,1);
			\draw [very thick] (1,-3)--(1,-1);

		\node[right] at (1.2,5) {$x_{g+1}$};
			\node[right] at (1.2,3) {$\alpha_g$};
		\node[right] at (1.2,1) {$x_g$};
		
		\node[right] at (1.5,-2) {$\cdot$};
		\node[right] at (1.5,-1.5) {$\cdot$};
		\node[right] at (1.5,-1) {$\cdot$};

     \end{tikzpicture}
     }
     \end{minipage}
     \caption{Part of the associated graded complex $X^\infty_n(T_{2,4k+3})\{i=0\}$,  the complex $A_s \{\text{max}(i,j-s)=0 \}$ with the $\JJ$ filtration level depicted. }
     \label{fig: reduced}
 \end{figure}
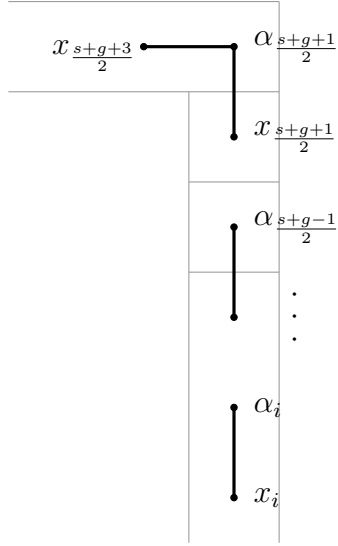
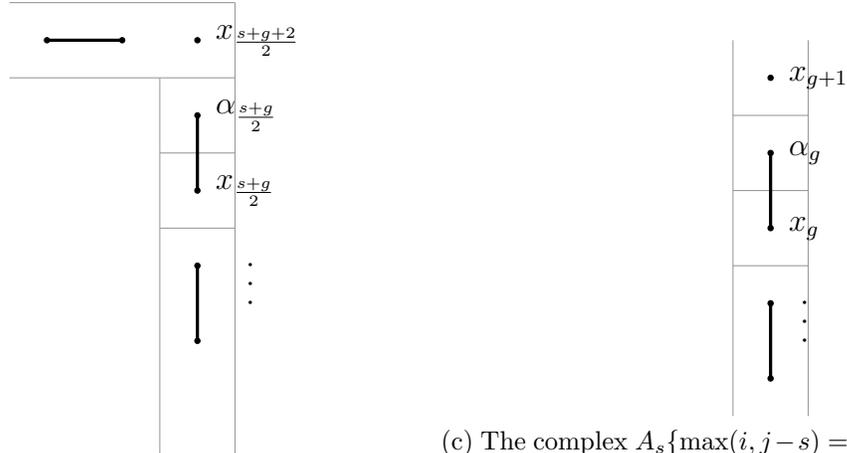

 In order to simplify the notation when talking about the filtration,   define 
 \begin{equation}\label{eq: def_f}
      f(n,s)=-\frac{(n-1)n}{2} + ns,
 \end{equation}
 so that the $\JJ$ filtration on the complex $A_s \{\text{max}(i,j-s)=0 \}$ ranges between $f(n,s)$ and $f(n,s)-n.$ When $s$ is even, for any $l\geq 0$  the only generator supported in the coordinate $(0,s-l)$  is $ x_{(s+g+2-l)/2}$ if $l$ is odd and $\alpha_{(s+g+1-l)/2}$ if $l$ is even. Therefore $\JJ(\alpha_i)=f(n,s)-n$ if and only if 
 \begin{align*}
     \begin{cases}
       i\leq (s+g-n)/2   & n     \text{ is odd}\\
        i\leq (s+g+1-n)/2    &  n    \text{ is even}.
     \end{cases}
 \end{align*}
In either case, the second half of the claim is verified. Thus we have obtained a reduced basis for the complex $X^\infty_n (T_{2,4k+3})$ as follows. The generators in $A^\infty_s$ are given by $(\alpha_i)_s, (x_i)_s$  for $s=-g+1,\cdots, g+n-1$, such that when $s\leq g,$  
\begin{gather*}
\begin{aligned}
  \frac{s+g+1}{2} \geq i \geq \text{max}\{1,\frac{s+g+1}{2}-\floor{\frac{n-1}{2}}\}  & \text{ if } s \text{ is even,} \\
  \frac{s+g+2}{2} \geq i \geq \text{max}\{1,\frac{s+g+2}{2}-\ceil{\frac{n-1}{2}}\}  & \text{ if } s \text{ is odd},
\end{aligned}  
\end{gather*}
  when $s\geq g$, 
\[
 g+1 \geq i\geq \text{max} \{1,\frac{g+3}{2} - \ceil{\frac{n-s}{2}}\};
\]
the generators in $B^\infty_s$ are $\beta_s$ for $s=-g+2,\cdots, g+n-1$ and the differentials are given by \eqref{eq: diff_a} and \eqref{eq: diff_x}. 

Before moving on, let us introduce some notational shorthands:
 denote by $x_s$ the ``top'' generator in each $A_s$ complex, i.e. the one with the largest $i$ index; for example, when $-g+1 \leq s \leq g$ and $s$ is even, $x_s=(x_{(s+g+1)/2})_s$. Let $i^{(t)}_{n,s}$ be this ``top'' $i$ index. Similarly, let  $i^{(b)}_{n,s}$  be the smallest integer that $i$ index can take in each case. Here letters $t$ and $b$ are chosen to indicate ``top'' and ``bottom'' respectively.

We want to take a look at the $(\II,\JJ)$ filtration shifts next. Clearly, the filtration shifts in the differential of each $(\alpha_i)_s$ are $(1,0)$ and $(0,1).$ So we focus on the filtration shifts in the differential of each $(x_i)_s$.  Suppose $U^c \beta_{s'}$ is a nontrivial term in $\partial (x_i)_s$, where $s'=s$ or $s+1.$ Define 
\begin{equation}
    \Delta_{\II,\JJ}((x_i)_s,\beta_{s'}) = (\II,\JJ)((x_i)_s) - (\II,\JJ)(U^c \beta_{s'})
\end{equation}
and similarly define $\Delta_\II$ and $\Delta_\JJ$. Note that we have  
\begin{equation}
    \Delta_{\II,\JJ}((x_{i-1})_s,\beta_{s'})=\Delta_{\II,\JJ}((x_i)_s,\beta_{s'}) + (1,-1)
\end{equation}
for $ i^{(t)}_{n,s} \geq i\geq i^{(b)}_{n,s}+1$, therefore  to completely understand the filtration shifts of all the generators in each complex $A_s$, we need only compute the filtration shift of the ``top'' generator $x_s$.

\begin{lemma}\label{le: del_ij}
When $-g+1\leq s \leq g,$ if $s$ is even,
\begin{align}
    \label{eq: del_sleqg_even_s} \Delta_{\II,\JJ}(x_s,\beta_s)&=\Big(\frac{-s+g+1}{2},n+\frac{-s+g-1}{2}\Big)\\
     \label{eq: del_sleqg_even_sp} \Delta_{\II,\JJ}(x_s,\beta_{s+1})&=\Big(\frac{s+g+1}{2},\frac{s+g-1}{2}\Big);\\
      \intertext{if $s$ is odd,}
     \label{eq: del_sleqg_odd_s} \Delta_{\II,\JJ}(x_s,\beta_s)&=\Big(\frac{-s+g}{2},n+\frac{-s+g}{2}\Big)\\
      \label{eq: del_sleqg_odd_sp} \Delta_{\II,\JJ}(x_s,\beta_{s+1})&=\Big(\frac{s+g}{2},\frac{s+g}{2}\Big);\\
       \intertext{when $g\leq s \leq g+n-1$,}
       \label{eq: del_sgeqg_s} \Delta_{\II,\JJ}(x_s,\beta_s)&=(0,n+g-s)\\
      \label{eq: del_sgeqg_sp} \Delta_{\II,\JJ}(x_s,\beta_{s+1})&=(s,g).
\end{align}
\end{lemma}

\begin{proof}
For each case, $\JJ(\beta_s)=f(n,s)-n$ and $\JJ(\beta_{s+1})=f(n,s)$ by \eqref{eq: filtration_s3_4}. When 
$-g+1\leq s \leq g,$ and $s$ is even,
\begin{align*}
        \partial x_s &= U^{(-s+g+1)/2} \beta_s + U^{(s+g+1)/2} \beta_{s+1},\\
        \JJ(x_s)&=f(n,s)-1.
  \end{align*}
 Compute
  \begin{align*}
        \Delta_{\II,\JJ}(x_s,\beta_s)&= \Big(\frac{-s+g+1}{2},\frac{-s+g+1}{2}\Big) + (0,n-1)\\
  &=\Big(\frac{-s+g+1}{2},n+\frac{-s+g-1}{2}\Big)\\
  \Delta_{\II,\JJ}(x_s,\beta_{s+1})&= \Big(\frac{s+g+1}{2},\frac{s+g+1}{2}\Big) + (0,-1)\\
  &=\Big(\frac{s+g+1}{2},\frac{s+g-1}{2}\Big).
  \end{align*}
When $-g+1\leq s \leq g,$ and $s$ is odd, the computation is parallel. Suppose now $g\leq s \leq g+n-1$, then we have
\begin{align*}
        \partial x_s &= \beta_s + U^s \beta_{s+1},\\
        \JJ(x_s) &= f(n,s) - (s-g).
  \end{align*}
So we may compute
\begin{align*}
        \Delta_{\II,\JJ}(x_s,\beta_s)&= (0,n+g-s)\\
  \Delta_{\II,\JJ}(x_s,\beta_{s+1})&= (s,s) + (0,g-s)\\
  &=(s,g).
  \end{align*}

\end{proof}

 Next we claim that for the purpose of computing concordance invariants, we need only  look at a portion of the mapping cone.
Define  $X^\infty_n (K) \langle l \rangle =  X^\infty_{1,n,-l+n,l} (K)$ for $l\in \Z$, which is the filtered mapping cone
 \begin{align*}
    \bigoplus^{l}_{s=-l+n}A^\infty_s \xrightarrow{v^\infty_s+h^\infty_s} \bigoplus^{l}_{s=-l+n+1}B^\infty_s.
\end{align*}
Note that under this notation  $X^\infty_n (K) = X^\infty_n (K) \langle g+n-1 \rangle$. 
 \begin{lemma}\label{le: localequi}
 For any $k \geq 1$ and $n\geq 1$, the filtered complex $X^\infty_n (T_{2,4k+3}) \langle g+n-1 \rangle$  is isomorphic to $ X^\infty_n (T_{2,4k+3}) \langle n \rangle \oplus D $ up to a change of basis, where $H_*(D)=0$.
 \end{lemma}
 \begin{proof}
 It suffices to show for each $l=n+1,\cdots, g+n-1,$ the complex $X^\infty_n (T_{2,4k+3}) \langle l \rangle$ is isomorphic to $X^\infty_n (T_{2,4s+3}) \langle l-1 \rangle \oplus D$ up to a change of basis, where $H_*(D)=0$. For every such $l$, we will demonstrate a filtered change of basis such that the complex given by
 \begin{equation}\label{eq: summand1}
     A^\infty_{-l+n} \xrightarrow[]{h_{-l+n}}B^\infty_{-l+n+1}
 \end{equation}
 becomes a summand under the new basis. Thus following from the symmetry, there is also a filtered change of basis such that 
 \begin{equation}\label{eq: summand2}
 A^\infty_{l} \xrightarrow[]{\hspace{0.3em} v_{l}\hspace{0.3em}}B^\infty_{l}
 \end{equation}
 becomes a summand under the new basis as required.
 
 Let $s=-l+n+1$ in the following proof, and we have $s\leq 0.$ There are two cases to consider depending on the parity of $s.$
 \begin{itemize}
     \item When $s$ is even, perform the change of basis
     \begin{align*}
         (x_i)_{s} \xmapsto[\qquad\null]{}  &(x_i)_{s} + U^{\frac{-s+g+3}{2} -i} x_{{s}-1},  
          \quad \text{for} \hspace{0.5em}   i^{(t)}_{n,s} \geq i \geq i^{(b)}_{n,s}.
     \end{align*}
     By \eqref{eq: diff_a} and \eqref{eq: diff_x} it is straightforward to check that the complex given by \eqref{eq: summand1} is a summand under this basis. (Note that in complex $X^\infty_n (T_{2,4k+3}) \langle l \rangle$ the map $v_{-l+n}=v_{s-1}$ is the zero map.)  It remains to show that this change of basis is filtered. This in fact amounts to showing for $i^{(t)}_{n,s} \geq i \geq i^{(b)}_{n,s}$,
     \begin{align*}
      \Delta_{\II,\JJ}((x_i)_s, \beta_s ) \geq  \Delta_{\II,\JJ}(x_{s-1}, \beta_s).
     \end{align*}
        Thus we compute 
     \begin{align*}
       \Delta_{\II,\JJ}&((x_i)_s, \beta_s ) - \Delta_{\II,\JJ}(x_{s-1}, \beta_s)\\
        &= \Delta_{\II,\JJ}(x_s, \beta_s ) + 
       \left( \frac{s+g+1}{2} - i\right)(1,-1) - \Delta_{\II,\JJ}(x_{s-1}, \beta_s) \\
         &= \left( \frac{-s+g+1}{2} ,  n +  \frac{-s+g-1}{2} \right) + \left( \frac{s+g+1}{2} - i ,  -\frac{s+g+1}{2} + i \right) \\
         & \hspace{2em}- \left( \frac{s-1+g}{2} ,   \frac{s-1+g}{2} \right)\\
         &= \left( \frac{-s+g+3}{2} -i ,   n-1 + \frac{-3s-g+1}{2} + i \right).
     \end{align*}
     The $\II$ filtration shift is clearly positive for all $i$ since
     \begin{align*}
         \frac{-s+g+3}{2} -i &\geq \frac{-s+g+3}{2} - \frac{s+g+1}{2}\\
         &=-s+1\\
         &> 0,
     \end{align*}
     while for the  $\JJ$ filtration shift there are two possible cases depending on the value of $i^{(b)}_{n,s}$:
     \begin{enumerate}[label=\roman*)]
         \item Suppose
         \[
         1\geq \frac{s+g+1}{2} - \floor{\frac{n-1}{2}},
         \]
         then we have
         \begin{align*}
             n-1 + \frac{-3s-g+1}{2} + i &\geq  n + \frac{-3s-g+1}{2} \\
            & \geq 2\floor{\frac{n-1}{2}} +1 + \frac{-3s-g+1}{2}\\
            &\geq s+g + \frac{-3s-g+1}{2}\\
            &\geq \frac{-s+g+1}{2}\\
            &>0.
         \end{align*}
         \item If instead 
         \[
         1\leq \frac{s+g+1}{2} - \floor{\frac{n-1}{2}},
         \]
         this is when $n$ is relatively small, and we have   
         \begin{align*}
             n-1 + \frac{-3s-g+1}{2} + i &\geq  n-1 + \frac{-3s-g+1}{2} + \frac{s+g+1}{2} - \floor{\frac{n-1}{2}} \\
            & \geq -s+1 - \floor{\frac{n-1}{2}} +n -1 \\
            &\geq -s+1 + \frac{n-1}{2} \\
            &>0.
         \end{align*}
     \end{enumerate}
     \item When $s$ is odd, perform the change of basis
     \begin{align*}
         (x_i)_s \xmapsto[\qquad\null]{}  &(x_i)_s + U^{\frac{-s+g+2}{2} -i} x_{s-1},  \quad \text{for} \hspace{0.5em}  i^{(t)}_{n,s} \geq i \geq i^{(b)}_{n,s}.
     \end{align*}
     Again one can verify under this basis the complex given by \eqref{eq: summand1} is a summand.  The difference in $(\II,\JJ)$ filtration is
     \begin{align*}
      \Delta_{\II,\JJ}&((x_i)_s, \beta_s ) - \Delta_{\II,\JJ}(x_{s-1}, \beta_s)\\
        &= \Delta_{\II,\JJ}(x_s, \beta_s ) + 
       \left( \frac{s+g+2}{2} - i\right)(1,-1) - \Delta_{\II,\JJ}(x_{s-1}, \beta_s) \\
         &= \left( \frac{-s+g}{2} ,  n +  \frac{-s+g}{2} \right) + \left( \frac{s+g+2}{2} - i ,  -\frac{s+g+2}{2} + i \right) \\
         & \hspace{2em}- \left( \frac{s+g}{2} ,   \frac{s+g-2}{2} \right)\\
         &= \left( \frac{-s+g+2}{2} -i ,   n + \frac{-3s-g}{2} + i \right).
     \end{align*}
     The $\II$ filtration shift is clearly positive since
     \begin{align*}
         \frac{-s+g+2}{2} -i &\geq \frac{-s+g+2}{2} - \frac{s+g+2}{2}\\
         &=-s\\
         &> 0,
     \end{align*}
     while the $\JJ$ filtration has two possible cases depending on the value of $i^{(b)}_{n,s}$.
      \begin{enumerate}[label=\roman*)]
         \item Suppose that
         \[
         1\geq \frac{s+g+2}{2} - \ceil{\frac{n-1}{2}},
         \]
         then we have
         \begin{align*}
             n + \frac{-3s-g}{2} + i &\geq  n + \frac{-3s-g}{2} + 1 \\
            & \geq 2\ceil{\frac{n-1}{2}} +  \frac{-3s-g}{2} + 1\\
            &\geq s+g + \frac{-3s-g}{2} + 1\\
            &\geq \frac{-s+g}{2} + 1\\
            &>0.
         \end{align*}
         \item If instead  
         \[
         1\leq \frac{s+g+1}{2} - \floor{\frac{n-1}{2}},
         \]
          in this case we have   
         \begin{align*}
            n + \frac{-3s-g}{2} + i &\geq  n    + \frac{-3s-g}{2} + \frac{s+g+2}{2} - \ceil{\frac{n-1}{2}} \\
            & \geq -s + 1 + \ceil{\frac{n-1}{2}}  \\
            &>0
         \end{align*}
         as required.
     \end{enumerate}
  \end{itemize}

 \end{proof}

\begin{proof}[Proof of Proposition \ref{prop: phi}]
By Lemma \ref{le: localequi}, in order to compute concordance invariants  it suffices to consider the complex $X^\infty_n(T_{2,4k+3})\langle n \rangle$. We claim that over the ring $\X$ discussed in the previous subsection \ref{ssec: ring}, after a change of basis $X^\infty_n(T_{2,4k+3})\langle n \rangle$ is a standard complex.

Translating to the ring $\X$, the differential on the complex is given by
\begin{align*}
    \partial  (\alpha_i)_s &= (U_B + W_{T,0}) (x_{i+1})_s + (W_{B,0} + V_T) (x_i)_s  \\
     \partial {{(x_i)}}_0 &= \big( U^{\Delta_\II({(x_i)}_0, \beta_{1})}_B W^{\Delta_\JJ({(x_i)}_0, \beta_{1})}_{B,0} + V^{\Delta_\JJ({(x_i)}_0, \beta_{1})}_T W^{\Delta_\II({(x_i)}_0, \beta_{s+1})}_{T,0}     \big)\beta_1\\
     \partial {{(x_i)}}_n &= \big( U^{\Delta_\II({(x_i)}_n, \beta_n)}_B W^{\Delta_\JJ({(x_i)}_n, \beta_n)}_{B,0} + V^{\Delta_\JJ({(x_i)}_n, \beta_n)}_T W^{\Delta_\II({(x_i)}_n, \beta_n)}_{T,0}     \big)\beta_n\\
     \intertext{\hspace{1em}and if $n>s>0,$ then}
    \partial {{(x_i)}}_s &= \big( U^{\Delta_\II({(x_i)}_s, \beta_s)}_B W^{\Delta_\JJ({(x_i)}_s, \beta_s)}_{B,0} + V^{\Delta_\JJ({(x_i)}_s, \beta_s)}_T W^{\Delta_\II({(x_i)}_s, \beta_s)}_{T,0}     \big)\beta_s \\ 
    +\big( &U^{\Delta_\II({(x_i)}_s, \beta_{s+1})}_B W^{\Delta_\JJ({(x_i)}_s, \beta_{s+1})}_{B,0} + V^{\Delta_\JJ({(x_i)}_s, \beta_{s+1})}_T W^{\Delta_\II({(x_i)}_s, \beta_{s+1})}_{T,0}     \big)\beta_{s+1}.
\end{align*}
We first perform the change of basis
\begin{align*}
    (\widetilde{x}_i)_s &= \begin{cases}
      (x_i)_s + U^{-1}_B W_{B,0} (x_{i-1})_s  &\quad \text{if} \hspace{0.5em} i=i^{(t)}_{n,s}\\
     (x_i)_s + U^{-1}_B W_{B,0} (x_{i-1})_s + V^{-1}_T W_{T,0} (x_{i+1})_s &\quad \text{if} \hspace{0.5em}  i^{(b)}_{n,s}<i<i^{(t)}_{n,s}\\
     (x_i)_s + V^{-1}_T W_{T,0} (x_{i+1})_s  &\quad \text{if} \hspace{0.5em} i=i^{(b)}_{n,s},
    \end{cases}
\intertext{which simplifies the differentials considerably:}
    \partial  (\alpha_i)_s &= U_B  (\widetilde{x}_{i+1})_s +  V_T (\widetilde{x}_i)_s  \\
     \partial {(\widetilde{x}_i)}_0 &=  \begin{cases}
    V^{\Delta_\JJ({(x_i)}_0, \beta_{1})}_T W^{\Delta_\II({(x_i)}_0, \beta_{1})}_{T,0} \beta_{1}, \hspace{0.5em} i=i^{(t)}_{n,0} \\
     U^{\Delta_\II({(x_i)}_0, \beta_{1})}_B W^{\Delta_\JJ({(x_i)}_0, \beta_{1})}_{B,0}\beta_{1}, \hspace{0.5em} i=i^{(b)}_{n,0}\\
    0,   \hspace{10em}\text{ otherwise}
         \end{cases}\\
     \partial {(\widetilde{x}_i)}_n &=\begin{cases}
      V^{\Delta_\JJ({(x_i)}_n, \beta_n)}_T W^{\Delta_\II({(x_i)}_n, \beta_n)}_{T,0} \beta_n, \hspace{0.5em} i=i^{(t)}_{n,n} \\
     U^{\Delta_\II({(x_i)}_n, \beta_n)}_B W^{\Delta_\JJ({(x_i)}_n, \beta_n)}_{B,0}\beta_n, \hspace{0.5em} i=i^{(b)}_{n,n}\\
    0,   \hspace{10.4em}\text{ otherwise}
         \end{cases}
         \end{align*}
          and for $n>s>0,$
         \begin{align*}
    \partial {(\widetilde{x}_i)}_s = &\begin{cases}
      V^{\Delta_\JJ({(x_i)}_s, \beta_s)}_T W^{\Delta_\II({(x_i)}_s, \beta_s)}_{T,0} \beta_s 
    + V^{\Delta_\JJ({(x_i)}_s, \beta_{s+1})}_T W^{\Delta_\II({(x_i)}_s, \beta_{s+1})}_{T,0} \beta_{s+1}, \hspace{0.5em} i=i^{(t)}_{n,s} \\
     U^{\Delta_\II({(x_i)}_s, \beta_s)}_B W^{\Delta_\JJ({(x_i)}_s, \beta_s)}_{B,0}\beta_s 
    + U^{\Delta_\II({(x_i)}_s, \beta_{s+1})}_B W^{\Delta_\JJ({(x_i)}_s, \beta_{s+1})}_{B,0}\beta_{s+1}, \hspace{0.5em} i=i^{(b)}_{n,s}\\
    0,   \hspace{24.5em}\text{ otherwise.}
         \end{cases}
   \end{align*}
   Next observe that 
   \begin{align*}
       \Delta_\II(x_s,\beta_{s+1}) -  \Delta_\II(x_{s+1},\beta_{s+1}) = \begin{cases}
         s+1 \quad &\text{if} \hspace{0.5em} 0\leq s \leq g\\
         s &\text{otherwise,}
       \end{cases}
   \end{align*}
 by \eqref{eq: del_sleqg_even_s},   \eqref{eq: del_sleqg_even_sp}, \eqref{eq: del_sleqg_odd_s},   \eqref{eq: del_sleqg_odd_sp}, \eqref{eq: del_sgeqg_s} and \eqref{eq: del_sgeqg_sp}. In particular, for any $s$ with $0\leq s \leq n-1,$ we have  $\Delta_\II(x_s,\beta_{s+1}) -  \Delta_\II(x_{s+1},\beta_{s+1})>0.$ Similarly, one can show $\Delta_\JJ((x_{i^{(b)}_{n,s}})_s,\beta_s) -  \Delta_\JJ((x_{i^{(b)}_{n,s-1}})_{s-1},\beta_s)>0$  for $n\geq s \geq 1$.
   
Therefore we may further perform a change of basis 
\begin{align}
\label{eq: changebasis1}
    (\widetilde{x}_{i^{(t)}_{n,s}})_s &\xmapsto[\hspace{2em}]{} \sum_{n\geq s'\geq s}V^{\lambda_{s'}}_T W^{\mu_{s'}}_{T,0} (\widetilde{x}_{i^{(t)}_{n,s'}})_{s'} \quad \text{for all} \hspace{0.5em} n > s \geq 0,\\
    \label{eq: changebasis2}
      (\widetilde{x}_{i^{(b)}_{n,s}})_s &\xmapsto[\hspace{2em}]{} \sum_{0\leq s'\leq s} U^{\theta_{s'}}_B W^{\eta_{s'}}_{B,0}(\widetilde{x}_{i^{(b)}_{n,s'}})_{s'} \quad \text{for all} \hspace{0.5em} 0 < s \leq n,
\end{align}
where
\begin{align*}
 \lambda_{s'} &= \sum_{s'-1 \geq t\geq s} \big(  \Delta_\JJ(x_t,\beta_{t+1}) -  \Delta_\JJ(x_{t+1},\beta_{t+1}) \big)\\
    \mu_{s'} &= \sum_{s'-1 \geq t\geq s} \big(  \Delta_\II(x_t,\beta_{t+1}) -  \Delta_\II(x_{t+1},\beta_{t+1}) \big)  \hspace{5em},
\end{align*} 
and $\theta_{s'}, \eta_{s'}$  can be defined similarly. Note that the definition implies $\lambda_s=\mu_s = 0.$ When $s' > s$ (resp.~$s' < s$), $\mu_{s'}$ (resp.~ $\eta_{s'}$) is positive, while $\lambda_{s'}$ (resp. ~$\theta_{s'}$) needs not to be (in fact they are negative).  This change of basis results in a standard complex as required.

Let $(\widetilde{x})_s$ denote the image of $(\widetilde{x}_{i^{(t)}_{n,s}})_s$ under the above change of basis. Then we have $(\widetilde{x})_0 \in \operatorname{ker} \partial$, and starting from $(\widetilde{x})_0$ every odd number of edge is marked by elements in $\mathcal{R}_U$ and every even number of edge is marked by elements in $\mathcal{R}_V$. In order to compute $\varphi_{i,j},$ we need only consider the set of (say) $\mathcal{R}_V$ edges. Each $\mathcal{R}_V$ edge that is not  equal to $V_T$ is marked by the differential of $(\widetilde{x})_s$ for some $s=1,\cdots,n.$ For $1\leq s\leq n,$
\begin{align*}
    \partial(\widetilde{x})_s = V^{\Delta_\JJ(x_s,\beta_s)}_T W^{\Delta_\II(x_s,\beta_s)}_{T,0}\beta_s
\end{align*}
thus we simply need to look at the values of $\Delta_{\II,\JJ}(x_s,\beta_s).$

 For the knot $T_{2,4k+3}$, $g=2k+1.$ By \eqref{eq: del_sleqg_even_s}, \eqref{eq: del_sleqg_odd_s} and \eqref{eq: del_sgeqg_s},  when $n\geq g$, the sequence of $\Delta_{\II,\JJ}(x_s,\beta_s)$ for $s=1,\cdots,n$ is given by
\begin{align*}
    (k,n+k),(k,n+k-1),\cdots,(1,n+1),(1,n),(0,n),(0,n-1),\cdots,(0,2k+1);
\end{align*}
when $n\leq g$, the above sequence terminates at $(k+1-n/2,n/2+k)$ if $n$ is even and at $(k-(n-1)/2,(n+1)/2+k)$ if $n$ is odd.

Therefore we conclude that for the interested complex, the concordance invariant $\varphi_{i,j}=-1$ if and only if $(i,j)=(n+k,k),(n+k-1,k),\cdots$ and $\varphi_{i,j}=0$ if $(i,j)$ is not among those values and also $(i,j)\neq (1,0).$

\end{proof}

We now prove Proposition \ref{prop: intro_middle} with a more precise restatement. The result is stated for knots in $S^3$, but a similar result should hold for knots in rational homology spheres as well.
\begin{proposition}\label{prop: app_middle}
 For a knot $K\subset S^3,$ when $n\geq 2g,$ as a quotient complex of the filtered mapping cone (for any surgery), $A_{\ceil{n/2}}$ is filtered homotopy equivalent to $\CFKi(S^3,K).$
\end{proposition}
\begin{proof}
When $n\geq 2g,$ $\ceil{n/2}\geq g.$ Note that all the generators of $\CFKi(S^3,K)$ are within the region of $-g\leq i+j \leq g.$ It is then straightforward to check using \eqref{eq: filtration_s3_1} and \eqref{eq: filtration_s3_2} that this region is filtered.
\end{proof}

Finally let us analyze the behavior of the complex $X^\infty_n(K)$ when $n\geq 2g.$ Comparing $X^\infty_n(K)$ and $X^\infty_{n+1}(K)$, we can identify $A_s$ in both complex for $s\leq \floor{n/2}.$  Under this identification, one can check that $\JJ(y_s) - \JJ(h_s(y_s))$ is constant while  $\JJ(y_s) - \JJ(v_s(y_s))$ increases as $n$ increases, and the $\II$ filtration shift is constant. At the same time, we can identify $A_s \subset X^\infty_n(K)$ with $A_{s+1} \subset X^\infty_{n+1}(K)$ for $s\geq \floor{n/2}+1 > g$, and under this identification, similarly the filtration shifts from the $v$ map is constant while the filtration shifts from the $h$ map increases as $n$ increases. Furthermore, by Proposition \ref{prop: app_middle}, the ``middle'' complex $A_{\ceil{(n+1)/2}}$ is simply a copy of $\CFKi(S^3,K)$ (with some differentials pointing to the rest of the complex).

Therefore, when $n\geq 2g$ and as $n$ increases by $1$, we conclude that there are two things happen to $X^\infty_n(K)$:
\begin{itemize}
    \item certain edges are extended;
    \item a copy of $\CFKi(S^3,K)$ is added to the ``middle'' of the complex.
\end{itemize}
Aside from these two changes, the complex remains the same ``shape'' as $n$ increases.  We name this phenomenon \emph{stabilization}. For an example, note that the whole family described by Proposition \ref{prop: t23_n1} can be seen as the result of continued stabilizations. In general, with slightly more careful analysis, one should be able to pin down the exact behavior of a given complex during the stabilization.

\bibliographystyle{amsalpha}
\bibliography{bibliography}

\end{document}